\documentclass[A4]{amsart}
\usepackage[square,sort,comma,numbers]{natbib}
\usepackage{amsmath, amssymb, amstext, amsfonts, textcomp, amsxtra, amsbsy, amsgen, amsopn, amscd, mathrsfs, amsthm, latexsym, array}
\usepackage[textwidth=16cm,textheight=22cm,centering]{geometry}

\usepackage{color}
\usepackage{graphicx}

\usepackage[all]{xy}

\usepackage{bbm, dsfont}
\usepackage{hyperref}

\newtheorem{theorem}             {Theorem}  [section]
\newtheorem{definition} [theorem] {Definition}
\newtheorem{lemma}      [theorem]{Lemma}
\newtheorem{corollary}  [theorem]{Corollary}
\newtheorem{proposition}[theorem]{Proposition}
\newtheorem{remark} [theorem] {Remark}

\numberwithin{equation}{section} \everymath{\displaystyle}

\newcommand{\Cont}{{\rm C}}

\newcommand{\sgn}{{\rm sgn}}

\newcommand{\Ht}{{\rm Ht}}

\newcommand{\Nr}{{\rm Nr}}
\newcommand{\Tr}{{\rm Tr}}

\newcommand{\Char}{\mathbbm{1}}

\newcommand{\gp}[1]{\mathbf{#1}}
\newcommand{\Isom}{{\rm Isom}}
\newcommand{\GL}{{\rm GL}}
\newcommand{\PGL}{{\rm PGL}}
\newcommand{\SL}{{\rm SL}}
\newcommand{\PSL}{{\rm PSL}}
\newcommand{\SO}{{\rm SO}}

\newcommand{\Cl}{\mathrm{Cl}}

\newcommand{\ag}[1]{\mathbb{#1}}
\newcommand{\N}{\mathbb{N}}
\newcommand{\Z}{\mathbb{Z}}

\newcommand{\Mat}{{\rm M}}

\newcommand{\IntP}[1]{\left\lfloor #1 \right\rfloor}
\newcommand{\IntCP}[1]{\left\lceil #1 \right\rceil}

\newcommand{\Q}{\mathbb{Q}}
\newcommand{\R}{\mathbb{R}}
\newcommand{\C}{\mathbb{C}}
\newcommand{\E}{\mathbf{E}}
\newcommand{\F}{\mathbf{F}}

\newcommand{\A}{\mathbb{A}}
\newcommand{\I}{\mathbb{I}}

\newcommand{\vO}{\mathcal{O}}
\newcommand{\vo}{\mathfrak{o}}
\newcommand{\vP}{\mathcal{P}}
\newcommand{\vp}{\mathfrak{p}}
\newcommand{\idlN}{\mathfrak{N}}

\newcommand{\Dis}{{\rm D}}

\newcommand{\norm}[1][\cdot]{\lvert #1 \rvert}
\newcommand{\extnorm}[1]{\left\lvert #1 \right\rvert}

\newcommand{\Pairing}[2]{\langle #1, #2 \rangle}
\newcommand{\extPairing}[2]{\left\langle #1, #2 \right\rangle}

\newcommand{\Res}{{\rm Res}}
\newcommand{\Ind}{{\rm Ind}}

\newcommand{\Intw}{\mathcal{M}}

\newcommand{\fin}{{\rm fin}}
\newcommand{\eis}{{\rm E}}
\newcommand{\Reis}{\mathcal{E}}

\newcommand{\Vol}{{\rm Vol}}
\makeatletter

\newcommand{\Rmnum}[1]{\expandafter\@slowromancap\romannumeral #1@}
\makeatother

\newcommand{\fa}{\mathbf{a}}
\newcommand{\fb}{\mathbf{b}}

\newcommand{\LatL}{\mathcal{L}}
\newcommand{\RSO}{\mathrm{RS-}\mathcal{O}}

\title{On Kuznetsov--Bykovskii's formula of counting prime geodesics}
\author{Giacomo Cherubini \qquad Han Wu \qquad Gergely Z\'abr\'adi}

\begin{document}

\begin{abstract}
	We generalize a formula on the counting of prime geodesics, due to Kuznetsov--Bykovskii, used in the work of Soundararajan--Young on the prime geodesic theorem. The method works over any number field and for any congruence subgroup. We give explicit computation in the cases of principal and Hecke subgroups.
\end{abstract}

	\maketitle
	
	\tableofcontents

\section{Introduction}

	\subsection{Prime Geodesic Theorem}

	The Selberg trace formula, in its classical invariant form, relates the mysterious (discrete) spectrum of the hyperbolic laplacian $\Delta$ on $\Gamma \backslash \ag{H}$ to the so-called ``length spectrum'' of $\Gamma \backslash \ag{H}$, which is closely related to the concrete arithmetics of the lattice $\Gamma$. Precisely, let $\lambda_j = 1/4+r_j^2$ run over the discrete spectrum of $\Delta$, let $g \in \Cont_c^{\infty}(\R)$ and $h(z) = \int_{\R} g(r) e^{-irz} dr$, then the formula takes the form (see \cite[(2.4)]{Mue07} and assume $\Gamma$ is torsion-free)
\begin{equation}
	\sum_{j=0}^{\infty} h(r_j) = \frac{\mathrm{Area}(\Gamma \backslash \ag{H})}{4\pi} \int_{\R} h(r) r \tanh(\pi r) dr + \sum_{[\gamma]_{\Gamma} \text{ hyperbolic}} \frac{l(\gamma_0)}{2 \sinh \left( \frac{l(\gamma)}{2} \right)} g(l(\gamma)) + \mathrm{(CSC)},
\label{STFCI}
\end{equation}
	where $l(\gamma)$ is the length of the closed geodesic in $\Gamma \backslash \ag{H}$ associated with the hyperbolic $\Gamma$-conjugacy class $[\gamma]_{\Gamma}$, and $\gamma_0$ represents the(a) \emph{primitive class} for $[\gamma]_{\Gamma}$ such that $[\gamma]_{\Gamma} = [\gamma_0^k]_{\Gamma}$ for some integer $k \geq 1$. We have omitted the contribution from the continuous spectrum by simply writing it as $\mathrm{(CSC)}$. In particular, this formula allows Vign\'eras \cite{Vi80_Ann} to construct isospectral but non-isometric compact hyperbolic manifolds, giving a beautiful counter-example to Kac's question \cite{Ka66}.
	
	The Selberg trace formula also has a lot of analogies with the explicit formula (of Weil, in this general form) of the Riemann zeta function. Precisely, let $\rho = \beta + i \gamma$ with $\beta \in [0,1], \gamma \in \R$ run over the set of non trivial zeros of the Riemann zeta function. Let $f$ be a nice function on $\R_{>0}$ defined by $F(x) = f(e^{-x})$. Define $M_{1/2}f(s) = \int_{\R} F(x) e^{(1/2-s)x} dx$. Then the explicit formula takes the form
\begin{align*}
	&\quad \lim_{T \to \infty} \sum_{\norm[\gamma_j] < T} M_{1/2}f(\rho) - \left( M_{1/2}f(0) + M_{1/2}f(1) \right) \\
	&= f(1) \log 2 + \sum_p \sum_{n \geq 1} \frac{-\log p}{p^{n/2}} \left( F(- n \log p) + F(n \log p) \right) + \textrm{(Weil's Functional)}.
\end{align*}
	(\ref{STFCI}) is analogous to the above formula if one regards $N([\gamma]_{\Gamma}) := \exp(l(\gamma))$ as the prime powers $p^n$ (and if $F$ is even). This analogy motivated Selberg to define the Selberg zeta function
	$$ Z_{\Gamma}(s) = \sideset{}{_{[\gamma_0]_{\Gamma}}} \prod \sideset{}{_{k=0}^{\infty}} \prod (1- N([\gamma_0]_{\Gamma})^{-s-k}), $$
and other mathematicians, especially analytic number theorists, to introduce counting functions analogous to the ones in the prime number theorem
	$$ \pi_{\Gamma}(x) := \extnorm{ \left\{ [\gamma_0]_{\Gamma} \ \middle| \ N([\gamma_0]_{\Gamma}) \leq x \right\} }, \quad \Psi_{\Gamma}(x) := \sideset{}{_{N([\gamma]_{\Gamma}) \leq x} } \sum \Lambda([\gamma]_{\Gamma}), $$
	where $\Lambda(\cdot)$ is the analogue of the von Mangoldt function defined by $\Lambda([\gamma]_{\Gamma}) = \log N([\gamma_0]_{\Gamma})$ if $[\gamma]_{\Gamma}$ is a power of the primitive hyperbolic class $[\gamma_0]_{\Gamma}$. One then expects the \emph{prime geodesic theorem}, analogous to the prime number theorem, of the form
\begin{equation}
	\pi_{\Gamma}(x) = \mathrm{li}(x) + O_{\epsilon}(x^{\eta+\epsilon}), \quad \mathrm{li}(x) = \int_2^x \frac{dt}{\log t}
\label{PGT}
\end{equation}
where $\eta<1$ is an absolute constant.

	\subsection{Kuznetsov--Bykovskii's formula}

	Finding the smallest possible value of the exponent $\eta$ in (\ref{PGT}) is a deep problem. For $\Gamma = \SL_2(\Z)$ Iwaniec \cite[Theorem 2]{Iw84} obtained $\eta = 35/48$ using the spectral theory of automorphic forms; this was improved subsequently by Luo--Sarnak \cite[Theorem 1.4]{LS95} to $\eta = 7/10$ and to $\eta = 71/102$ by Cai \cite{Cai02}. Recently, Soundararajan--Young \cite[Theorem 1.1]{SY13} succeeded to improve this exponent to $\eta = 25/36$. Their method has the following new ingredient, i.e., a formula due to Kuznetsov, quoted by Bykovskii \cite[(2.2)]{By94} or \cite[Proposition 2.2]{SY13}, used to establish an estimation of $\Psi_{\Gamma}(x+u) - \Psi_{\Gamma}(x)$ for small $u$ compared to $x$ \cite[Theorem 3.2]{SY13}: Setting $X = \sqrt{x} + 1/\sqrt{x}$, one has for $\Gamma = \PSL_2(\Z)$
\begin{equation}
	\Psi_{\Gamma}(x) = 2 \sideset{}{_{n \leq X}} \sum \sqrt{n^2-4} \cdot L(1,n^2-4),
\label{ByF}
\end{equation}
	where $L(s,\delta)$ is a certain $L$-series closely related to quadratic Dirichlet $L$-functions. Its origin goes back to the work of Zagier \cite{Za76} and even Siegel \cite{Sie56}. Among other things, we only recall that $L(s,\delta)$ has the following representation as Dirichlet series \cite[(3)]{SY13}
	$$ L(s,\delta) = \frac{\zeta(2s)}{\zeta(s)} \sum_{q=1}^{\infty} \rho_q(\delta) q^{-s}, \quad \rho_q(\delta) = \extnorm{ \left\{ x \pmod{2q} : x^2 \equiv \delta \pmod{4q} \right\} }; $$
	that it has a meromorphic continuation and satisfies the following functional equation \cite[Lemma 2.1]{SY13}
	$$ \Lambda(s,\delta) := \left( \frac{\norm[\delta]}{\pi} \right)^{s/2} \Gamma \left( \frac{s+a}{2} \right) L(s,\delta) = \Lambda(1-s,\delta). $$
	
\begin{remark}
	Recall \cite[(4) \& Lemma 2.1]{SY13} that if we write $\delta = D l^2$ for a fundamental discriminant $D$ and an integer $l > 0$, then
	$$ L(s,\delta) = l^{\frac{1}{2}-s} T_l^{(D)}(s) L(s,\chi_D), $$
where $\chi_D$ is the non-trivial quadratic Dirichlet character of modulus $D$ and $T_l^{(D)}(s)$ is a product of polynomials in $p^s, p^{-s}$ for $p \mid l$. In particular, the above functional equation is equivalent to
	$$ T_l^{(D)}(s) = T_l^{(D)}(1-s). $$
\end{remark}

\begin{remark}
	There is another property of $L(s,\delta)$, which is interesting but does not enter into the proof of the Prime Geodesic Theorem (PGT). Namely, the zeros of $T_l^{(D)}(s)$ all lie on the critical line $\Re(s) = 1/2$. Following Nelson, Pitale and Saha \cite[\S 1.7]{NPS13}, who encountered a similar phenomenon in the explicit computation of certain local Rankin--Selberg integrals, we shall refer to this property as the ``(geometric) local Riemann hypothesis''.
\label{LocRH}
\end{remark}

\begin{remark}
	There is a parallel story for the upper half space $\ag{H}_3$ and $\Gamma = \PSL_2(\Z[i])$, starting from Sarnak \cite{Sar83}, Koyama \cite{Koy01} and continued with the recent work of Balkanova--Chatzakos--Cherubini--Frolenkov--Laaksonen \cite{BCCFL18} and Balkanova--Frolenkov \cite{BF18}. The approach of Balkanova--Frolenkov \cite{BF18} is the closest to that of Soundararajan--Young. In fact, although they did not mention the (generalization of) Kuznetsov-Bykovskii's formula in their paper, they did relate the relevant mean value of the symmetric $L$-functions to the analogue of $L(s,\delta)$, see \cite[(3.18) \& Lemma 4.6]{BF18}. The equivalence of these two methods fit into the framework of the \emph{Rankin--Selberg trace formula}, initiated and explained in \cite{Wu9}.
\end{remark}

	\subsection{Main Results}
	
	In this paper, we simplify and generalize the proof of (\ref{ByF}). We shall work over a number field $\F$ which is
\begin{itemize}
	\item either $\Q$, in which case we denote by $\ag{H}_{\F} = \ag{H}$ the usual upper half plane;
	\item or an imaginary quadratic field, in which case we denote by $\ag{H}_{\F} = \ag{H}_3$ the upper half space.
\end{itemize}
In either case we denote by  $\vo$ the ring of integers in $F$ and by $\Cl(\F)$ the ideal class group.
	For lattices, we will take $\Gamma < \PSL_2(\vo)$ to be
\begin{itemize}
	\item either a principal congruence subgroup of level $\idlN$, i.e., those $\gamma < \PSL_2(\vo)$ which has a representative in $\SL_2(\vo)$ congruent to the identity matrix modulo the integral ideal $\idlN \subseteq \vo$;
	\item or a Hecke subgroup of level $\idlN$, i.e., those $\gamma < \PSL_2(\vo)$ which has a representative in $\SL_2(\vo)$ congruent to an upper triangular matrix modulo $\idlN$.
\end{itemize}
	Rigorously speaking, there are issues from geometry:
\begin{itemize}
	\item[(1)] If $\Gamma$ is not torsion-free, the quotient space $\Gamma \backslash \ag{H}_{\F}$ is not a Riemannian manifold but an orbifold. Therefore a priori the notion of ``closed geodesics'' is not defined.
	\item[(2)] Even with an adequate definition of closed geodesics in the case of $\ag{H}_{\F} = \ag{H}_3$, the bijective correspondence between the set of closed geodesics and the conjugacy classes in $\Gamma$ fails in general. However, this failure is only ``up to a finite number''. In particular, primitive hyperbolic conjugacy classes cannot be defined as $[ \gamma_0 ]$ for which there exists no other class $[\gamma]$ such that $[\gamma_0] = [\gamma^k]$ for some integer $\norm[k] \geq 2$. This is clear from the shape of the (invariant) Selberg trace formula (see \cite[Theorem 2.2]{BCCFL18}), but still needs a geometric clarification.
\end{itemize}
	Both issues will be carefully looked at in Section 2. In particular, we will replace the notion of a conjugacy class in $\Gamma$ by ``root conjugacy class'' but keep the same notation, which remedy the bijective correspondence. We will also introduce $\Reis(\gamma) \in 2\Z_{>0}$ in Remark \ref{Length} for such classes, so that our counting function becomes (drop the subscript $\Gamma$ for simplicity, or regard it as root-conjugacy class)
	$$ \Psi_{\Gamma}(x) = \sideset{}{_{N([\gamma]) \leq x}} \sum \frac{\Lambda([\gamma])}{\Reis(\gamma)}. $$

\begin{theorem}
(1) There exists $L_{\Gamma}(s,\delta)$ such that
	$$ \Psi_{\Gamma}(x) = \sideset{}{_n} \sum \extnorm{d_{n^2-4}}_{\infty}^{1/2} \cdot L_{\Gamma}(1, n^2-4), $$
where $n$ runs over the set of $\Tr(\gamma)$ for hyperbolic $\gamma \in \Gamma$ (considered as a subset of $\vo$) such that
	$$ \max \left\{ \extnorm{\frac{n+\sqrt{n^2-4}}{2}}_{\infty}, \extnorm{\frac{n-\sqrt{n^2-4}}{2}}_{\infty} \right\} \leq x; $$
and $d_{n^2-4} \in \vo$ is the fundamental discriminant of the quadratic extension $\F[\sqrt{n^2-4}] / \F$.

\noindent (2) The $L$-series $L_{\Gamma}(s,\delta)$ has the following factorization
	$$ L(s,\delta) = \varepsilon_{\infty} \cdot \norm[\Cl(\F)] \cdot [\PSL_2(\vo) : \Gamma] \cdot P_{\Gamma}(s) \cdot L(s,\chi_{d_{\delta}}), $$
where $\chi_{d_{\delta}}$ is the quadratic character associated with $\F[\sqrt{\delta}]/\F$, $\varepsilon_{\infty} = 1$ or $(2\pi)^{-1}$ according as $\F_{\infty} = \R$ or $\C$, and $P_{\Gamma}(s)$ is a product of polynomials $P_{\Gamma,\vp}(s)$ in $\Nr(\vp)^s, \Nr(\vp)^{-s}$ for $\vp \mid (\delta/d_{\delta})$ satisfying
	$$ P_{\Gamma,\vp}(s) = P_{\Gamma,\vp}(1-s). $$

\noindent (3) Moreover, $P_{\Gamma,\vp}(s)$ satisfies the local Riemann hypothesis if $\Gamma$ is a principal congruence subgroup.
\label{MainPS}
\end{theorem}
\begin{remark}
	The precise form of the polynomials $P_{\Gamma,\vp}(s)$ will be given in (\ref{PSLocPoly}) for the case of principal congruence subgroups, in (\ref{PSLocPolyUnr1}), (\ref{PSLocPolyUnr2}), (\ref{PSLocPolyR1}), (\ref{PSLocPolyR2}), (\ref{PSLocPolyS1}) and (\ref{PSLocPolyS2}) for the case of Hecke subgroups. In particular, $P_{\Gamma,\vp}(s)$ has the same shape for the principal congruence subgroups, with a ``shift'' determined by the level of $\Gamma$.
\end{remark}
\begin{remark}
	We have restricted to principal and Hecke subgroups for explicit computation, but our method is applicable to any congruence subgroup. However, the analogue of $\Psi_{\Gamma}$ for number fields is unclear to us. This paper should also be viewed as a refinement of the explicit computation of the elliptic terms in the geometric side of the Rankin--Selberg trace formula in our previous paper \cite{Wu9}. Hence although for $\eta \neq 1$, the terms $I_{\eta}(s;t,f)$, which will be defined in (\ref{GL2Form}), do not contribute to the final formula, we have included their computation in detail.  
\end{remark}
\begin{remark}
	Section \S 3.3 contains the beginning of a non-adelic treatment, which leads directly to the Dirichlet series representation of $L(s,\delta)$ by Rankin--Selberg unfolding.
\end{remark}
\begin{remark}
	Although we have written this paper as a preparation for the prime geodesic theorems, the method of computation is potentially useful for the beyond endoscopy proposal. For example, our method should give a simpler proof of the Eichler-Selberg formula as treated in Rudnick's thesis. All these will come in a later paper.
\end{remark}

As an application, we deduce a prime geodesic theorem for
principal congruence subgroups of $\SL_2(\Z)$ and $\SL_2(\Z[i])$.
Like in \cite{SY13} and \cite{BBCL20},
our results are expressed in terms of the subconvexity exponent~$\theta$
for quadratic Dirichlet $L$-functions over $\mathbf{F}$,
that is, a number such that
\begin{equation}\label{def:theta}
L(1/2+it,\chi_D) \ll (1+|t|)^A |D|_{\infty}^{\theta+\epsilon},
\end{equation}
where $\chi_D$ is any quadratic character over $\mathbf{F}$
and $A$ is some fixed constant.
The convexity bound is $\theta=1/4$ and
if $\mathbf{F}=\Q$ one can take $\theta=1/6$
by the work of Conrey and Iwaniec \cite{CI00}.
More recently, Nelson \cite{Ne20}
has announced the same exponent over number fields.

\begin{theorem}
Let $\Gamma$ be a principal congruence subgroup of $\mathrm{SL}_2(\Z[i])$
and let $\epsilon>0$. Then
\begin{equation}\label{eq:mainpgt1}
\Psi_\Gamma(X) = \frac{1}{2}X^2
+ O(X^{\frac{3}{2}+\frac{2\theta}{3}+\epsilon}),
\end{equation}
where $\theta$ is as in \eqref{def:theta}.
\label{MainPGT}
\end{theorem}

Notice that \eqref{eq:mainpgt1} corresponds to the bound obtained in \cite[Remark 4]{BBCL20}.
By inserting the convexity exponent $\theta=1/4$, one recovers the classical exponent $5/3$
due to Sarnak \cite[Theorem 5.1]{Sar83}.
Therefore, any subconvexity exponent gives a non-trivial result.
One could reduce further the error in \eqref{eq:mainpgt1} if one had
better bounds for a spectral exponential sum featuring in the proof (see Section~\ref{S5}),
which are available for $\mathrm{SL}_2(\Z[i])$, see \cite{BCCFL18,BF20},
but not for general congruence groups.

We give a proof of Theorem \ref{MainPGT} in Section~\ref{S5.2}
and at the end of that section we sketch the case when $\Gamma$
is a congruence subgroup of $\mathrm{SL}_2(\Z)$,
which leads to the following generalization of \cite[Theorem 1.1]{SY13}.
\begin{theorem}\label{MainPGT2}
Let $\Gamma$ be a principal congruence subgroup of $\mathrm{SL}_2(\Z)$
and let $\epsilon>0$. Then we have
\begin{equation}\label{eq:mainpgt2}
\Psi_\Gamma(X) = X
+ O(X^{\frac{2}{3}+\frac{\theta}{6}+\epsilon}),
\end{equation}
with $\theta$ is as in \eqref{def:theta}.
\end{theorem}

The proof of Theorem \ref{MainPGT2} implicitly uses an
estimate for Rankin--Selberg $L$-functions attached to Maass forms
for $\Gamma$, which is essentially the same ingredient needed to improve
Theorem \ref{MainPGT}. Over $\Q$, such a bound was proved by Luo and Sarnak
\cite{LS95} in detail for the modular group and they mention in \cite[p.211]{LS95}
that the proof extends to congruence subgroups.
By analogy, we are inclined to believe that the exponent $3/2+2\theta/3+\epsilon$
in Theorem \ref{MainPGT} may be lowered to
\begin{equation}
\frac{3}{2}+\frac{32\theta^2+28\theta-1}{46+40\theta}+\epsilon,
\end{equation}
matching the currently best known result for $\mathrm{SL}_2(\Z[i])$,
see \cite{BF20}, where a Luo--Sarnak-type bound is proved with an additional
dependence on $\theta$.

As a more technical remark, we point out that in Theorem \ref{MainPS}
the function $\Psi_\Gamma$ counts geodesics without orientation,
which differs by a factor of two compared to the more common
definition used in e.g.~\cite{BCCFL18,BBCL20,BF20,Iw84,Koy01,LS95,Sar83,SY13}
(see also Remark \ref{1104:rmk}).
For consistency with the rest of the literature,
and by a slight abuse of notation,
Theorems \ref{MainPGT} and \ref{MainPGT2}
are stated for the non-oriented counting function.

	\subsection{Acknowledgement}
	
	H.Wu would like to thank the R\'enyi institute, EPFL, the IMS at NUS and QMUL for providing stimulating working conditions during the preparation of this paper, and the support of the Leverhulme Trust Research Project Grant RPG-2018-401. G.Z\'abr\'adi was supported by the MTA R\'enyi Int\'ezet Lend\"ulet Automorphic Research Group, by the J\'anos Bolyai Research Scholarship of the Hungarian Academy of Sciences, and by the NKFIH Research Grant FK-127906, and by Project ED 18-1-2019-0030 (Application-specific highly reliable IT solutions) under the Thematic Excellence Programme funding scheme.

\section{Geometric Preliminaries}

	\subsection{Closed Geodesics}

	Throughout this paper, $\F$ is either $\ag{Q}$ or a quadratic imaginary number field. The two cases are distinguished by $\F_{\infty} = \ag{R}$ or $\F_{\infty} = \ag{C}$. We write
	$$ \ag{H}_{\F} = \left\{ \begin{matrix} \ag{H}_2 & \text{if } \F_{\infty} = \ag{R} \\ \ag{H}_3 & \text{if } \F_{\infty} = \ag{C} \end{matrix}  \right. , \quad \ag{H}_{\F} \ni i_{\F} = \left\{ \begin{matrix} i & \text{if } \F_{\infty} = \ag{R} \\ j & \text{if } \F_{\infty} = \ag{C} \end{matrix} \right. , $$
	where $\ag{H}_2$ resp. $\ag{H}_3$ is the half upper plane resp. half upper space and $j=(0,0,1)$.

\noindent We work in the category of Riemannian manifolds with orientation. Recall that a \emph{geodesic flow with unit speed} or simply \emph{geodesic flow} is an orientation-preserving isometric embedding $\ell: \ag{R} \to \ag{H}_{\F}$ which satisfies the (second order) differential equation of geodesics. Here $\ag{R}$ is regarded as a one-dimensional Riemannian manifold with orientation, whose group of orientation-preseving isometries is
	$$ \Isom_+(\ag{R}) \simeq \ag{R}. $$
We define a \emph{geodesic curve with orientation} or simply \emph{geodesic} to be a class of geodesic flows
	$$ [\ell] := \left\{ \ell \circ \sigma \ \middle| \ \sigma \in \Isom_+(\ag{R}) \right\}. $$	
An element $g \in \Isom_+(\ag{H}_{\F}) \simeq \PSL_2(\F_{\infty})$ \emph{stabilizes} $[\ell]$ if for some $\sigma \in \Isom_+(\ag{R})$ and any $t\in \ag{R}$
	$$ g.\ell(t) = \ell(\sigma(t)). $$
Changing $\ell$ to another $\ell \circ \sigma'$ in the above equation will change $\sigma$ to the conjugate $\sigma' \sigma (\sigma')^{-1}$. Since $\Isom_+(\ag{R})$ is abelian, $\sigma$ does not change. The group of stabilizers of $[\ell]$ is denoted by $N_{[\ell]}$. Hence we get a well-defined homomorphism of (Lie) groups
	$$ \pi_{[\ell]}: N_{[\ell]} \to \Isom_+(\ag{R}), \quad g \mapsto \sigma. $$
There is a special geodesic flow
	$$ \varphi(t) = e^t i_{\F}. $$
	
\begin{lemma}
	The group of stabilizers $N_{[\varphi]}$ equals $\gp{A}_1(\F_{\infty})$, where
	$$ \gp{A}_1(\F_{\infty}) = \left\{ a(\lambda)=\begin{pmatrix} \lambda & \\ & \lambda^{-1} \end{pmatrix} \ \middle| \ \lambda \in \F_{\infty}^{\times} \right\} / \{ \pm 1 \}. $$
\label{StgpGeod}
\end{lemma}
\begin{proof}
	Let $g \in N_{[\varphi]}$. This is equivalent to the existence of $t_0 \in \ag{R}$ such that
	$$ g.\varphi(t) = \varphi(t+t_0) \quad \Longleftrightarrow \quad a(e^{-t_0/2})g.\varphi(t) = \varphi(t). $$
	The last equation is equivalent to that $a(e^{-t_0/2})g$ fixes the point $(i_{\F}, i_{\F})$ in the unit tangent bundle of $\ag{H}_{\F}$. We conclude by showing that the stabilizer of this point is
	$$ \gp{T}_1(\ag{R}) = \{ 1 \} \quad \text{or} \quad \gp{T}_1(\ag{C}) = \left\{ \begin{pmatrix} e^{i\theta} & \\ & e^{-i\theta} \end{pmatrix} \ \middle| \ \theta \in \ag{R} \right\} / \{ \pm 1 \}. $$
	This is elementary and left to the reader.
\end{proof}

\begin{corollary}
	There is a bijection between the set of geodesics in $\ag{H}_{\F}$ and $\PSL_2(\F_{\infty}) / \gp{A}_1(\F_{\infty})$.
\end{corollary}
\begin{proof}
	By the transitivity of the action of $\PSL_2(\F_{\infty})$ on the unit tangent bundle of $\ag{H}_{\F}$, any geodesic flow in $\ag{H}_{\F}$ is of the form $\ell(t) = g.\varphi(t)$ for some $g \in \PSL_2(\F_{\infty})$. The association
$\ell\mapsto g$ descends to $g\in \PSL_2(\F_{\infty}/\A_1(\F_{\infty}))$ and establishes a bijective correspondence by the above discussion.
\end{proof}

\noindent Let $\Gamma < \PSL_2(\F_{\infty})$ be a lattice. Some geodesics $[\ell]$ are globally stable by non-trivial elements in $\Gamma$. We formalize this property in the following definition.

\begin{definition}
	A geodesic $[\ell]$ is called $\Gamma$-periodic, if $\pi_{[\ell]}(N_{[\ell]} \cap \Gamma)$ is a lattice in $\Isom_+(\ag{R}) \simeq \ag{R}$. In this case, $N_{[\ell]} \cap \Gamma$ is called the \emph{group of automorphs} of $[\ell]$, denoted by $A_{[\ell]}$ if the lattice $\Gamma$ is clear from the context.
\end{definition}

\begin{proposition}
	For each $\Gamma$-periodic geodesic $[\ell]$, $A_{[\ell]}$ is a maximal abelian subgroup of $\Gamma$. It is isomorphic to $\ag{Z} \times \ag{Z}/n\ag{Z}$ for some $n \in \ag{Z}_{> 0}$, and if $\F_{\infty} = \ag{R}$ we must have $n=1$.
\end{proposition}
\begin{proof}
	If $\ell(t) = g.\varphi(t)$, then $N_{[\ell]} = g N_{[\varphi]} g^{-1}$, which by Lemma \ref{StgpGeod} is isomorphic to
	$$ \F_{\infty}^{\times} / \{ \pm 1 \} \simeq \left\{ \begin{matrix} \ag{R} & \text{if } \F_{\infty} = \ag{R} \\ \ag{R} \times (\ag{R}/\ag{Z}) & \text{if } \F_{\infty} = \ag{C} \end{matrix} \right. . $$
	With this identification, $\pi_{[\ell]}$ is identified with the projection onto the $\ag{R}$-component. Consequently, $A_{[\ell]}$ is identified with a discrete subgroup with non-trivial $\ag{R}$-component, hence is of the asserted form. It remains to show that $A_{[\ell]}$ is maximal abelian. If $\gamma \in \Gamma$ commutes with $A_{[\ell]}$, then it commutes with $N_{[\ell]}$, hence lies in $N_{[\ell]}$ since the later is a maximal abelian Lie subgroup of $\PSL_2(\F_{\infty})$. Thus $\gamma \in N_{[\ell]} \cap \Gamma = A_{[\ell]}$ by definition.
\end{proof}

\noindent In general, $\Gamma \backslash \ag{H}_{\F}$ is not necessarily a Riemannian manifold, but an orbifold. We do not know an intrinsic way to define closed geodesics on an orbifold. In our special case, we make use of geodesics on $\ag{H}_{\F}$.

\begin{definition}
	A closed geodesic on $\Gamma \backslash \ag{H}_{\F}$ is the image under the natural projection $\ag{H}_{\F} \to \Gamma \backslash \ag{H}_{\F}$ of a $\Gamma$-periodic geodesic.
\end{definition}

\begin{remark}
	In the special case that $\Gamma \backslash \ag{H}_{\F}$ does admit the structure of a Riemannian manifold, our notion of ``closed geodesics'' is that of ``compact geodesics'' or ``closed geodesics with finite length'', in the sense that they are (classes of) isometric embeddings of the form
	$$ \ell: \ag{R} / T \ag{R} \to \Gamma \backslash \ag{H}_{\F}, $$
	where both $\ag{R} / T \ag{R}$ and $\Gamma \backslash \ag{H}_{\F}$ are regarded as Riemannian manifolds with orientation and $T > 0$ is the length of $[\ell]$.
\end{remark}

	\subsection{Relation with Conjugacy Classes}

	We recall the standard classification of elements in $\PSL_2(\F_{\infty})$ (or $\SL_2(\F_{\infty})$). For $\F_{\infty} = \ag{R}$, this is standard. $\gamma \in \SL_2(\ag{R})$ is \emph{elliptic} resp. \emph{hyperbolic} resp. \emph{parabolic} if $\norm[\Tr \gamma] < 2$ resp. $\norm[\Tr \gamma] > 2$ resp. $\norm[\Tr \gamma] = 2$. For $\F_{\infty} = \ag{C}$, we follow \cite[Definition 2.1.3]{EGM98}. Namely, $\gamma \in \SL_2(\ag{C})$ is \emph{elliptic} resp. \emph{hyperbolic} resp. \emph{parabolic} if $\Tr \gamma \in \ag{R}$ and $\norm[\Tr \gamma] < 2$ resp. $\norm[\Tr \gamma] > 2$ resp. $\norm[\Tr \gamma] = 2$; it is \emph{loxodromic} if $\Tr \gamma \notin \ag{R}$.
	
\begin{remark}
	For the purpose of this paper, it is \emph{not} important to distinguish \emph{hyperbolic} elements from \emph{loxodromic} elements. We will simply call them \emph{hyperbolic}.
\end{remark}
	
\noindent We also recall the \emph{arithmetic} classification of elements $\gamma \in \GL_2(\F) - \gp{Z}(\F)$: $\gamma$ is called $\F$-elliptic resp. $\F$-hyperbolic resp. $\F$-parabolic if the $\F$-algebra $\F[\gamma]$ is a quadratic field extension of $\F$ resp. isomorphic to $\F \times \F$ resp. isomorphic to $\F[X]/(X^2)$. For congruence subgroups, the hyperbolic elements are automatically $\F$-elliptic as the following lemma shows.

\begin{lemma}
	If $\gamma \in \SL_2(\vo)$ is hyperbolic, then it is $\F$-elliptic.
\label{FEllip}
\end{lemma}
\begin{proof}
	The element $\gamma$ cannot be $\F$-parabolic, otherwise either $(\gamma + 1)^2 = 0$ or $(\gamma - 1)^2 = 0$, which implies $\gamma$ is parabolic. If $\gamma$ is $\F$-hyperbolic, then it is conjugate in $\GL_2(\F)$ to a diagonal matrix with entries $t$ and $t^{-1}$ for some $t \in \F$. We thus get $\Tr(\gamma) = t+t^{-1} \in \vo$, hence both $t$ and $t^{-1}$ are integral over $\vo$. It follows that $t \in \vo^{\times}$. Under our assumption on $\F$, $\vo^{\times}$ is a finite group. Hence $\gamma$ is of finite order and must be elliptic. The only remaining possibility is $\F$-elliptic.
\end{proof}
\begin{remark}
	If $\F=\Q$, the condition ``$\gamma \in \SL_2(\vo)$'' in the above lemma can be relaxed to ``$\gamma \in \SL_2(\Q)$'', since ``hyperbolic'' is the same as ``$\R$-elliptic'' in this situation.
\end{remark}

\begin{remark}
	The set of elliptic elements in $\Gamma$ will be denoted by $U(\Gamma)$. It has the following description
	$$ U(\Gamma) = \left\{ \gamma \in \Gamma \ \middle| \ \exists n \in \ag{Z}_{>0}, \gamma^n=1 \right\}. $$
	It is the set of elements in $\Gamma$ which admits (at least) a fixed point in $\ag{H}_{\F}$ (c.f. \cite[Theorem 1.3.1]{Miy06} and \cite[Proposition 2.1.4]{EGM98}).
\end{remark}

\begin{definition}
	For hyperbolic $\gamma \in \Gamma$, we denote by $C_{\gamma}$ the centralizer of $\gamma$ in $\PSL_2(\F_{\infty})$. We define the \emph{root-conjugacy class} of $\gamma$ by
	$$ \{ \gamma \} := \sideset{}{_{u \in C_{\gamma} \cap U(\Gamma)}} \bigcup [\gamma u], $$
	where for any $\gamma' \in \Gamma$, $[\gamma']$ is the usual conjugacy class in $\Gamma$.
\end{definition}

\begin{lemma}
	For hyperbolic $\gamma \in \Gamma$, $C_{\gamma}$ is a maximal $\F_{\infty}$-split torus in $\PSL_2(\F_{\infty})$, hence isomorphic to $\F_{\infty}^{\times}$. Under this isomorphism, $C_{\gamma} \cap \Gamma$ is a lattice in $\F_{\infty}^{\times}$.
\end{lemma}
\begin{proof}
	Only the last assertion needs some explanation. In fact, under the isomorphism mentioned in the statement, $\gamma \in C_{\gamma}$ is mapped to some $r \in \F_{\infty}^{\times}$ with $\norm[r] \neq 1$. Hence $\gamma^{\ag{Z}}$ is identified with some lattice in $\F_{\infty}^{\times}$. \emph{A fortiori}, the discrete subgroup $C_{\gamma} \cap \Gamma > \gamma^{\ag{Z}}$ is a lattice.
\end{proof}

\begin{remark}
	In concrete terms, the isomorphism takes $\gamma$ to $r$, one of its eigenvalue in $\F_{\infty}^{\times}$. We call
	$$ l(\gamma) := 2 \log \max(\norm[r], \norm[r]^{-1}) \in \ag{R}_{>0} $$
	the \emph{length} of $\gamma$. This quantity is unchanged by conjugation, hence passes to conjugacy and root-conjugacy classes, i.e., $l([\gamma])$ and $l(\{ \gamma \})$ are well-defined. It is equal to the length of the geodesic on $\Gamma \backslash \ag{H}_{\F}$ to which it corresponds. We also denote by $\Reis(\gamma) \in \Z_{> 0}$ such that the torsion part of $C_{\gamma} \cap \Gamma$ is isomorphic to $\Z/\Reis(\gamma)\Z$. Since we work with $\PSL_2$, we have
\begin{itemize}
	\item[(1)] $\Reis(\gamma)=2$ if $\F_{\infty} = \R$;
	\item[(2)] $\Reis(\gamma)$ is an even positive integer if $\F_{\infty} = \C$.
\end{itemize}
	The fact that $\Reis(\gamma)$ is always even reflects the geometric view that a geodesic flow has a direction.
\label{Length}
\end{remark}

\begin{definition}
	A hyperbolic $\gamma \in \Gamma$ is called \emph{primitive}, if its length $l(\gamma)$ attains the minimum among elements in $C_{\gamma} \cap \Gamma$. A root-conjugacy class $\{ \gamma \}$ of a hyperbolic element $\gamma \in \Gamma$ is called \emph{primtive} if $\gamma$ is primitive.
\label{PrimConDef}
\end{definition}

\begin{proposition}
	Let $\Gamma < \PSL_2(\F_{\infty})$ be a discrete subgroup. Closed geodesics on $\Gamma \backslash \ag{H}_{\F}$ are in bijection with the primitive root-conjugacy classes of hyperbolic elements in $\Gamma$.
\label{GeodConjCor}
\end{proposition}
\begin{proof}
	Let $[\ell]$ be a closed geodesic on $\Gamma \backslash \ag{H}_{\F}$. Since $\pi(A_{[\ell]})$ is a lattice in $\ag{R}$, it is $t_0 \ag{Z}$ for some unique $t_0 > 0$. Let $\gamma \in A_{[\ell]}$ be any element such that $\pi_{[\ell]}(\gamma) = t_0 = l(\gamma)$. Other choices of $\gamma$ are of the form $\gamma u$ for some $u \in U(\Gamma)$. Moreover, it is easy to see
	$$ N_{[\ell]} = C_{\gamma}. $$
	Hence if $\Gamma \ni \gamma' \in C_{\gamma}$, $\gamma' \in A_{[\ell]} = N_{[\ell]} \cap \Gamma$. Thus $\pi_{[\ell]}(\gamma') = l(\gamma')$ is an integral multiple of $t_0$. This proves that $\gamma$ is primitive. We thus get a well-defined map from the set of closed geodesics to the set of primitive root-conjugacy classes
	$$ \iota: [\ell] \mapsto \{ \gamma \}. $$
	Conversely, if $\gamma \in \Gamma$ is hyperbolic, then for some $g \in \PSL_2(\F_{\infty})$, $g^{-1} \gamma g \in \gp{A}_1(\F_{\infty})$, hence stabilizes the unique geodesic flow $\varphi(t)$. Changing $g$ to $gw$ if necessary, we may assume that $g^{-1} \gamma g$ fixes the orientation of $\varphi(t)$. If $\gamma' \in C_{\gamma}$, then $g^{-1} \gamma' g \in C_{g^{-1} \gamma g} = \gp{A}_1(\F_{\infty})$. Thus $g^{-1} \gamma' g$ also stabilizes $\varphi(t)$ and fixes its orientation. We thus get a map from the set of root-conjugacy classes in $\Gamma$ to the set of closed geodesics on $\Gamma \backslash \ag{H}_{\F}$
	$$ \tau: \{ \gamma \} \mapsto [g.\varphi]. $$
	It is easy to verify that $ \tau \circ \iota = \mathrm{id} $ is the identity map on the set of closed geodesics. Hence $\iota$ is a bijection onto its image, i.e., the set of primitive root-conjugacy classes.
\end{proof}
\begin{remark}
	It may be clearer if we summarize the above proof in words. Closed geodesics $[\ell]$ correspond bijectively to $\Gamma$-conjugacy classes of maximal split tori $N_{[\ell]}$ (the stabilizer group of $[\ell]$) in $\PSL_2(\F_{\infty})$ for which $A_{[\ell]} = N_{[\ell]} \cap \Gamma$ is a lattice in $N_{[\ell]}$. In the group of automorphs $A_{[\ell]}$, there is $\{ \gamma \}$ inducing the translation $t \mapsto t+t_0$ in the parameters $\ag{R}$ of $\ell$ such that $t_0 > 0$ is smallest possible. These are the primitive hyperbolic elements. In particular, the fibers of $\tau$ in the above proof are precisely hyperbolic conjugacy classes which admit the same centralizer group up to $\Gamma$-conjugacy.
\end{remark}

\begin{proposition}
	$U(\Gamma)$ is a finite union of conjugacy classes in $\Gamma$.
\label{FinEllConj}
\end{proposition}

\noindent We leave the technical detail of the proof in an appendix. For the moment, we are content with the following remark, since we will eventually work with arithmetic non-uniform lattices.
\begin{remark}
	In the case $\F=\ag{Q}$, this is part of \cite[Theorem 1.7.8]{Miy06}. But we do not see how to extend this method to the case $\F_{\infty} = \ag{C}$. However, the argument given in the appendix for the case $\F_{\infty} = \ag{C}$ can be easily adapted to the case $\F=\ag{Q}$, replacing \cite[Theorem 2.7]{EGM98} by \cite[(1.9.9)]{Miy06}. Moreover, if $\Gamma$ is a congruence subgroup, we have a simpler proof. We first reduce to the case $\Gamma = \PSL_2(\vo_{\F})$ for the ring of integers $\vo_{\F}$ of $\F$ by noticing that Proposition \ref{FinEllConj} for $\Gamma_1$ and $\Gamma_2$ are equivalent if $\Gamma_2 < \Gamma_1$ and $[\Gamma_1 : \Gamma_2] < \infty$. Then we take a Siegel domain $\mathcal{S}$ which contains a fundamental domain for $\PSL_2(\vo_{\F})$. Since the number of elements $\gamma \in \PSL_2(\vo_{\F})$ such that $\gamma.\mathcal{S} \cap \mathcal{S} \neq \emptyset$ is finite \cite[Lemma (3.3)]{GJ79}, we conclude.
\end{remark}

\begin{corollary}
	If $\Gamma < \PSL_2(\F_{\infty})$ is a lattice, then the number of primitive root-conjugacy classes $\{ \gamma \}$ which contains more than one conjugacy class is finite.
\end{corollary}
\begin{proof}
	There is nothing to prove in the case $\F_{\infty} = \ag{R}$. In the case $\F_{\infty} = \ag{C}$, if $\{ \gamma \}$ is such a primitive root-conjugacy class, then $C_{\gamma} = C_u$ for some $1 \neq u \in U(\Gamma)$. Thus $\gamma$ is an element in $C_u \cap \Gamma$ with minimal positive length. But the $\Gamma$-conjugacy classes of possible $C_u \cap \Gamma$ are finite by Proposition \ref{FinEllConj}. Thus the possible root-conjugacy classes of $\gamma$ are finite.
\end{proof}

\section{Transforming to Rankin--Selberg Integrals}

	\subsection{Relevant Orbital Integrals}
	\label{ROI}
	
	Take $f_{\infty}: \PSL_2(\F_{\infty}) \to \ag{C}$ to be a nice test function (ie.\ smooth with compact support). Let $\Gamma < \PSL_2(\vo)$ be a congruence subgroup. For every hyperbolic conjugacy class $[\gamma]$ in $\Gamma$, the following sum
	$$ K(x; f_{\infty}, [\gamma]) := \sideset{}{_{\gamma' \in [\gamma]}} \sum f_{\infty}(x^{-1} \gamma' x) $$
	is a well-defined function on $\Gamma \backslash \PSL_2(\F_{\infty})$. Fix a $t \in \F_{\infty}$ lying in the image of hyperbolic elements in $\Gamma$ under the trace map. We are interested in
	$$ I(t,f_{\infty}) := \sideset{}{_{[\gamma]: \Tr(\gamma) = t}} \sum \int_{\Gamma \backslash \PSL_2(\F_{\infty})} K(x; f_{\infty}, [\gamma]) dx. $$
	The computation is part of the orbital integrals in the geometric side of the Selberg trace formula. We include it for convenience. Recall that $C_{\gamma}$ is the centralizer of $\gamma$ in $\PSL_2(\F_{\infty})$. We have
\begin{align*}
	I(t,f_{\infty}) &= \sideset{}{_{[\gamma]: \Tr(\gamma) = t}} \sum \int_{\Gamma \backslash \PSL_2(\F_{\infty})} \sideset{}{_{\sigma \in C_{\gamma} \cap \Gamma \backslash \Gamma}} \sum f_{\infty}(x^{-1} \sigma^{-1} \gamma \sigma x) dx \\
	&= \sideset{}{_{[\gamma]: \Tr(\gamma) = t}} \sum \int_{C_{\gamma} \cap \Gamma \backslash \PSL_2(\F_{\infty})} f_{\infty}(x^{-1} \gamma x) dx \\
	&= \sideset{}{_{[\gamma]: \Tr(\gamma) = t}} \sum \Vol(C_{\gamma} \cap \Gamma \backslash C_{\gamma}) \int_{C_{\gamma} \backslash \PSL_2(\F_{\infty})} f_{\infty}(x^{-1} \gamma x) dx.
\end{align*}
	Now $C_{\gamma} \cap \Gamma$ is a lattice in $C_{\gamma}$, a maximal $\F_{\infty}$-split torus isomorphic to $\F_{\infty}^{\times}$. If $\gamma_0 \in C_{\gamma} \cap \Gamma$ is a primitive element, and if $g \in \PSL_2(\F_{\infty})$ such that $g^{-1} \gamma g \in \gp{A}_1(\F_{\infty})$, we get an identification via conjugation by $g$
	$$ \begin{matrix} C_{\gamma} & \simeq & \F_{\infty}^{\times}/\{ \pm 1\} \simeq \ag{R} \text{ or } \ag{R} \times (\ag{R} / \ag{Z}) \\ \uparrow & & \uparrow \\ C_{\gamma} \cap \Gamma & \simeq & l(\gamma_0) \ag{Z} \text{ or } l(\gamma_0) \ag{Z} \times ( \ag{Z} / \Reis(\gamma) \ag{Z} ) \end{matrix}. $$
	If we transport the Haar measure of $\F_{\infty}^{\times}$ to $C_{\gamma}$, then we obtain (note that the measure $dz$ on $\C$ is twice the Lebesgue measure and the measure on $\C^{\times}$ is $dz/\norm[z]_{\C}$)
	$$ \Vol((C_{\gamma} \cap \Gamma) \backslash C_{\gamma}) = \left\{ \begin{matrix} l(\gamma_0) & \text{if } \F_{\infty} = \ag{R} \\ \frac{l(\gamma_0)}{\Reis(\gamma)} \cdot 2\pi & \text{if } \F_{\infty} = \ag{C} \end{matrix} \right. . $$
	The orbital integral is transformed to
\begin{equation}
    \vO_{\gamma}(f_{\infty}) := \int_{C_{\gamma} \backslash \PSL_2(\F_{\infty})} f_{\infty}(x^{-1} \gamma x) dx = \int_{\gp{A}_1(\F_{\infty}) \backslash \PSL_2(\F_{\infty})} f_{\infty}(x^{-1} g^{-1} \gamma g x) dx.
\label{OrbIntA}
\end{equation} 
\begin{remark}
	From now on, we choose $f_{\infty}$ to be bi-$\gp{K}_{\infty}$-invariant. Since $\gamma \in \PSL_2(\F_{\infty})$, the above integral depends only on the trace $t = \Tr(\gamma) \in \F_{\infty}^{\times}/\{ \pm 1 \}$, and we shall also write $\vO(t,f_{\infty})$ instead of $\vO_{\gamma}(f_{\infty})$. In fact, the above transform is given by the Selberg transform. We recall: for $\F_{\infty} = \ag{R}$
\begin{equation}
	\int_{\ag{R}} f_{\infty}\left( \begin{pmatrix} 1 & -x \\ 0 & 1 \end{pmatrix} \begin{pmatrix} y^{1/2} & 0 \\ 0 & y^{-1/2} \end{pmatrix} \begin{pmatrix} 1 & x \\ 0 & 1 \end{pmatrix} \right) dx = \frac{g(\log y)}{\sqrt{y} - \sqrt{y}^{-1}}, \quad y > 1;
\label{STR}
\end{equation}
	for $\F_{\infty} = \ag{C}$
\begin{equation}
	\int_{\ag{C}} f_{\infty}\left( \begin{pmatrix} 1 & -x \\ 0 & 1 \end{pmatrix} \begin{pmatrix} y^{1/2} & 0 \\ 0 & y^{-1/2} \end{pmatrix} \begin{pmatrix} 1 & x \\ 0 & 1 \end{pmatrix} \right) dx = \frac{g(\log \norm[y])}{4 \norm[ \sqrt{y} - \sqrt{y}^{-1} ]^2}, \quad y \in \ag{C}^{\times}, \norm[y] > 1.
\label{STC}
\end{equation}
    We do not need these explicit formulae in this paper, but only need that for any $t$ there is $f_{\infty}$ such that $\vO(t,f_{\infty}) \neq 0$.
\end{remark}
\noindent If for some $P \in \GL_2(\F_{\infty})$ we have
	$$ P^{-1} \gamma P = \begin{pmatrix} N(\gamma)^{1/2} & 0 \\ 0 & N(\gamma)^{-1/2} \end{pmatrix}, \quad \norm[N(\gamma)] > 1, $$
	so that $l(\gamma) = \log \norm[N(\gamma)]$, we arrive at the formula
\begin{equation}
	I(t,f_{\infty}) = \left\{ \begin{matrix} \sideset{}{_{[\gamma]: \Tr(\gamma) = t}} \sum l(\gamma_0) \cdot \vO_{\gamma}(f_{\infty}) = \vO(t,f_{\infty}) \cdot \sideset{}{_{[\gamma]: \Tr(\gamma) = t}} \sum l(\gamma_0) & \text{if } \F_{\infty} = \ag{R} \\ 2 \pi \vO(t,f_{\infty}) \cdot \sideset{}{_{[\gamma]: \Tr(\gamma) = t}} \sum \frac{l(\gamma_0)}{\Reis(\gamma)} & \text{if } \F_{\infty} = \ag{C} \end{matrix} \right. .
\label{STGeom}
\end{equation}

	\subsection{Spherical Eisenstein Series}
	
	We need to work with two algebraic groups $\SL_2$ and $\GL_2$. Let $\gp{P}$ resp. $\gp{P}_1$ denote the upper triangular subgroup of $\GL_2$ resp. $\SL_2$. Let $\gp{N}$ denote the upper unipotent subgroup of both $\SL_2$ and $\GL_2$, $\gp{A}_1$ resp. $\gp{A}$ the split sub-torus of diagonal matrices, $\gp{Z}$ resp. $\gp{Z}_1$ be the center of $\GL_2$ resp. $\SL_2$, $\gp{K}$ resp. $\gp{K}_1$ be the standard maximal compact subgroup of $\GL_2$ resp. $\SL_2$. Let $\F$ be a general number field with ring of integers $\vo$, ring of adeles $\A = \F_{\infty}^{\times} \times \A_{\fin}$. Let $\widehat{\vo}$ be the closure of $\vo$ in $\A_{\fin}$. The class group admits a description
	$$ \A^{\times} \to \F^{\times} \backslash \A_{\fin}^{\times} / \widehat{\vo}^{\times} \simeq \Cl(\F). $$
	Hence any character $\chi$ of $\Cl(\F)$ can be viewed as a Hecke character of $\A^{\times}$ which is
\begin{itemize}
	\item trivial on $\F_{\infty}^{\times}$,
	\item unramified at every finite place $\vp$.
\end{itemize}
	Let $e_0$ be the function on $\gp{K}$ taking constant value $1$. It can be viewed as a spherical element $e_{\chi}$ in
	$$ \pi(1, \chi) := \Ind_{\gp{P}(\A)}^{\GL_2(\A)} (1,\chi) = \left\{ f: \GL_2(\A) \to \C \ \middle| \ f\left( \begin{pmatrix} t_1 & u \\ & t_2 \end{pmatrix} g \right) = \chi(t_2) \extnorm{\frac{t_1}{t_2}}_{\A}^{\frac{1}{2}} f(g) \right\}. $$
	It determines a flat section $e_{\chi,s} \in \pi(\norm_{\A}^s, \chi \norm_{\A}^{-s})$
	$$ e_{\chi,s} \left( \begin{pmatrix} t_1 & x \\ 0 & t_2 \end{pmatrix} \kappa \right) = \chi(t_2) \extnorm{\frac{t_1}{t_2}}_{\A}^{s+\frac{1}{2}}, \quad \forall t_1,t_2 \in \A^{\times}, x \in \A, \kappa \in \gp{K}. $$
	Hence we get a collection of spherical Eisenstein series
	$$ \eis(s,e_{\chi})(g) := \sideset{}{_{\gamma \in \gp{P}(\F) \backslash \GL_2(\F)}} \sum e_{\chi,s}(\gamma g), \quad g \in \GL_2(\A). $$
\begin{remark}
	If $\chi = 1$ is the trivial character, we shall write $e_s$ for $e_{1,s}$ and $\eis(s,g)$ for $\eis(s,e_1)(g)$.
\end{remark}
	
\begin{lemma}
	The set of double cosets
	$$ \gp{P}_1(\F) \backslash \SL_2(\F) / \SL_2(\vo) $$
	is in bijection with the class group $\Cl(\F)$ of $\F$.
\label{ClId}
\end{lemma}
\begin{proof}
	This is the content of \cite[Proposition 20]{Sie61}. For convenience we include a proof. Recall
	$$ \SL_2(\F) = \gp{P}_1(\F) w \bigsqcup \gp{P}_1(\F) \gp{N}_-(\F), $$
	where $\gp{N}_-$ is the lower unipotent subgroup, and record
	$$ \lambda: \gp{N}_-(\F) \to \Cl(\F), \quad n_-(a/b) \mapsto [\vo a + \vo b], \quad a,b \in \vo, b \neq 0. $$
	Then $\lambda$ extends to $\SL_2(\F)$ by $\lambda(w)=1$ and left invariance of $\gp{P}_1(\F)$. Since
	$$ \begin{pmatrix} a & b \\ c & d \end{pmatrix} \in \SL_2(\vo) \quad \Rightarrow \quad c \vo + d \vo = \vo, $$
	$$ \begin{pmatrix} a & b \\ c & d \end{pmatrix} \in \gp{P}_1(\F) n_-(-c/d) \quad \text{if } d \neq 0, $$
	we easily deduce that for any $\gamma \in \SL_2(\F)$
	$$ \lambda(\gamma) = 1 \quad \Longleftrightarrow \quad \gamma \in \gp{P}_1(\F) \SL_2(\vo). $$
\end{proof}
\begin{lemma}
	Fix once and for all a uniformizer $\varpi_{\vp}$ of $\F_{\vp}$ at each finite place $\vp < \infty$. To any $x \in \F$ we associate an idele $\iota(x)$ defined by $\iota(0)=1$ and for $x \neq 0$
	$$ \iota(x) = (\varpi_{\vp}^{\min(0,v_{\vp}(x))})_{\vp} \in \A_{\fin}^{\times}. $$
	Then we have for any $x \in \F$
	$$ n_-(x) \in \begin{pmatrix} \iota(x)^{-1} & -1 \\ & \iota(x) \end{pmatrix} \SL_2(\widehat{\vo}). $$
	Moreover, $\lambda(n_-(x)) = 1$ if and only if $\iota(x) \in \F^{\times} \widehat{\vo}^{\times}$.
\label{IdeleDes}
\end{lemma}
\begin{proof}
	The first relation follows from
	$$ n_-(x) = \begin{pmatrix} -1/x & -1 \\ & -x \end{pmatrix} \begin{pmatrix} & 1 \\ -1 & -1/x \end{pmatrix}. $$
	For the moreover part, write $x = a/b$ with $a,b \in \vo$ and $b \neq 0$. Then a representative ideal in the class of $\lambda(n_-(x))$, i.e., $\vo a + \vo b$ corresponds to the idele $(\varpi_{\vp}^{\min(v_{\vp}(a), v_{\vp}(b))})_{\vp} \in b^{-1} \iota(x) \widehat{\vo}^{\times}$.
\end{proof}
	
	Recall the height function
	$$ \Ht_{\infty} : \SL(\F_{\infty}) \to \C, \quad \begin{pmatrix} t & u \\ & t^{-1} \end{pmatrix} \kappa \mapsto \extnorm{t}_{\infty}^2, \quad \forall t \in \F_{\infty}^{\times}, u \in \F_{\infty}, \kappa \in \gp{K}_{\infty}. $$
	The classical spherical Eisenstein series is defined by
\begin{equation} 
	\eis^1(s,g_{\infty}) := \sideset{}{_{\gamma \in \gp{P}_1(\vo) \backslash \SL_2(\vo)}} \sum \Ht_{\infty}(\gamma g_{\infty})^{\frac{1}{2}+s}, \quad g_{\infty} \in \SL_2(\F_{\infty}).
\label{Eis1Def}
\end{equation}
\begin{proposition}
	Write $g_{\infty} \in \SL_2(\F_{\infty})$ and let $1_{\fin}$ be the identity element in $\SL_2(\A_{\fin})$ (or $\GL_2(\A_{\fin})$), then we have the relation
	$$ \eis^1(s,g_{\infty}) = \frac{1}{\norm[\Cl(\F)]} \sideset{}{_{\chi \in \Cl(\F)^{\vee}}} \sum \eis(s,e_{\chi})(g_{\infty},1_{\fin}). $$
\label{EisRel}
\end{proposition}
\begin{proof}
	It is easy to see that the sum
	$$ \sideset{}{_{\chi}} \sum e_{\chi,s}(g_{\infty}, g_{\fin}) $$
	is non-vanishing only at elements $(g_{\infty}, g_{\fin})$ for which
	$$ g_{\fin} = \begin{pmatrix} t_1 & * \\ & t_2 \end{pmatrix} \kappa, \quad t_1 \in \A_{\fin}^{\times}, t_2 \in \F^{\times} \widehat{\vo}^{\times}, \kappa \in \gp{K}_{\fin}, $$
	since for $t \in \A^{\times}$
	$$ \frac{1}{\norm[\Cl(\F)]} \sideset{}{_{\chi}} \sum \chi(t) = \Char_{\F^{\times} (\F_{\infty}^{\times} \widehat{\vo}^{\times})}(t). $$
	It follows that for $\gamma \in \SL_2(\F)$ the sum
	$$ \sideset{}{_{\chi}} \sum e_{\chi,s}(\gamma g_{\infty}, \gamma) $$
	is non-vanishing only if $\gamma \in \gp{P}_1(\F) \SL_2(\vo)$ by the above Lemma \ref{ClId} and \ref{IdeleDes}. Consequently,
\begin{align*}
	\frac{1}{\norm[\Cl(\F)]} \sideset{}{_{\chi \in \Cl(\F)^{\vee}}} \sum \eis(s,e_{\chi})(g_{\infty},1_{\fin}) &= \frac{1}{\norm[\Cl(\F)]} \sideset{}{_{\gamma \in \gp{P}_1(\vo) \backslash \SL_2(\vo)}} \sum \sideset{}{_{\chi}} \sum e_{\chi,s}(\gamma g_{\infty}, \gamma) \\
	&= \sideset{}{_{\gamma \in \gp{P}_1(\vo) \backslash \SL_2(\vo)}} \sum \Ht(\gamma g_{\infty})^{\frac{1}{2}+s} = \eis^1(s,g_{\infty}).
\end{align*}
\end{proof}
\begin{corollary}
	$\eis^1(s,g_{\infty})$ has a constant residue at $s=1/2$ equal to
	$$ \Res_{s=1/2} \eis^1(s,g_{\infty}) = \frac{1}{\norm[\Cl(\F)]} \Res_{s=1/2} \eis(s,e_1) = \frac{1}{\norm[\Cl(\F)]} \cdot \frac{\Lambda_{\F}^*(1)}{2\Lambda_{\F}(2)}, $$
	where $\Lambda_{\F}(s)$ is the complete Dedekind zeta-function of $\F$ and $\Lambda^*(1)$ is its residue at $1$.
\label{ResEis}
\end{corollary}
\begin{proof}
	It suffices to notice that $\chi$ is not trivial on $\A^1$, the subgroup of $\A^{\times}$ with adelic norm $1$, unless $\chi=1$ is trivial, which then implies that the intertwining operator
	$$ \Intw(s,\chi): \Ind_{\gp{P}(\A)}^{\GL_2(\A)} (\norm_{\A}^s,\chi \norm_{\A}^{-s}) \to \Ind_{\gp{P}(\A)}^{\GL_2(\A)} (\chi \norm_{\A}^{-s}, \norm_{\A}^s), $$
	whose restriction to spherical elements is multiplication by
	$$ \frac{\Lambda(1-2s, \chi)}{\Lambda(1+2s, \chi^{-1})}, $$
	is holomorphic at $s=1/2$. Hence $\eis(s,e_{\chi})$ is holomorphic at $s=1/2$ for $\chi \neq 1$.
\end{proof}

	\subsection{Heuristic: Rankin--Selberg Method}
	
	In this subsection, we suppose $\Gamma = \PSL_2(\vo)$ is the full modular group. We propose to compute
	$$ I(s;t,f_{\infty}) := \sideset{}{_{[\gamma]: \Tr(\gamma) = t}} \sum \int_{\Gamma \backslash \PSL_2(\F_{\infty})} K(x; f_{\infty}, [\gamma]) \eis^1(s,x) dx $$
	where $f_{\infty}$ is a smooth bi-$\gp{K}_{\infty}$-invariant test function with compact support. By (\ref{Eis1Def}), we can apply the Rankin--Selberg unfolding method to get
\begin{align*}
	I(s;t,f_{\infty}) &= \sideset{}{_{[\gamma]: \Tr(\gamma) = t}} \sum \int_{\gp{P}_1(\vo) \backslash \PSL_2(\F_{\infty})} K(x; f_{\infty}, [\gamma]) \Ht_{\infty}(x)^{\frac{1}{2}+s} dx \\
	&= \int_{\gp{P}_1(\vo) \backslash \PSL_2(\F_{\infty})} \left( \sideset{}{_{\substack{\gamma \in \PSL_2(\vo) \\ \Tr(\gamma) = t}}} \sum f_{\infty}(x^{-1} \gamma x) \right) \cdot \Ht_{\infty}(x)^{\frac{1}{2}+s} dx.
\end{align*}
\begin{definition}
	We introduce the $\gp{P}_1(\vo)$-conjugacy classes of $\PSL_2(\vo)$ by
	$$ [\gamma_1]_{\gp{P}} := \left\{ p^{-1} \gamma p \ \middle| \ p \in \gp{P}_1(\vo) \right\}. $$
\end{definition}
\noindent We regroup the inner summation by $\gp{P}_1(\vo)$-conjugacy classes and get
	$$ I(s;t,f_{\infty}) = \sideset{}{_{[\gamma_1]_{\gp{P}} : \Tr(\gamma_1) = t }} \sum \int_{\PSL_2(\F_{\infty})} f_{\infty}(x^{-1} \gamma_1 x) \Ht_{\infty}(x)^{\frac{1}{2}+s} dx. $$
	If we regard $f_{\infty}$ as a function on $\PGL_2(\ag{R})$ with support contained in the connected component of identity if $\F_{\infty} = \ag{R}$, then we can relate the integral above to the \emph{Zagier's transform} (c.f. \cite[(3.10)]{Za81}, \cite[(3.24)]{Jo90} or \cite[Definition 1.6]{Wu9}) as (c.f. \cite[Proposition 4.10]{Wu9})
\begin{equation}
	\mathcal{Z}f_{\infty}(s,u) := \int_{\PSL_2(\F_{\infty})} f_{\infty}(x^{-1} \begin{pmatrix} -u & -1 \\ 1 & 0 \end{pmatrix} x) \Ht_{\infty}(x)^s dx,
\label{ZagT}
\end{equation}
	$$ \int_{\PSL_2(\F_{\infty})} f_{\infty}(x^{-1} \gamma_1 x) \Ht_{\infty}(x)^{\frac{1}{2}+s} dx = \frac{1}{\norm[c(\gamma_1)]_{\infty}^{\frac{1}{2}+s}} \cdot \mathcal{Z}f_{\infty}(\frac{1}{2}+s, -t), $$
	where we denote by $c=c(\gamma_1)$ such that
	$$ \gamma_1 = \begin{pmatrix} a & b \\ c & d \end{pmatrix}. $$
	We thus obtain
	$$ I(s;t,f_{\infty}) = \mathcal{Z}f_{\infty}(\frac{1}{2}+s, -t) \cdot \sideset{}{_{[\gamma_1]_{\gp{P}} : \Tr(\gamma_1) = t }} \sum \norm[c(\gamma_1)]_{\infty}^{-(\frac{1}{2}+s)}. $$
	
\begin{lemma}
	If we write the summation on the RHS of the above equation as
	$$ L(1/2+s,t^2-4) := \sideset{}{_{[\gamma_1]_{\gp{P}} : \Tr(\gamma_1) = t }} \sum \norm[c(\gamma_1)]_{\infty}^{-\frac{1}{2}-s}, $$
	then we have
	$$ L(s, t^2 - 4) = \norm[\vo^{\times} / (\vo^{\times})^2] \sideset{}{_{(n) \subset \vo}} \sum \frac{\rho((n),t)}{\Nr(n)^{s}}, $$
	where the sum is over all \emph{principal} integral ideals $(n) \neq 0$ and
	$$ \rho((n),t) := \extnorm{ \left\{ x \pmod{(2n)} \ \middle| \ x \in \vo, x^2 \equiv t^2-4 \pmod{(4n)} \right\} }. $$
\label{ArithL}
\end{lemma}
\begin{proof}
	Writing an element $\gamma \in [\gamma_1]_{\gp{P}}$ as
	$$ \gamma = \begin{pmatrix} a & b \\ n & t-a \end{pmatrix}, $$
	and observing the two possible ways of conjugation by elements in $\gp{P}_1(\vo)$
	$$ \begin{pmatrix} 1 & -u \\ & 1 \end{pmatrix} \begin{pmatrix} a & b \\ n & t-a \end{pmatrix} \begin{pmatrix} 1 & u \\ & 1 \end{pmatrix} = \begin{pmatrix} a-un & * \\ n & t-a+un \end{pmatrix}, \quad u \in \vo, $$
	$$ \begin{pmatrix} v^{-1} & \\ & v \end{pmatrix} \begin{pmatrix} a & b \\ n & t-a \end{pmatrix} \begin{pmatrix} v & \\ & v^{-1} \end{pmatrix} = \begin{pmatrix} a & v^{-2}b \\ v^2n & t-a \end{pmatrix}, \quad v \in \vo^{\times} $$
	we see that $\rho((n),t)$ counts the number of solutions to
	$$ a(t-a) \equiv 1 \pmod{n}, \quad a \in \vo/(n). $$
	If we write $x=t+2a$, then the above equation is equivalent to
	$$ x^2 \equiv t^2-4 \pmod{(4n)}, \quad x \equiv t \pmod{2}, \quad x \in \vo/(2n). $$
	We claim that the first equation implies the second one, i.e., if $\vp \mid 2$ with $v_{\vp}(2) = e$, then
	$$ (x+t)(x-t) \equiv 0 \pmod{4} \quad \Rightarrow \quad x \equiv t \pmod{\vp^e}. $$
	Otherwise, assume $v_{\vp}(x-t) = f < e$, then $v_{\vp}(x+t) = v_{\vp}(x-t+2t) = v_{\vp}(x-t) = f$. Hence $v_{\vp}((x+t)(x-t)) = 2f < 2e = v_{\vp}(t)$, contradiction.
\end{proof}

\noindent The above lemma shows that the Rankin--Selberg unfolding gives the direct link between the counting of geodesics and the counting of the solutions of the relevant congruence equation, which appeared in \cite[p.108]{SY13}. Hence one can carry out the rest of the calculation just like in \cite{SY13} and obtain the desired final formula. But the generalization to arbitrary congruence subgroups of this approach is not obvious. For this reason, we prefer the adelic translation which we now develop.

	\subsection{Adelization}
	
	We return to the general case that $\Gamma < \PSL_2(\vo)$ is a congruence subgroup. This means that there is an integral ideal $\idlN \subseteq \vo$ such that
	$$ \Gamma^1(\idlN) < \Gamma < \PSL_2(\vo), \quad \Gamma^1(\idlN) := \left\{ \gamma \in \PSL_2(\vo) \ \middle| \ \gamma \equiv \begin{pmatrix} 1 & 0 \\ 0 & 1 \end{pmatrix} \pmod{\idlN} \right\}. $$
	Taking closure at a place $\vp < \infty$, we get
	$$ \gp{K}_{\vp}^1(\idlN) < \Gamma_{\vp} < \PSL_2(\vo_{\vp}), \quad \gp{K}_{\vp}^1(\idlN) := \left\{ \gamma \in \PSL_2(\vo_{\vp}) \ \middle| \ \gamma \equiv \begin{pmatrix} 1 & 0 \\ 0 & 1 \end{pmatrix} \pmod{\idlN} \right\}; $$
	or we can write down its pro-finite version
	$$ \gp{K}_{\fin}^1(\idlN) < \widehat{\Gamma} < \PSL_2(\widehat{\vo}). $$
	We take a test function of the form $f = f_{\infty} \otimes \Char_{\widehat{\Gamma}}$ with $f_{\infty}$ bi-$\gp{K}_{\infty}$-invariant, i.e.,
\begin{equation}
	f(x_{\infty},x_{\fin}) = \left\{ \begin{matrix} f_{\infty}(x_{\infty}) & \text{if } x_{\fin} \in \widehat{\Gamma} \\ 0 & \text{otherwise} \end{matrix} \right. ,
\label{TestFctAdel}
\end{equation}
	and consider the orbital integral for a conjugacy class $[\gamma]_1$ in $\PSL_2(\F)$
	$$ I(f,[\gamma]_1) := \int_{\PSL_2(\F) \backslash \PSL_2(\A)} K(x; f, [\gamma]_1) dx, $$
	$$ K(x; f, [\gamma]_1) := \sideset{}{_{\gamma' \in [\gamma]_1}} \sum f(x^{-1} \gamma' x). $$
	Since $\Gamma^1(\idlN) \triangleleft \PSL_2(\vo)$, it is a normal subgroup of $\Gamma$. Hence $\gp{K}_{\fin}^1(\idlN)$ is a normal subgroup of $\widehat{\Gamma}$, i.e.,
	$$ \Char_{\widehat{\Gamma}}(x^{-1} y x) = \Char_{\widehat{\Gamma}}(y), \quad \forall x \in \gp{K}_{\fin}^1(\idlN), y \in \PSL_2(\A_{\fin}). $$
	Hence the integrand in the defining integral of $I(f,[\gamma]_1)$ is a function in $x \in \PSL_2(\F) \backslash \PSL_2(\A)$ invariant by right translation by $\gp{K}_{\fin}^1(\idlN)$. By the strong approximation theorem for $\SL_2$, we have
	$$ \PSL_2(\F) \backslash \PSL_2(\A) / \gp{K}_{\fin}^1(\idlN) \simeq \Gamma^1(\idlN) \backslash \PSL_2(\F_{\infty}). $$
	Hence we can dis-adelize the integral and get
\begin{align*}
	I(f,[\gamma]_1) &= \Vol(\gp{K}_{\fin}(\idlN)) \int_{\Gamma^1(\idlN) \backslash \PSL_2(\F_{\infty})} \sideset{}{_{\gamma' \in [\gamma]_1}} \sum f((x_{\infty}^{-1},1) \gamma' (x_{\infty},1)) dx_{\infty} \\
	&= \Vol(\gp{K}_{\fin}(\idlN)) \int_{\Gamma^1(\idlN) \backslash \PSL_2(\F_{\infty})} \sideset{}{_{\gamma' \in [\gamma]_1 \cap \Gamma}} \sum f_{\infty}(x_{\infty}^{-1} \gamma' x_{\infty}) dx_{\infty}.
\end{align*}
	By definition, $[\gamma]_1 \cap \Gamma$ is the set of elements in $\Gamma$ conjugate to $\gamma$ in $\PSL_2(\F)$. It is a union of conjugacy classes in $\Gamma$. In particular, we have
	$$ \Tr(\gamma') = \Tr(\gamma), \quad \forall \gamma' \in [\gamma]_1 \cap \Gamma. $$
	Hence for any $t \in \vo$ lying in the image of $\Gamma$ under the trace map, we get
	$$ \sideset{}{_{[\gamma]_1 : \Tr(\gamma) = t}} \sum I(f, [\gamma]_1) = \Vol(\gp{K}_{\fin}(\idlN)) \int_{\Gamma^1(\idlN) \backslash \PSL_2(\F_{\infty})} \sideset{}{_{[\gamma]: \Tr(\gamma) = t}} \sum K(x_{\infty}; f_{\infty}, [\gamma]) dx_{\infty}, $$
	where $[\gamma]$ denotes a conjugacy class in $\Gamma$ and we recall the notation in Section \ref{ROI}
	$$ K(x_{\infty}; f_{\infty}, [\gamma]) := \sideset{}{_{\gamma' \in [\gamma]}} \sum f_{\infty}(x_{\infty}^{-1} \gamma x_{\infty}). $$
	But the integrand is a function in $x_{\infty}$ obviously invariant by left translation by $\Gamma$. Hence we get
\begin{align*}
	\sideset{}{_{[\gamma]_1 : \Tr(\gamma) = t}} \sum I(f, [\gamma]_1) &= \Vol(\widehat{\Gamma}) \int_{\Gamma \backslash \PSL_2(\F_{\infty})} \sideset{}{_{[\gamma]: \Tr(\gamma) = t}} \sum K(x_{\infty}; f_{\infty}, [\gamma]) dx_{\infty} \\
	&= \Vol(\widehat{\Gamma}) \cdot I(t,f_{\infty}).
\end{align*}

	Similarly, if we define
	$$ I(s;t,f) := \sideset{}{_{[\gamma]_1 : \Tr (\gamma) = t}} \sum \int_{\PSL_2(\F) \backslash \PSL_2(\A)} K(x; f, [\gamma]_1) \eis(s,x) dx, $$
	then we arrive at
	$$ I(s;t,f) = \Vol(\widehat{\Gamma}) \cdot I(s;t,f_{\infty}). $$
	We record what we have done in the following lemma.
	
\begin{lemma}
	Let $\Gamma < \PSL_2(\vo)$ be a congruence subgroup. Take $f = f_{\infty} \otimes \Char_{\widehat{\Gamma}}$ given in (\ref{TestFctAdel}). For any $\gamma \in \Gamma$ resp. $\PSL_2(\F)$, write $[\gamma]$ resp. $[\gamma]_1$ for the conjugacy class of $\gamma$ in $\Gamma$ resp. $\PSL_2(\F)$. Define
	$$ K(x; f, [\gamma]_1) := \sideset{}{_{\gamma' \in [\gamma]_1}} \sum f(x^{-1} \gamma' x), $$
	$$ I(t,f) := \sideset{}{_{[\gamma]_1 : \Tr(\gamma) = t}} \sum \int_{\PSL_2(\F) \backslash \PSL_2(\A)} K(x; f, [\gamma]_1) dx, $$
	$$ I(s;t,f) := \sideset{}{_{[\gamma]_1 : \Tr (\gamma) = t}} \sum \int_{\PSL_2(\F) \backslash \PSL_2(\A)} K(x; f, [\gamma]_1) \eis(s,x) dx. $$
	Then we have
	$$ I(t,f) = \Vol(\widehat{\Gamma}) \cdot I(t,f_{\infty}), \quad I(s;t,f) = \Vol(\widehat{\Gamma}) \cdot I(s;t,f_{\infty}). $$
\label{AdelRel}
\end{lemma}

	\subsection{From $\SL_2$ to $\GL_2$}
	
	We need to further analyze the conjugacy classes $[\gamma]_1$ in $\PSL_2(\F)$ with $\Tr(\gamma) = t$. Any such $\gamma$ satisfies the equation
	$$ X^2 - t X + 1 = 0, $$
	hence the $\F$-algebra $\F[\gamma] \simeq \E$ for some quadratic separable extension of $\F$ depends only on $t$. In particular, any two such $\gamma$ are \emph{stably conjugate} to each other. By Skolem-Noether theorem, this is equivalent to that they are conjugate in $\PGL_2(\F)$. We record this observation in the following lemma.

\begin{lemma}
	For any $\gamma \in \PSL_2(\F)$ with $\Tr(\gamma) = t$, the disjoint union
	$$ \{ \gamma \} := \sideset{}{_{[\gamma']_1: \Tr(\gamma') = t}} \bigsqcup [\gamma'] $$
	is the set of elements of $\Gamma$ in a single (stable) conjugacy class in $\PGL_2(\F)$.
\end{lemma}

\noindent Taking any hyperbolic $\gamma \in \PSL_2(\vo)$ with $\Tr(\gamma) = t$ and defining
	$$ K(x; f, \{ \gamma \}) := \sideset{}{_{\gamma' \in \{ \gamma \}}} \sum f(x^{-1} \gamma' x) = \sideset{}{_{[\gamma']_1: \Tr(\gamma') = t}} \sum K(x;f,[\gamma']_1), $$
	we can consequently rewrite
	$$ I(s;t,f) = \int_{\PSL_2(\F) \backslash \PSL_2(\A)} K(x; f, \{ \gamma \}) \eis(s,x) dx. $$
	The rest of the subsection is devoted to the analysis of the right hand side of the above equation. There are \emph{a priori} three cases according to the nature of $\E = \F[\gamma]$:
\begin{itemize}
	\item[(1)] $\E$ is a quadratic field extension of $\F$, which is equivalent to $t^2-4 \notin (\F^{\times})^2$;
	\item[(2)] $\E \simeq \F \oplus \F$ is split, which is equivalent to $t^2-4 \in (\F^{\times})^2$;
	\item[(3)] $\E \simeq \F[x]/(x^2)$ is not separable, which is equivalent to $t^2 = 4$.
\end{itemize}
	But Lemma \ref{FEllip} shows that only the case (1) is possible. The realization of $\E$ as a subalgebra in $\Mat_2(\F)$ implies that there exists a(n) (abstract) basis $e_1,e_2 \in \E$ such that
\begin{equation}
	(\gamma e_1, \gamma e_2) = (e_1,e_2) \gamma,
\label{AbsBasis}
\end{equation}
where the multiplication in LHS is interpreted as the abstract one in the algebra $\E$ while the multiplication in RHS is the matrix multiplication. In particular, the matrix realization of $\E^{\times}$ resp. the elements of norm one $\E^1$ becomes an $\F$-torus $\gp{T} = \gp{T}_{\E}$ resp. $\gp{T}^1 = \gp{T}_{\E}^1$ in $\GL_2$ resp. $\SL_2$. Recall
	$$ \eis(s,x) = \sideset{}{_{\sigma \in \gp{P}(\F) \backslash \GL_2(\F)}} \sum e_s(\sigma x) =  \sideset{}{_{\gp{P}_1(\F) \backslash \SL_2(\F)}} \sum e_s(\sigma x), $$
a standard Rankin--Selberg unfolding yields
	$$ I(s;t,f) = \int_{\gp{Z}_1(\A) \gp{P}_1(\F) \backslash \SL_2(\A)} \left( \sideset{}{_{\sigma \in \gp{T}(\F) \backslash \GL_2(\F)}} \sum f(x^{-1} \sigma^{-1} \gamma \sigma x) \right) \cdot e_s(x) dx. $$

\begin{lemma}
	We have a double coset decomposition
	$$ \GL_2(\F) = \sideset{}{_{\alpha \in \F^{\times} / (\F^{\times})^2}} \bigsqcup \gp{T}(\F) \begin{pmatrix} \alpha & \\ & 1 \end{pmatrix} \gp{P}_1(\F). $$
	Moreover, for every $\alpha$ we have
	$$ \gp{T}(\F) \backslash \gp{T}(\F) \begin{pmatrix} \alpha & \\ & 1 \end{pmatrix} \gp{P}_1(\F) \simeq \{ \pm 1 \} \backslash \gp{P}_1(\F) = \gp{Z}_1(\A) \cap \gp{P}_1(\F) \backslash \gp{P}_1(\F). $$
\label{GenIwaDecomp}
\end{lemma}
\begin{proof}
	Identifying $\E$ with $\F \oplus \F$ (row vectors) in view of (\ref{AbsBasis}), the set
	$$ \left\{ (a,b) \in \F \oplus \F \ \middle| \ (a,b) \neq (0,0) \right\} $$
	becomes a single orbit of $\GL_2(\F)$ and corresponds to $\E^{\times}$. The stabilizer of $(0,1)$ being
	$$ \gp{B}_1(\F) = \left\{ \begin{pmatrix} t & x \\ 0 & 1 \end{pmatrix} \ \middle| \ t \in \F^{\times}, x \in \F \right\}, $$
	we deduce the decomposition
	$$ \GL_2(\F) = \gp{B}_1(\F) \gp{T}(\F) = \gp{T}(\F) \gp{B}_1(\F), \quad \gp{T}(\F) \cap \gp{B}_1(\F) = \{ 1 \}. $$
	We introduce the group
	$$ \gp{B}_2(\F) = \left\{ \begin{pmatrix} t^2 & x \\ 0 & 1 \end{pmatrix} \ \middle| \ t \in \F^{\times}, x \in \F \right\}. $$
	Then every element in $\gp{P}_1(\F)$ is the product of an element in $\gp{B}_2(\F)$ and an element in the center $\gp{Z}(\F)$. The desired decomposition then follows from
	$$ \gp{B}_1(\F) = \sideset{}{_{\alpha \in \F^{\times} / (\F^{\times})^2}} \bigsqcup \begin{pmatrix} \alpha & \\ & 1 \end{pmatrix} \gp{B}_2(\F). $$
	The other assertion is easy.
\end{proof}

\noindent We deduce from the previous lemma that
	$$ I(s;t,f) = \sideset{}{_{\alpha \in \F^{\times} / (\F^{\times})^2}} \sum \int_{\gp{Z}_1(\A) \backslash \SL_2(\A)} f(x^{-1} a(\alpha)^{-1} \gamma a(\alpha) x) e_s(x) dx, \quad a(\alpha) = \begin{pmatrix} \alpha & \\ & 1 \end{pmatrix}. $$
	
\begin{lemma}
\begin{itemize}
	\item[(1)] The following function defined over $\A^{\times}$
	$$ h(y) := \norm[y]_{\A}^{\frac{1}{2}+s} \int_{\gp{Z}_1(\A) \backslash \SL_2(\A)} f(x^{-1} a(y)^{-1} \gamma a(y) x) e_s(x) dx $$
is invariant by multiplication by elements in $(\A^{\times})^2$ for whichever test function $f$.
	\item[(2)] Consider $f = f_{\infty} \otimes \Char_{\widehat{\Gamma}}$ as before (with bi-$\gp{K}_{\infty}$-invariant test function $f_{\infty}$). Let $N(\widehat{\Gamma})$ be the normalizer group of $\widehat{\Gamma}$ in $\PGL_2(\A_{\fin})$ and
	$$ \widehat{\vo}_{\Gamma}^{\times} := \left\{ u \in \widehat{\vo}^{\times} \ \middle| \ a(u) \in N(\widehat{\Gamma}) \cap \GL_2(\widehat{\vo}) \right\}. $$
	Then $h(y)$ is invariant by multiplication by elements in $\widehat{\vo}_{\Gamma}^{\times}$.
\end{itemize}
\label{Smoothh}
\end{lemma}
\begin{proof}
	(1) For $z \in \A^{\times}$, we have a decomposition
	$$ a(yz^2) = \begin{pmatrix} z & \\ & z \end{pmatrix} a(y) \begin{pmatrix} z & \\ & z^{-1} \end{pmatrix}. $$
	We get the desired invariance from
\begin{align*}
	&\quad \int_{\gp{Z}_1(\A) \backslash \SL_2(\A)} f(x^{-1} a(yz^2)^{-1} \gamma a(yz^2) x) e_s(x) dx \\
	&= \int_{\gp{Z}_1(\A) \backslash \SL_2(\A)} f(x^{-1} a(y)^{-1} \gamma a(y) x) e_s(\begin{pmatrix} z^{-1} & \\ & z \end{pmatrix}x) dx \\
	&= \norm[z]_{\A}^{-1-2s} \int_{\gp{Z}_1(\A) \backslash \SL_2(\A)} f(x^{-1} a(y)^{-1} \gamma a(y) x) e_s(x) dx.
\end{align*}

\noindent (2) By definition, for any $u \in \widehat{\vo}_{\Gamma}^{\times}$ we have
	$$ \Char_{\widehat{\Gamma}}(a(u)^{-1}ga(u)) = \Char_{\widehat{\Gamma}}(g), \quad \forall g \in \PSL_2(\A_{\fin}). $$
	Writing $x_u = a(u) x a(u)^{-1}$, we thus get
	$$ h(yu) = \norm[y]_{\A}^{\frac{1}{2}+s} \int_{\gp{Z}_1(\A) \backslash \SL_2(\A)} f(x_u^{-1} a(y)^{-1} \gamma a(y) x_u) e_s(x) dx $$
	Since conjugation by $a(u)$ stabilizes both $\gp{B}(\A)$ and $\gp{K}$ resp. $\gp{K}_1$ and leaves the height unchanged, we have $e_s(x) = e_s(x_u)$ and $dx = dx_u$. It follows that $h(yu) = h(y)$.
\end{proof}

\noindent Consequently, $I(s;t,f)$ is the sum over $\F^{\times}/(\F^{\times})^2$ of $h(y)$, a smooth function on $\A^{\times} / (\A^{\times})^2$, to which we can apply Fourier inversion. If we write
	$$ \I_{\F} := \F^{\times} \backslash \A^{\times} $$
for the idele class group of $\F$ and $\I_{\F}^{\vee}$ for its unitary dual group, then we get\footnote{This identity has the following explanation: The image of $\gp{Z}_1(\A) \SL_2(\F) \backslash \SL_2(\A) \hookrightarrow \gp{Z}(\A) \GL_2(\F) \backslash \GL_2(\A)$ is characterized by $\det g \in \I_{\F}^2$.}
\begin{align*} 
	I(s; t, f) &= \frac{1}{\Vol(\I_{\F} / \I_{\F}^2)} \sideset{}{_{\substack{\eta \in \I^{\vee}_{\F} \\ \eta^2 = 1}}} \sum \int_{\A^{\times} / (\A^{\times})^2} \eta(y) \norm[y]_{\A}^{\frac{1}{2}+s} \int_{\gp{Z}_1(\A) \backslash \SL_2(\A)} f(x^{-1} a(y)^{-1} \gamma a(y) x) e_s(x) dx  d^{\times}y \\
	&= \frac{1}{\Vol(\I_{\F} / \I_{\F}^2)} \sideset{}{_{\substack{\eta \in \I^{\vee}_{\F} \\ \eta^2 = 1}}} \sum \int_{\gp{Z}(\A) \backslash \GL_2(\A)} f(x^{-1} \gamma x) e_s(x) \eta(\det x) dx,
\end{align*}
	where $\eta$ runs over quadratic Hecke characters trivial on $\widehat{\vo}_{\Gamma}^{\times}$ defined in Lemma \ref{Smoothh} (2). In particular, the sum over $\eta$ is finite and the number depends only on $\Gamma$.
\begin{lemma}
	If $\F_{\infty} = \R$, $f_{\infty}$ is bi-$\SO_2(\R)$-invariant, and $\eta_{\infty} = \sgn$ is non-trivial, then
	$$ \int_{\gp{Z}(\A) \backslash \GL_2(\A)} f(x^{-1} \gamma x) e_s(x) \eta(\det x) dx = 0. $$
\end{lemma}
\begin{proof}
	There is an outer automorphism of $\SL_2(\R)$ given by
	$$ g \mapsto \begin{pmatrix} -1 & \\ & 1 \end{pmatrix} g \begin{pmatrix} -1 & \\ & 1 \end{pmatrix}, $$
	which obviously leaves a bi-$\SO_2(\R)$-invariant function invariant. Moreover, we also have
	$$ e_{\infty,s} \left( \kappa \begin{pmatrix} -1 & \\ & 1 \end{pmatrix} \right) = e_{\infty,s} \left( \begin{pmatrix} -1 & \\ & 1 \end{pmatrix} \kappa \begin{pmatrix} -1 & \\ & 1 \end{pmatrix} \right) = 1, \quad \kappa \in \SO_2(\R). $$
	Hence the two parts of the infinite component of the integral over $\PGL_2(\R)^+$ and $\PGL_2(\R) - \PGL_2(\R)^+$ cancel with each other, yielding a vanishing integral.
\end{proof}
	
\begin{remark}
	For $\Gamma$ being \emph{principal} or \emph{Hecke} congruence subgroups, we have $\widehat{\vo}_{\Gamma}^{\times} = \widehat{\vo}^{\times}$. The sum in $\eta$ is over quadratic characters unramified at every finite place and each real place, i.e., quadratic class group characters. We thus obtain
	$$ I(s; t, f) = \frac{1}{\Vol(\I_{\F} / \I_{\F}^2)} \sideset{}{_{\substack{\eta \in \Cl(\F)^{\vee} \\ \eta^2 = 1}}} \sum I_{\eta}(s; t, f), $$
\begin{equation}
	I_{\eta}(s; t, f) := \int_{\gp{Z}(\A) \backslash \GL_2(\A)} f(x^{-1} \gamma x) e_s(x) \eta(\det x) dx.
\label{GL2Form}
\end{equation}
\end{remark}

\noindent Note that the notations in (\ref{GL2Form}) suggest that the RHS is independent of the choice of $\gamma \in \SL_2(\F)$ such that $\Tr(\gamma)=t \in \F^{\times} - (\F^{\times})^2$. This is indeed true for arbitrary test function $f$. In fact, any two such $\gamma$ are conjugate by an element $P \in \GL_2(\F)$. Moreover, the proof of Lemma \ref{GenIwaDecomp} shows that we can take $P \in \gp{B}(\F)$. Since $e_s(x)$ is left invariant by $\gp{B}(\F)$ and $\eta$ is trivial on $\F^{\times}$, we get the independence of the choice of $\gamma$.

\noindent Before ending this section, we shall calculate the volume $\Vol(\I_{\F}/\I_{\F}^2)$.

\begin{lemma}
	In (\ref{GL2Form}), the volume $\Vol(\I_{\F}/\I_{\F}^2) = 2$.
\label{QTamM}
\end{lemma}
\begin{proof}
	Following the procedure of passing from $\SL_2$ to $\GL_2$ in (\ref{GL2Form}), the quotient $\I_{\F}/\I_{\F}^2$ is interpreted as the quotient of $\PGL_2(\A)/\PGL_2(\F)$ by $\PSL_2(\A) / \PSL_2(\F)$ in the following way. We have both locally and globally semi-direct product decompositions
	$$ \GL_2(\F_v) = \begin{pmatrix} \F_v^{\times} & 0 \\ 0 & 1 \end{pmatrix} \ltimes \SL_2(\F_v), \quad \GL_2(\A) = \begin{pmatrix} \A^{\times} & 0 \\ 0 & 1 \end{pmatrix} \ltimes \SL_2(\A). $$
	Compatible with these decompositions are the Tamagawa measures on $\GL_2, \GL_1$ and $\SL_2$. In fact, if $\omega_1$ is the $\F$-differential form in the following coordinates of $\SL_2$
	$$ \begin{pmatrix} x & y \\ z & (1+yz)/x \end{pmatrix}, \quad \omega_1 = \norm[x]^{-1} dx dy dz, $$
	and if $\omega$ is the $\F$-differential form in the following coordinates of $\GL_2$
	$$ \begin{pmatrix} x & y \\ z & w \end{pmatrix}, \quad \omega = \norm[xw-yz]^{-2} dxdydzdw, $$
	then in the following coordinates of $\GL_2$
	$$ \begin{pmatrix} t & 0 \\ 0 & 1 \end{pmatrix} \begin{pmatrix} x & y \\ z & (1+yz)/x \end{pmatrix} $$
	one verifies easily that
	$$ \omega = \norm[t]^{-1} dt \omega_1. $$
	If one take $U$ a fundamental domain for $\I_{\F}=\A^{\times} / \F^{\times}$ and $S_1$ a fundamental domain for $\SL_2(\A) / \SL_2(\F)$, then it is easy to verify that
	$$ \begin{pmatrix} U & 0 \\ 0 & 1 \end{pmatrix} S_1 $$
	is a fundamental domain for $\GL_2(\A) / \GL_2(\F)$. Quotient by the center gives
	$$ \GL_2(\A)/\GL_2(\F)\gp{Z}(\A) \simeq \begin{pmatrix} \I_{\F} / \I_{\F}^2 & 0 \\ 0 & 1 \end{pmatrix} \times \SL_2(\A)/\SL_2(\F) \gp{Z}_1(\A). $$
	Thus the volume $\Vol(\I_{\F}/\I_{\F}^2)$ is the ratio of the Tamagawa number of $\GL_2$ by the Tamagawa number of $\SL_2$, which is $2/1=2$.
\end{proof}

\section{Explicit Computation}

	\subsection{Some Arithmetics of Quadratic Orders}
	
	Consider a quadratic field extension $\E / \F$ with ring of integers $\vO_{\E}$. Let $\vO \subset \vO_{\E}$ be a sub-$\vo$-order. At any finite prime $\vp$ of $\vo$, we write $\varpi_{\vp}$ for a uniformizer. There is $n_{\vp} \in \N$ such that
	$$ \vO_{\vp} = \vo_{\vp} + \varpi_{\vp}^{n_{\vp}} \vO_{\E,\vp} $$
	and $n_{\vp} \neq 0$ for finitely many $\vp$. $n_{\vp}$ is called the (local) \emph{level} of $\vO_{\vp}$. It follows that
	$$ \vO = \vo + J \vO_{\E}, \quad J = \sideset{}{_{\vp}} \prod \vp^{n_{\vp}} $$
	and we call the integral ideal $J \subset \vo$ the \emph{level ideal} of $\vO$. To any $\beta \in \vO_{\E}$, we associate an order
	$$ \vO_{\beta} := \vo + \beta \vo. $$
\begin{lemma}
	The level ideal $J_{\beta}$ of $\vO_{\beta}$ as above satisfies
	$$ (\beta - \bar{\beta})^2 \vo = J_{\beta}^2 \Dis_{\E}, $$
	where $\bar{\beta}$ is the conjugation of $\beta$ in $\E$ and $\Dis_{\E} = \Dis(\E/\F)$ is the relative discriminant ideal of $\E/\F$.
\label{LevelIdeal}
\end{lemma}
\begin{proof}
	At a prime $\vp < \infty$ of $\vo$, $\vo_{\vp}$ is a PID. Hence there exists $\theta_{\vp} \in \vO_{\E,\vp}$ such that
	$$ \vO_{\E,\vp} = \vo_{\vp} + \vo_{\vp} \theta_{\vp} \quad \Rightarrow \quad \vO_{\beta,\vp} = \vo_{\vp} + J_{\beta,\vp} \vO_{\E,\vp} = \vo_{\vp} + \varpi_{\vp}^{n_{\vp}} \theta_{\vp} \vo_{\vp}, n_{\vp} = \mathrm{ord}_{\vp}(J_{\beta}). $$
	We also have $\vO_{\beta,\vp} = \vo_{\vp} + \beta \vo_{\vp}$. Hence we can calculate the discriminant of $\vO_{\beta,\vp}$ in two ways and get
	$$ (\beta - \bar{\beta})^2 \vo_{\vp} = \varpi_{\vp}^{2n_{\vp}} (\theta_{\vp} - \bar{\theta}_{\vp})^2 \vo_{\vp}, $$
	from which we deduce the desired equality.
\end{proof}
\begin{remark}
	Obviously $J_{\beta}$ depends only on $t := \Tr(\beta)$. Hence we can write $J_t$ instead of $J_{\beta}$. In the sequel, we shall denote by $[J_t]$ the image of $J_t$ in $\Cl(\F)$.
\end{remark}

	\subsection{Rankin--Selberg Orbital Integrals}

	In this subsection we shall consider a general test function $f$ (see (\ref{GL2Form})) and introduce some notations and terminologies which are parallel to those in the study of trace formulae. We expect these results to be useful for the potential application of the comparison of Rankin--Selberg trace formulae. In particular, one could try and simplify certain proofs towards the Sato--Tate conjectures \cite{SugiyamaTsuzuki} and towards the bias of signs \cite{Martin}. In the next two subsections, we will specialize (the finite part of) $f$ to be the characteristic functions of the principal or Hecke congruence subgroups, and carry out an explicit computation. Recall that in this case $\eta$ appearing in (\ref{GL2Form}) must be quadratic class group characters. We still assume $f_{\infty}$ is bi-$\gp{K}_{\infty}$-invariant.
	
	As noted after (\ref{GL2Form}), the right hand side of that equation depends only on the trace $t$ of $\gamma$. In fact, all $\gamma$'s with trace $t$ are conjugate under $\gp{B}(\F)$ and we shall choose a particular element in the stable conjugacy class of $\gamma$. To this end, we denote by $\E = \E_t$ the quadratic field extension $\F[X] / (X^2 - tX + 1)$. Obviously, $\gamma$ corresponds to an (abstract) element $\beta$ in $\E$ such that
	$$ \E = \F \oplus \F \beta, \quad \beta^2 - t \beta + 1 = 0 \text{ or } \beta \begin{pmatrix} \beta \\ 1 \end{pmatrix} = \begin{pmatrix} t & -1 \\ 1 & 0 \end{pmatrix} \begin{pmatrix} \beta \\ 1 \end{pmatrix}. $$
	We choose $\gamma$ according to this embedding, i.e., we can assume in the following discussion that
	$$ \gamma = \begin{pmatrix} t & -1 \\ 1 & 0 \end{pmatrix}. $$
	
	Secondly $I_{\eta}(s; t, f)$ is decomposable for decomposable $f = \otimes_v' f_v$ with
	$$ I_{\eta}(s;t,f) = \sideset{}{_v} \prod I_{\eta_v}(s; t, f_v), \quad I_{\eta_v}(s; t, f_v) := \int_{\PGL_2(\F_v)} f_v(x^{-1} \gamma x) e_{v,s}(x) \eta_v(\det x) dx. $$
	The computation of the infinite component is simply given by \cite[Proposition 4.10]{Wu9} since $\eta_{\infty} = 1$, i.e., we have (note that difference in the definitions of $e_s$!)
	$$ I_{\eta_{\infty}}(s; t, f_{\infty}) = \int_{\PGL_2(\F_{\infty})} f_{\infty}(x^{-1} \gamma x) e_s(x) dx = \mathcal{Z}f_{\infty}(s+1/2, -t), $$
	since the lower-left entry of $\gamma$ is $1$.
\begin{lemma}
	Suppose $\gamma$ is $\F_{\infty}$-hyperbolic such that for some $P \in \GL_2(\F_{\infty})$
	$$ P^{-1} \gamma P = \begin{pmatrix} N(\gamma)^{1/2} & 0 \\ 0 & N(\gamma)^{-1/2} \end{pmatrix}, \quad \norm[N(\gamma)] > 1. $$
	Then we have (recall the orbital integral (\ref{OrbIntA}))
	$$ \mathcal{Z}f_{\infty}(1,-t) = \frac{\Gamma_{\F_{\infty}}(1)}{\Gamma_{\F_{\infty}}(2)} \vO(t,f_{\infty}). $$
\label{ZagTransAt1}
\end{lemma}
\begin{proof}
	This is \cite[(4.12)]{Za81} for $\F_{\infty} = \R$ and \cite[(3.15)]{Sz83} for $\F_{\infty} = \C$. For convenience, we include a proof, which follows the style of \cite[Proposition 4.9]{Wu9}. We treat the case $\F_{\infty}=\ag{R}$ with details. By assumption $\norm[t] > 2$. An easy computation shows
	$$ \gamma_t P_t = P_t e_t, \quad \gamma_t = \begin{pmatrix} t & -1 \\ 1 & 0 \end{pmatrix}, \quad P_t = \begin{pmatrix} x_1 & x_2 \\ 1 & 1 \end{pmatrix}, \quad e_t = \begin{pmatrix} x_1 & 0 \\ 0 & x_2 \end{pmatrix}, \quad \text{with} $$
	$$ x_1 = \frac{t+\sqrt{t^2-4}}{2}, \quad x_2 = \frac{t-\sqrt{t^2-4}}{2}; \quad x_1 \text{ or } x_2 = N(\gamma). $$
	With the choice of test function $\Phi$ in the Godement section, we have
	$$ \norm[\det g]^s \int_{\ag{R}^{\times}} \Phi((0,t)g) \norm[t]_{\ag{R}}^{2s} d^{\times}t = \frac{\Gamma_{\ag{R}}(2s)}{\Gamma_{\R}(1)} \cdot e_{\infty,s}(g), \quad \Phi(x,y) = e^{-\pi(x^2+y^2)}. $$
	It follows that
\begin{align*}
	\mathcal{Z}f_{\infty}(s,-t) &= \frac{\Gamma_{\R}(1)}{\Gamma_{\ag{R}}(2s)} \int_{\GL_2(\ag{R})} f_{\infty}(g^{-1} \gamma_t g) \Phi((0,1)g) \norm[\det g]^s dg \\
	&= \frac{\Gamma_{\R}(1)}{\Gamma_{\ag{R}}(2s)} \norm[\det P_t]^s \int_{\GL_2(\ag{R})} f_{\infty}(g^{-1} e_t g) \Phi((1,1)g) \norm[\det g]^s dg \\
	&= \frac{\Gamma_{\R}(1)}{\Gamma_{\ag{R}}(2s)} \norm[t^2-4]^{\frac{s}{2}} \int_{\gp{A}(\ag{R}) \backslash \GL_2(\ag{R})} f_{\infty}(g^{-1} e_t g) \int_{\gp{A}(\ag{R})} \Phi((1,1)eg) \norm[\det eg]^s de dg.
\end{align*}
	By the Iwasawa decomposition, we can write $g=n(x)\kappa$. Evaluated at $s=1$, the last integral
\begin{align*}
	\int_{\gp{A}(\ag{R})} \Phi((1,1)eg) \norm[\det eg]^s de &= \int_{\R^2} e^{-\pi(t_1^2 + (t_2+xt_1)^2)} dt_1 dt_2 \\
	&= \int_{\R^2} e^{-\pi(t_1^2 + t_2^2)} dt_1 dt_2 = 1
\end{align*}
	is independent of $g$. We get the desired equality by (\ref{OrbIntA}) and the identification $\gp{A}(\R) \backslash \GL_2(\R) \simeq \gp{A}_1(\R) \backslash \PSL_2(\R)$. For the case $\F_{\infty}=\C$, we take
	$$ \Phi(x,y) = e^{-2\pi x\bar{x} + y\bar{y}} $$
	and change the subscript $\R$ to $\C$ in the above proof.
\end{proof}

	We are left for the computation at finite primes $\vp < \infty$. Let's make some first reductions. We assume that the support of $f_{\vp}$ is contained in $\SL_2(\vo_{\vp})$. Obviously, the non-vanishing of all $I_{\eta_{\vp}}(s; t, f_{\vp})$ implies that $\beta$ is integral over $\vo$ (even a unit in $\vO_{\E}$ of norm $1$). Since $\vo_{\vp}$ is a PID, there exists $\theta_{\vp} \in \E_{\vp}$ such that the ring $\vO_{\vp}$ of integers of $\E_{\vp} = \E \otimes_{\F} \F_{\vp}$ is a free $\vo_{\vp}$-module with basis $\{ 1, \theta_{\vp} \}$
	$$ \vO_{\vp} = \vo_{\vp} + \vo_{\vp} \theta_{\vp}, \quad \theta_{\vp}^2 - \fb_{\vp} \theta_{\vp} + \fa_{\vp} = 0, \quad \text{with} \quad \fa_{\vp}, \fb_{\vp} \in \vo_{\vp}. $$
	It gives an embedding $\iota_{\vp}: \E_{\vp} \to \Mat_2(\F_{\vp})$ determined by
	$$ \theta_{\vp} \begin{pmatrix} \theta_{\vp} \\ 1 \end{pmatrix} = \iota_{\vp}(\theta_{\vp}) \begin{pmatrix} \theta_{\vp} \\ 1 \end{pmatrix}, \quad \text{i.e. } \iota_{\vp}(\theta_{\vp}) = \begin{pmatrix} \fb_{\vp} & - \fa_{\vp} \\ 1 & 0 \end{pmatrix}. $$
\begin{remark}
	We shall identify $\E_{\vp}^{\times}$ with its image under $\iota_{\vp}$ in the sequel. 
\end{remark}
\noindent Since $\beta \in \vO_{\vp} - \F_{\vp}$, we can find $u_{\vp}, v_{\vp} \in \vo_{\vp}$ with $v_{\vp} \neq 0$ such that
	$$ \beta = u_{\vp} + v_{\vp} \theta_{\vp} \quad \Rightarrow \quad \iota_{\vp}(\beta) = \begin{pmatrix} v_{\vp} & u_{\vp} \\ 0 & 1 \end{pmatrix}^{-1} \gamma \begin{pmatrix} v_{\vp} & u_{\vp} \\ 0 & 1 \end{pmatrix} = \begin{pmatrix} * & * \\ v_{\vp} & * \end{pmatrix}. $$
	Consequently, we get
	$$ I_{\eta_{\vp}}(s; t, f_{\vp}) = \eta_{\vp}(v_{\vp}) \norm[v_{\vp}]_{\vp}^{s+\frac{1}{2}} \tilde{I}_{\eta_{\vp}}(s; t, f_{\vp}), $$
	$$ \tilde{I}_{\eta_{\vp}}(s; t, f_{\vp}) := \int_{\PGL_2(\F_{\vp})} f_{\vp}(x^{-1} \iota_{\vp}(\beta) x) e_{\vp,s}(x) \eta_{\vp}(\det x) dx. $$
	By definition, we obviously have (recall $\eta_{\vp}$ is unramified)
\begin{equation}
	(\beta - \bar{\beta})^2 \vo_{\vp} = v_{\vp}^2 (\theta_{\vp} - \bar{\theta}_{\vp})^2 \vo_{\vp} = v_{\vp}^2 \Dis_{\E,\vp} \quad \Rightarrow \quad \eta_{\vp}(v_{\vp}) \norm[v_{\vp}]_{\vp}^{s+\frac{1}{2}} = \eta_{\vp}(J_{t,\vp}) \Nr(J_{t,\vp})^{-(s+\frac{1}{2})}.
\label{EtaFactor}
\end{equation}
	We have two possibilities for $\eta_{\vp}$: $\eta_{\vp}(\varpi_{\vp}) = 1$ (equivalent to $\eta_{\vp}=1$) or $\eta_{\vp}(\varpi_{\vp}) = -1$. We observe that the second case can be reduced to the first one as follows. We can write
	$$ x = \begin{pmatrix} t_1 & * \\ 0 & t_2 \end{pmatrix} \kappa, \quad t_1,t_2 \in \F_{\vp}^{\times}, \kappa \in \GL_2(\vo_{\vp}). $$
	Then we have
	$$ e_{\vp,s}(x) = \extnorm{\frac{t_1}{t_2}}_{\vp}^{s+\frac{1}{2}} = \left( \Nr(\vp)^{s+\frac{1}{2}} \right)^{\mathrm{ord}_{\vp}(t_2) - \mathrm{ord}_{\vp}(t_1)}, \quad \eta_{\vp}(\det x) = (-1)^{\mathrm{ord}_{\vp}(t_1t_2)} = (-1)^{\mathrm{ord}_{\vp}(t_2) - \mathrm{ord}_{\vp}(t_1)}. $$
	Thus if $\tilde{I}_1(s;t,f_{\vp})$ is expressed as a function $H(\Nr(\vp)^{s})$, then $\tilde{I}_{\eta_{\vp}}(s;t,f_{\vp})$ is simply $H(-\Nr(\vp)^{s})$.
	
	We are finally reduced to computing $\tilde{I}_1(s;t,f_{\vp})$, which we write as $\tilde{I}_{\vp}(s;t,f_{\vp})$. Now at a finite place $\vp < \infty$ such that $\E/\F$ is not split, the principality of lattices implies (for details, see the discussion leading to \cite[(4.2)]{Wu9})
	$$ \GL_2(\F_{\vp}) = \sideset{}{_{r=0}^{\infty}} \bigsqcup \E_{\vp}^{\times} a(\varpi_{\vp}^{-r}) \GL_2(\vo_{\vp}), \quad \E_{\vp}^{\times} \subset \F_{\vp}^{\times} \GL_2(\vo_{\vp}), $$
where $\varpi_{\vp}$ is a uniformizer of $\F_{\vp}$. Choosing a Haar measure $de$ on $\F_{\vp}^{\times} \backslash \E_{\vp}^{\times}$, we can define
\begin{equation}
	d_r := \Vol( \E_{\vp}^{\times} \backslash \E_{\vp}^{\times} a(\varpi_{\vp}^{-r}) \GL_2(\vo_{\vp}) ).
\label{RSdrNS}
\end{equation}
	If we define the \emph{(normalized) Rankin--Selberg orbital integral} as
	$$ \RSO_{\gamma}(r, f_{\vp}) := \frac{\int_{\GL_2(\vo_{\vp})} f_{\vp}( \kappa^{-1} a(\varpi_{\vp}^r) \iota_{\vp}(\beta) a(\varpi_{\vp}^{-r}) \kappa ) d\kappa}{\Vol(\GL_2(\vo_{\vp}))}, $$
	denote $\zeta_{\E,\vp}$ for the product of local factors of the Dedekind zeta-function of $\E$ at primes above $\vp$, then we can rewrite
\begin{align}
	\tilde{I}_{\vp}(s; t, f_{\vp}) &= \sideset{}{_{r=0}^{\infty}} \sum \RSO_{\gamma}(r, f_{\vp}) \cdot d_r \cdot \int_{\F_{\vp}^{\times} \backslash \E_{\vp}^{\times}} e_{\vp,s}(ea(\varpi_{\vp}^{-r})) de \nonumber \\
	&=: \frac{\zeta_{\E,\vp}(s+1/2)}{\zeta_{\F,\vp}(2s+1)} \sideset{}{_{r=0}^{\infty}} \sum \RSO_{\gamma}(r, f_{\vp}) \cdot \mathrm{wt}_{\vp}(s, r, \E_{\vp} / \F_{\vp}), \label{RSwtNS}
\end{align}
	Similarly, at a finite place $\vp < \infty$ such that $\E/\F$ is split, the usual Iwasawa decomposition implies
	$$ \GL_2(\F_{\vp}) = \sideset{}{_{r=0}^{\infty}} \bigsqcup P \gp{A}_{\vp} n(\varpi_{\vp}^{-r}) \GL_2(\vo_{\vp}), \quad P = \begin{pmatrix} \theta_{\vp} & \bar{\theta}_{\vp} \\ 1 & 1 \end{pmatrix} \in \GL_2(\vo_{\vp}), $$
where $\gp{A}_{\vp} = \gp{A}(\F_{\vp})$ is the diagonal torus and we have identified $\beta$ with an element in $\vo_{\vp}$, $\bar{\theta}_{\vp} := \fb_{\vp} - \theta_{\vp}$. Choosing a Haar measure $de$ on $\F_{\vp}^{\times} \backslash \E_{\vp}^{\times}$, we can define
\begin{equation}
	d_r := \Vol( \E_{\vp}^{\times} \backslash \E_{\vp}^{\times} P n(\varpi_{\vp}^{-r}) \GL_2(\vo_{\vp}) ) = \Vol( \gp{A}_{\vp} \backslash \gp{A}_{\vp} n(\varpi_{\vp}^{-r}) \GL_2(\vo_{\vp}) ).
\label{RSdrS}
\end{equation}
	If we define the \emph{(normalized) Rankin--Selberg orbital integral} as
	$$ \RSO_{\gamma}(r, f_{\vp}) := \frac{\int_{\GL_2(\vo_{\vp})} f_{\vp}( \kappa^{-1} n(-\varpi_{\vp}^{-r}) P^{-1} \iota_{\vp}(\beta) P n(\varpi_{\vp}^{-r}) \kappa ) d\kappa}{\Vol(\GL_2(\vo_{\vp}))}, $$
	then we can rewrite
\begin{align} 
	\tilde{I}_{\vp}(s; t, f_{\vp}) &= \sideset{}{_{r=0}^{\infty}} \sum \RSO_{\gamma}(r, f_{\vp}) \cdot d_r \cdot \int_{\F_{\vp}^{\times} \backslash \E_{\vp}^{\times}} e_{\vp,s}(ePn(\varpi_{\vp}^{-r})) de \nonumber \\
	&=: \frac{\zeta_{\E,\vp}(s+1/2)}{\zeta_{\F,\vp}(2s+1)} \sideset{}{_{r=0}^{\infty}} \sum \RSO_{\gamma}(r, f_{\vp}) \cdot \mathrm{wt}_{\vp}(s, r, \E_{\vp} / \F_{\vp}), \label{RSwtS}
\end{align}
	The \emph{Rankin--Selberg weights} $\mathrm{wt}_{\vp}(s; r, \E_{\vp} / \F_{\vp})$ are independent of $f_{\vp}$. We record their explicit values and postpone their computation to the next subsection.
\begin{proposition}
	We write $q=q_{\vp}$ for the cardinality of $\vo/\vp$ and $Z := q^s$. Then we have
	$$ \frac{\Vol(\vo_{\vp}^{\times})}{\Vol(\GL_2(\vo_{\vp}))} \mathrm{wt}_{\vp}(s; 0, \E_{\vp} / \F_{\vp}) = L_{\vp}(1, \eta_{\E_{\vp} / \F_{\vp}}) = \left\{ \begin{matrix} (1+q^{-1})^{-1} & \text{if } \E_{\vp} / \F_{\vp} \text{ is unramified} \\ 1 & \text{if } \E_{\vp} / \F_{\vp} \text{ is ramified} \\ (1-q^{-1})^{-1} & \text{if } \E_{\vp} / \F_{\vp} \text{ is split} \end{matrix} \right. , $$
	where $\eta_{\E_{\vp} / \F_{\vp}}$ is the quadratic character associated with the quadratic extension $\E_{\vp} / \F_{\vp}$. We have the following formulae of the weights $\mathrm{wt}_{\vp}(s; r, \E_{\vp} / \F_{\vp})$ for $r \geq 1$.
\begin{itemize}
	\item[(1)] If $\E_{\vp} / \F_{\vp}$ is unramified, then
	$$ \frac{\Vol(\vo_{\vp}^{\times})}{\Vol(\GL_2(\vo_{\vp}))} \mathrm{wt}_{\vp}(s; r, \E_{\vp} / \F_{\vp}) = \frac{ (Z-q^{-1}Z^{-1}) (q^{\frac{1}{2}}Z)^r - (Z^{-1}-q^{-1}Z) (q^{\frac{1}{2}}Z^{-1})^r }{Z-Z^{-1}}. $$
	\item[(2)] If $\E_{\vp} / \F_{\vp}$ is ramified, then
	$$ \frac{\Vol(\vo_{\vp}^{\times})}{\Vol(\GL_2(\vo_{\vp}))} \mathrm{wt}_{\vp}(s; r, \E_{\vp} / \F_{\vp}) = \frac{ (Z-q^{-\frac{1}{2}}) (q^{\frac{1}{2}}Z)^r - (Z^{-1}-q^{-\frac{1}{2}}) (q^{\frac{1}{2}}Z^{-1})^r }{Z-Z^{-1}}. $$
	\item[(3)] If $\E_{\vp} / \F_{\vp}$ is split, then
	$$ \frac{\Vol(\vo_{\vp}^{\times})}{\Vol(\GL_2(\vo_{\vp}))} \mathrm{wt}_{\vp}(s; r, \E_{\vp} / \F_{\vp}) = \frac{ (Z+q^{-1}Z^{-1}-2q^{-\frac{1}{2}}) (q^{\frac{1}{2}}Z)^r - (Z^{-1}+q^{-1}Z-2q^{-\frac{1}{2}}) (q^{\frac{1}{2}}Z^{-1})^r }{Z-Z^{-1}}. $$
\end{itemize}
\label{ExpWts}
\end{proposition}

	We are thus reduced to the computation of the Rankin--Selberg orbital integrals $\RSO_{\gamma}(r,f_{\vp})$ for various concrete choices of $f_{\vp}$. For further convenience of notations, we denote
\begin{equation}
	a_r = \left\{ \begin{matrix} a(\varpi_{\vp}^{-r}) & \text{if } \E_{\vp} \text{ non-split} \\ \begin{pmatrix} \theta_{\vp} & \overline{\theta}_{\vp} \\ 1 & 1 \end{pmatrix} n(\varpi_{\vp}^{-r}) & \text{if } \E_{\vp} \simeq \F_{\vp}^2 \end{matrix} \right. .
\label{RSar}
\end{equation}
	Thus we get a uniform form of the Rankin--Selberg orbital integrals
	$$ \RSO_{\gamma}(r, f_{\vp}) := \frac{\int_{\GL_2(\vo_{\vp})} f_{\vp}( \kappa^{-1} a_r^{-1} \iota_{\vp}(\beta) a_r \kappa ) d\kappa}{\Vol(\GL_2(\vo_{\vp}))}. $$

	\subsection{Rankin--Selberg weights}

	For simplicity of notations, let's drop the subscript $\vp$ in this subsection. We shall compute the Rankin--Selberg weights explicitly, i.e., prove Proposition \ref{ExpWts}. Recall from (\ref{RSwtNS}) and (\ref{RSwtS}) that these weights have the form
	$$ \mathrm{wt}_{\vp}(s, r, \E / \F) = \frac{\zeta_{\F}(2s+1)}{\zeta_{\E}(s+1/2)} \cdot d_r \cdot \int_{\F^{\times} \backslash \E^{\times}} e_s(e a_r) d^{\times}e, $$
where $d_r$ is given in (\ref{RSdrNS}) and (\ref{RSdrS}) and $a_r$ is given in (\ref{RSar}). In particular, these weights do not depend on our choice of measure on $\E^{\times}$. We shall achieve the computation by replacing the flat sections $e_s$ with the \emph{Godement section}, i.e., for $\Phi(x,y) := \mathbbm{1}_{\vo \times \vo}(x,y)$
	$$ e_s(g) = \frac{ \norm[\det g]^{s+1/2} \int_{\F^{\times}} \Phi((0,t)g) \norm[t]^{s+1/2} d^{\times}t }{ \int_{\F^{\times}} \Phi(0,t) \norm[t]^{s+1/2} d^{\times}t }. $$
	We can then decompose the weight as
\begin{align*}
	\mathrm{wt}_{\vp}(s, r, \E / \F) &= \left( d_0 \cdot \frac{\int_{\E^{\times}} \Phi((0,1)e) \norm[e]_{\E}^{s+1/2} d^{\times}e}{\zeta_{\E}(s+1/2)} \right) \cdot \frac{\zeta_{\F}(2s+1)}{\int_{\F^{\times}} \Phi(0,t) \norm[t]^{2s+1} d^{\times}t } \cdot \\
	&\quad \frac{d_r}{d_0} \cdot \norm[\det a_r]^{s+1/2} \cdot \frac{\int_{\E^{\times}} \Phi((0,1)ea_r) \norm[e]_{\E}^{s+1/2} d^{\times}e}{\int_{\E^{\times}} \Phi((0,1)e) \norm[e]_{\E}^{s+1/2} d^{\times}e},
\end{align*}
where each term in the first line depends only on the choice of the measure on $\GL_2(\F)$ resp. $\gp{Z}(\F)$, while each term in the second line depends only on $r$.

\begin{lemma}
	Whatever the measure on $\E^{\times}$ we choose, we have
	$$ \frac{\int_{\F^{\times}} \Phi(0,t) \norm[t]^{2s+1} d^{\times}t }{\zeta_{\F}(2s+1)} = \Vol(\vo^{\times}), \qquad d_0 \cdot \frac{\int_{\E^{\times}} \Phi[(0,1)e] \norm[\det e]^{s+\frac{1}{2}} d^{\times}e }{\zeta_{\E}(s+1/2)} = \Vol(\GL_2(\vo)). $$
\label{MeasCal}
\end{lemma}
\begin{proof}
	The first equation is standard. It implies that the zeta-integral in the second equation is equal to $\zeta_{\E}(s+1/2) \cdot \Vol(\vO_{\E}^{\times})$. By the definition of the quotient measure and the fact that $\vO_{\E}$ is optimally embedded in $\Mat_2(\vo)$, we get
	$$ \Vol(\GL_2(\vo)) = d_0 \cdot \Vol(\vO_{\E}^{\times}), $$
hence the second equation.
\end{proof}

\begin{lemma}
	If $q$ denotes the cardinality of $\vo/\vp$, then we have for $r \geq 1$
	$$ \frac{d_r}{d_0} = \left\{ \begin{matrix} q^r(1+q^{-1}) & \text{if } \E / \F \text{ is unramified} \\ q^r & \text{if } \E/\F \text{ is ramified} \\ q^r(1-q^{-1}) & \text{if } \E/\F \text{ is split} \end{matrix} \right.. $$
\label{VolRatio}
\end{lemma}
\begin{proof}
	Let $\vO_r := \vo + \varpi^r \beta \vo$ denote the order in $\E$ of level $r$. In the non-split case, we have an identification of spaces of orbits
	$$ \E^{\times} \backslash \E^{\times} \begin{pmatrix} \varpi^{-r} & 0 \\ 0 & 1  \end{pmatrix} \GL_2(\vo) \simeq \begin{pmatrix} \varpi^r & 0 \\ 0 & 1  \end{pmatrix} \E^{\times} \begin{pmatrix} \varpi^{-r} & 0 \\ 0 & 1  \end{pmatrix} \cap \GL_2(\vo) \backslash \GL_2(\vo). $$
	We also have
	$$ \begin{pmatrix} \varpi^r & 0 \\ 0 & 1  \end{pmatrix} \E^{\times} \begin{pmatrix} \varpi^{-r} & 0 \\ 0 & 1  \end{pmatrix} \cap \GL_2(\vo) = \begin{pmatrix} \varpi^r & 0 \\ 0 & 1  \end{pmatrix} \iota(\vO_r^{\times}) \begin{pmatrix} \varpi^{-r} & 0 \\ 0 & 1  \end{pmatrix}. $$
	Hence we get the ratio
	$$ d_r/d_0 = [\vO^{\times} : \vO_r^{\times}], $$
	which will be explicitly determined in the next Lemma \ref{OrderUnitsInd}. In the split case, we have similarly
\begin{align*}
	&\quad \E^{\times} \backslash \E^{\times} P \begin{pmatrix} 1 & \varpi^{-r} \\ 0 & 1 \end{pmatrix} \GL_2(\vo) \simeq \gp{A} \backslash \gp{A} \begin{pmatrix} 1 & \varpi^{-r} \\ 0 & 1 \end{pmatrix} \GL_2(\vo) \\
	&\simeq \begin{pmatrix} 1 & -\varpi^{-r} \\ 0 & 1 \end{pmatrix} \gp{A} \begin{pmatrix} 1 & \varpi^{-r} \\ 0 & 1 \end{pmatrix} \cap \GL_2(\vo) \backslash \GL_2(\vo).
\end{align*}
	It is easy to see
	$$ \begin{pmatrix} 1 & -\varpi^{-r} \\ 0 & 1 \end{pmatrix} \gp{A} \begin{pmatrix} 1 & \varpi^{-r} \\ 0 & 1 \end{pmatrix} \cap \GL_2(\vo) = \begin{pmatrix} 1 & -\varpi^{-r} \\ 0 & 1 \end{pmatrix} \left\{ \begin{pmatrix} x & 0 \\ 0 & y \end{pmatrix} \middle| x,y \in \vo^{\times}, x-y \in \varpi^r \vo \right\}  \begin{pmatrix} 1 & \varpi^{-r} \\ 0 & 1 \end{pmatrix}. $$
	Hence we get the desired ratio
	$$ d_r / d_0 = [\vo^{\times} : (1+\varpi^r \vo)] = q^r(1-q^{-1}). $$
\end{proof}

\begin{lemma}
	Let $v_{\E}$ be the normalized additive valuation of $\E$. For any $\ell \in \ag{Z}_{\geq 0}$, consider
	$$ \vO_r^{(\ell)} := \left\{ e \in \vO_r \ \middle| \ v_{\E}(e) = \ell \right\}. $$
\begin{itemize}
	\item[(1)] If $e(\E/\F) \in \{ 1,2 \}$ is the ramification index, then we have
	$$ \vO_r^{(\ell)} = \left\{ \begin{matrix} \varpi_{\E}^{\ell} \vO^{\times} & \text{if } \ell \geq e(\E/\F) r \\ \varpi_{\E}^{\ell} \vO_{r-\ell/e(\E/\F)}^{\times} & \text{if } e(\E/\F) \mid \ell < e(\E/\F) r \\ \emptyset & \text{otherwise} \end{matrix} \right. . $$
	\item[(2)] We have $\vO_r^{\times} < \vO^{\times}$ and
	$$ [\vO^{\times} : \vO_r^{\times}] = \left\{ \begin{matrix} q_{\F}^r (1+q_{\F}^{-1}) & \text{if } e(\E/\F) = 1 \\ q_{\F}^r & \text{if } e(\E/\F) = 2 \end{matrix} \right. . $$
\end{itemize}
\label{OrderUnitsInd}
\end{lemma}
\begin{proof}
	Both assertions are elementary. We omit the details and only point out that (1) follows from the formula
	$$ v_{\E}(x+\beta y) = \left\{ \begin{matrix} \min(v_{\F}(x), v_{\F}(y)) & \text{if } e(\E/\F) = 1 \\ \min(2v_{\F}(x), 2v_{\F}(y)+1) & \text{if } e(\E/\F) = 2 \end{matrix} \right. ; $$
	and (2) follows from the tower and equality
	$$ U_{\E}^{e(\E/\F)r+e(\E/\F)-1} < \vO_r^{\times} = \vo^{\times} U_{\E}^{e(\E/\F)r+e(\E/\F)-1} < \vO^{\times}, \quad \vo^{\times} \cap U_{\E}^{e(\E/\F)r+e(\E/\F)-1} = U_{\F}^{r+e(\E/\F)-1}, $$
	where we have written $U_{\F}^{(\ell)}$ resp. $U_{\E}^{(\ell)}$ for the standard neighborhoods of identity
	$$ U_{\F}^{(\ell)} := 1+\varpi_{\F}^{\ell} \vo, \quad U_{\E}^{(\ell)} := 1+\varpi_{\E}^{\ell} \vO. $$
	Then note that
	$$ [\vO^{\times} : \vO_r^{\times}] = [\vO^{\times} : U_{\E}^{e(\E/\F)r+e(\E/\F)-1}] / [\vo^{\times} : U_{\F}^{r+e(\E/\F)-1}]. $$
\end{proof}

	We have obviously
	$$ \norm[\det a_r] = \left\{ \begin{matrix} q^r & \text{if } \E/\F \text{ is non-split}, \\ 1 & \text{if } \E/\F \text{ is split}. \end{matrix} \right. $$
	
	It remains the last term. It is an analogue of Legendre functions. Exploiting the construction of our embedding $\E \to \Mat_2(\F)$, it is not difficult to identify the last term as the following \emph{Legendre functions associated with the quadratic extension} $\E / \F$.
\begin{definition}
\begin{itemize}
	\item[(1)] If $\E / \F$ is non-split with ring of integers $\vO$, we take $\theta \in \vO - \vo$ such that $\vO = \vo \theta + \vo$. Hence all the $\vo$-orders in $\E$ are listed by
	$$ \vO_r = \vo \varpi^r \theta + \vo, \quad r \in \Z{\geq 0}. $$
	The $r$-th Legendre function is defined to be
	$$ P_s(r, \E/\F) = \frac{\int_{\E} \Char_{\vO_r}(e) \norm[e]_{\E}^s d^{\times}e}{\int_{\E} \Char_{\vO}(e) \norm[e]_{\E}^s d^{\times}e}. $$
	\item[(2)] If $\E / \F$ is split, we identify it with $\F \times \F$. Write $L_0 = \vo \oplus \vo$ and consider the lattices
	$$ L_r = L_0 n(-\varpi^{-r}) = \vo (1, -\varpi^{-r}) + \vo(0,1), \quad r \in \Z_{\geq 0}. $$
	The $r$-th Legendre function is defined to be
	$$ P_s(r, \F^2/\F) = \frac{\int_{\F \times \F} \Char_{L_r}(x,y) \norm[xy]^s d^{\times}x d^{\times}y}{\int_{\F \times \F} \Char_{L_0}(x,y) \norm[xy]^s d^{\times}x d^{\times}y}. $$
\end{itemize}
\label{LegFDef}
\end{definition}
\begin{proposition}
	 If $\F$ is non-archimedean with $q_{\F}$ the cardinality of the residue class field and $\E$ is a quadratic field extension of $\F$, then, writing $Z = q_{\F}^{s-1/2}$, we have for $r \in \ag{Z}_{\geq 0}$
	$$ P_s(r, \E / \F) = \left\{ \begin{matrix} \frac{(Z-q_{\F}^{-1}Z^{-1})(q_{\F}^{\frac{1}{2}}Z)^r - (Z^{-1}-q_{\F}^{-1}Z)(q_{\F}^{\frac{1}{2}}Z^{-1})^r}{q_{\F}^{r(1+s)}(1+q_{\F}^{-1}) (Z-Z^{-1})} & \text{if } \E/\F \text{ is unramified} \\ \frac{(Z-q_{\F}^{-\frac{1}{2}})(q_{\F}^{\frac{1}{2}}Z)^r - (Z^{-1}-q_{\F}^{-\frac{1}{2}})(q_{\F}^{\frac{1}{2}}Z^{-1})^r}{q_{\F}^{r(1+s)} (Z-Z^{-1})} & \text{if } \E/\F \text{ is ramified}  \end{matrix} \right. ; $$
	while if $\E = \F \oplus \F$, then we have for $r \in \ag{Z}_{\geq 0}$
	$$ P_s(r, \F^2 / \F) = \frac{(Z+q_{\F}^{-1}Z^{-1}-2q_{\F}^{-\frac{1}{2}})(q_{\F}^{\frac{1}{2}}Z)^r - (Z^{-1}+q_{\F}^{-1}Z-2q_{\F}^{-\frac{1}{2}})(q_{\F}^{\frac{1}{2}}Z^{-1})^r}{q_{\F}^{r}(1-q_{\F}^{-1}) (Z-Z^{-1})}. $$
\end{proposition}
\begin{proof}
	Write $q_{\E}$ for the cardinality of the residue class field of $\E$. If $\E/\F$ is unramified, then
\begin{align*}
	P_s(r, \E/\F) &= \frac{ \int_{\E} 1_{\vO_r}(e) \norm[e]_{\E}^s d^{\times}e }{ \int_{\E} 1_{\vO}(e) \norm[e]_{\E}^s d^{\times}e } = \frac{ \sideset{}{_{\ell =0}^{\infty}} \sum q_{\E}^{-\ell s} \Vol(\vO_r^{(\ell)}) }{ \Vol(\vO^{\times}) \zeta_{\E}(s) } \\
	&= \left( \sideset{}{_{\ell = 0}^{r-1}} \sum q_{\F}^{- 2 \ell s} [\vO^{\times} : \vO_{r-\ell}^{\times}]^{-1} \right) (1-q_{\F}^{-2s}) + q_{\F}^{-2rs} \\
	&= \frac{1-q_{\F}^{-2s}}{1+q_{\F}^{-1}} \left( \sideset{}{_{\ell=0}^{r-1}} \sum q_{\F}^{-(r-\ell)} q_{\F}^{-2\ell s} \right) + q_{\F}^{-2rs}.
\end{align*}
	If $\E/\F$ is ramified, then
\begin{align*}
	P_s(r, \E/\F) &= \frac{ \int_{\E} 1_{\vO_r}(e) \norm[e]_{\E}^s d^{\times}e }{ \int_{\E} 1_{\vO}(e) \norm[e]_{\E}^s d^{\times}e } = \frac{ \sideset{}{_{\ell =0}^{\infty}} \sum q_{\E}^{-\ell s} \Vol(\vO_r^{(\ell)}) }{ \Vol(\vO^{\times}) \zeta_{\E}(s) } \\
	&= \left( \sideset{}{_{\ell = 0}^{r-1}} \sum q_{\F}^{- 2 \ell s} [\vO^{\times} : \vO_{r-\ell}^{\times}]^{-1} \right) (1-q_{\F}^{-s}) + q_{\F}^{-2rs} \\
	&= (1-q_{\F}^{-s}) \sideset{}{_{\ell = 0}^{r-1}} \sum q_{\F}^{-(r-\ell)} q_{\F}^{- 2 \ell s} + q_{\F}^{-2rs}.
\end{align*}
	In the split case $\E = \F \times \F$, we have
\begin{align*}
	P_s(r, \F^2/\F) &= \frac{ \int_{\F^{\times} \times \F^{\times}} 1_{L_r}(x,y) \norm[xy]_{\F}^s d^{\times}xd^{\times}y }{ \int_{\F^{\times} \times \F^{\times}} 1_{L_0}(x,y) \norm[xy]_{\F}^s d^{\times}xd^{\times}y } = \frac{ \int_{\F^{\times} \times \F^{\times}} 1_{\vo}(x) 1_{\vo}(y+\varpi^{-r}x) \norm[xy]_{\F}^s d^{\times}xd^{\times}y }{ \int_{\F^{\times} \times \F^{\times}} 1_{\vo}(x) 1_{\vo}(y) \norm[xy]_{\F}^s d^{\times}xd^{\times}y } \\
	&= (1-q_{\F}^{-1})^{-1} \cdot \left( \sideset{}{_{\ell=0}^{r-1}} \sum q_{\F}^{-\ell s} q_{\F}^{-(r-\ell)(1-s)} \right) \cdot (1-q_{\F}^{-s})^2 + q_{\F}^{-rs}.
\end{align*}
	We get the desired formulas after some elementary manipulation.
\end{proof}

\noindent Proposition \ref{ExpWts} is thus proved by combining all the above computation.

\begin{remark}
	It is obvious from the explicit formulas that the weights satisfy the functional equation
	$$ \mathrm{wt}(s; r, \E/\F) = \mathrm{wt}(-s; r, \E/\F). $$
	This can be proved via the functional equation for local zeta-integrals.
\end{remark}

	\subsection{Principal Congruence Subgroups}
	
	The case of the full modular group $\Gamma = \PSL_2(\vo)$ has been essentially dealt with in the proof of \cite[Proposition 1.9]{Wu9}. The case of principal congruence subgroups is similar, since $f_{\fin} = \Char_{\widehat{\Gamma}}$ is still invariant by $\GL_2(\widehat{\vo})$-conjugation. For definiteness, let
	$$ \Gamma = \Gamma^1(\idlN) = \left\{ \gamma \in \PSL_2(\vo) \ \middle| \ \gamma \equiv \begin{pmatrix} 1 & 0 \\ 0 & 1 \end{pmatrix} \pmod{\idlN} \right\}, $$
	where $\idlN \subseteq \vo$ is an integral ideal. Suppose $\F[\gamma] \simeq \E$ with $\Tr(\gamma) = t$. We have specified an abstract element $\beta \in \E$ corresponding to $\gamma$. Recall that we have chosen the global and local embeddings so that
	$$ \gamma = \begin{pmatrix} t & -1 \\ 1 & 0 \end{pmatrix}, \qquad \iota_{\vp}(\beta) = u_{\vp} + v_{\vp} \iota_{\vp}(\theta_{\vp}) = \begin{pmatrix} u_{\vp} + v_{\vp} \fb_{\vp} & - v_{\vp} \fa_{\vp} \\ v_{\vp} & u_{\vp} \end{pmatrix}. $$

\begin{lemma}
\begin{itemize}
	\item[(1)] The non-vanishing of $I(s;t,f)$, i.e., the conditions that for any finite place $\vp$ there exists $r \geq 0$ such that
	$$ \RSO_{\gamma}(r,f_{\vp}) = \Char_{\Gamma_{\vp}}(a_r^{-1} \iota_{\vp}(\beta) a_r) = 1, $$
	and that the weighted sum of $I_{\eta}(s;t,f)$ (\ref{GL2Form}) is non zero implies
	$$ t \in \Tr(\Gamma). $$
	\item[(2)] Under this condition, the non-vanishing of $\RSO_{\gamma}(r,f_{\vp})$ is equivalent to
	$$ 0 \leq r \leq \mathrm{ord}_{\vp}(v_{\vp}) - \mathrm{ord}_{\vp}(\idlN). $$
	\item[(3)] Recall $\Dis_{\E}$ the relative discriminant ideal of the quadratic field extension $\E \simeq \F[X]/(X^2-tX+1)$. We have
	$$ \mathrm{ord}_{\vp}(v_{\vp}) = (\mathrm{ord}_{\vp}(t^2-4) - \mathrm{ord}_{\vp}(\Dis_{\E}))/2 = \mathrm{ord}_{\vp}(J_t). $$
\end{itemize}
\label{NonVanLoc}
\end{lemma}
\begin{proof}
	(1) follows directly from the non-adelic translation of $I(s;t,f)$ established in Lemma \ref{AdelRel}. We assume this condition $t \in \Tr(\Gamma)$ in the rest of the proof. 
	
\noindent (2) At $\vp$ for which $\E_{\vp}$ is non-split, we have
	$$ a_r^{-1} \iota_{\vp}(\beta) a_r = \begin{pmatrix} v_{\vp} \fb_{\vp} + u_{\vp} & -v_{\vp} \fa_{\vp} \varpi_{\vp}^r \\ v_{\vp} \varpi_{\vp}^{-r} & u_{\vp} \end{pmatrix}, $$
	while at $\vp$ for which $\E_{\vp} \simeq \F_{\vp}^2$ is split, we have
	$$  a_r^{-1} \iota_{\vp}(\beta) a_r = \begin{pmatrix} v_{\vp} \theta_{\vp} + u_{\vp} & \varpi_{\vp}^{-r} v_{\vp} (\theta_{\vp} - \bar{\theta}_{\vp}) \\ 0 & v_{\vp} \bar{\theta}_{\vp} + u_{\vp} \end{pmatrix}, \quad \text{and} \quad \theta_{\vp} \not\equiv \bar{\theta}_{\vp} \pmod{\vp}. $$
	It follows that if $a_r^{-1} \iota_{\vp}(\beta) a_r \in \Gamma_{\vp}$ for some $r$, then $a_l^{-1} \iota_{\vp}(\beta) a_l \in \Gamma_{\vp}$ for all $0 \leq l \leq r$. The largest such $r$ satisfies
	$$ \mathrm{ord}_{\vp}(v_{\vp} \varpi_{\vp}^{-r}) \quad \text{resp.} \quad \mathrm{ord}_{\vp}(\varpi_{\vp}^{-r} v_{\vp} (\theta_{\vp} - \bar{\theta}_{\vp})) = \mathrm{ord}_{\vp}(\idlN) \quad \Longleftrightarrow \quad r=\mathrm{ord}_{\vp}(v_{\vp}) - \mathrm{ord}_{\vp}(\idlN). $$
(3) This equation follows from the definition, the following relation and Lemma \ref{LevelIdeal}
	$$ t^2-4 = (\beta - \bar{\beta})^2 = v_{\vp}^2 (\theta_{\vp} - \bar{\theta}_{\vp})^2 = v_{\vp}^2 \Dis_{\E}. $$
\end{proof}

\noindent By the above lemma, the proof of \cite[Lemma 3.3]{Wu9} then implies
	$$ \tilde{I}_{\vp}(s; t, \Char_{\Gamma_{\vp}}) = \frac{\zeta_{\E,\vp}(s+1/2)}{\zeta_{\F,\vp}(2s+1)} \cdot P_{\vp}(s+1/2, l_{\vp}), \quad l_{\vp} = \mathrm{ord}_{\vp}(J_t) - \mathrm{ord}_{\vp}(\idlN). $$
	with a polynomial $P_{\vp}(s,l_{\vp})$ in $\Nr(\vp)^s$ and $\Nr(\vp)^{-s}$ which satisfies the functional equation
	$$ P_{\vp}(s,l_{\vp}) = P_{\vp}(1-s,l_{\vp}). $$
	Precisely, writing $Z = \Nr(\vp)^{s-1/2}$ then we have
\begin{equation}
	P_{\vp}(s, l_{\vp}) = \Nr(\vp)^{\frac{l_{\vp}}{2}} \frac{(Z^{l_{\vp}+1} - Z^{-(l_{\vp}+1)}) - \varepsilon_{\E,\vp} \Nr(\vp)^{-\frac{1}{2}} (Z^{l_{\vp}} - Z^{-l_{\vp}})}{Z-Z^{-1}},
\label{PSLocPoly}
\end{equation}
	where $\varepsilon_{\E,\vp} = -1$ resp. $0$ resp. $1$ according as $\E_{\vp}$ is unramified resp. ramified resp. split over $\F_{\vp}$. It satisfies the local Riemann hypothesis by the last part of \cite[Lemma 2.1]{SY13}. We deduce the general case
	$$ \tilde{I}_{\eta_{\vp}}(s; t, \Char_{\Gamma_{\vp}}) = \frac{L_{\vp}(s+1/2, \eta_{\vp}) L_{\vp}(s+1/2, \eta_{\vp} \eta_{\E,\vp})}{\zeta_{\F,\vp}(2s+1)} \cdot P_{\eta_{\vp}}(s+1/2, l_{\vp}), \quad l_{\vp} = \mathrm{ord}_{\vp}(J_t) - \mathrm{ord}_{\vp}(\idlN), $$
	where $P_{\eta_{\vp}}(s+1/2, l_{\vp})$ is obtained from $P_{\vp}(s, l_{\vp})$ by re-defining $Z = \eta_{\vp}(\varpi_{\vp}) \Nr(\vp)^{s-1/2}$.

	\subsection{Hecke Congruence Subgroups}
	
	By omparison with the previous case, the difficulty lies in the computation of the Rankin--Selberg orbital integrals, which are no longer indicator functions, since $\Gamma_{\vp}$ is no longer invariant under conjugation by $\GL_2(\vo_{\vp})$, but only under conjugation by the congruence subgroup
	$$ \gp{K}_0[\idlN] = \left\{ \gamma \in \GL_2(\vo_{\vp}) \ \middle| \ \gamma \equiv \begin{pmatrix} * & * \\ 0 & * \end{pmatrix} \pmod{\idlN} \right\}. $$
	We fix a finite place $\vp$. For simplicity of notations, we introduce
	$$ n := \mathrm{ord}_{\vp}(\idlN), \quad l := \mathrm{ord}_{\vp}(v_{\vp}) = \frac{1}{2} \left( \mathrm{ord}_{\vp}(t^2-4) - \mathrm{ord}_{\vp}(\Dis_{\E}) \right) = \mathrm{ord}_{\vp}(J_t). $$
	The case $n=0$ is already treated. We assume $n \geq 1$ in what follows. Since $\Gamma_{\vp} < \SL_2(\vo_{\vp})$, we deduce from the previous case (with $\idlN = \vo$) that the non-vanishing requires the condition
	$$ (n \geq 1;) \quad 0 \leq r \leq l. $$
	The computation of the split case is different from the non-split case in nature.
	
		\subsubsection{$\vp$ is not split}
		
	For simplicity of notations, we omit the subscript $\vp$ and write $q = \Nr(\vp)$. Recall the decomposition
	$$ \GL_2(\vo) = \sideset{}{_{\alpha_1 \in \vp/\vp^n}} \bigsqcup \begin{pmatrix} 1 & \\ \alpha_1 & 1 \end{pmatrix} \gp{K}_0[\vp^n] \bigsqcup \sideset{}{_{\alpha_2 \in \vo/\vp^n}} \bigsqcup w \begin{pmatrix} 1 & \\ \alpha_2 & 1 \end{pmatrix} \gp{K}_0[\vp^n]. $$
	We have
	$$ \begin{pmatrix} 1 & \\ -\alpha_1 & 1 \end{pmatrix} a_r^{-1} \iota_{\vp}(\beta) a_r \begin{pmatrix} 1 & \\ \alpha_1 & 1 \end{pmatrix} = \begin{pmatrix} v\fb+u - \alpha_1 v \fa \varpi^r & -v\fa \varpi^r \\ v\varpi^{-r} - \alpha_1 v\fb + \alpha_1^2 v\fa \varpi^r & u+\alpha_1 v\fa \varpi^r \end{pmatrix} \in \Gamma_{\vp} $$
	$$ \Longleftrightarrow v\varpi^{-r} - \alpha_1 v\fb + \alpha_1^2 v\fa \varpi^r \in \vp^n \Longleftrightarrow 1 - \alpha_1 \fb \varpi^r + \alpha_1^2 \fa \varpi^{2r} \in \vp^{n+r-l}; $$
	$$ \begin{pmatrix} 1 & \\ -\alpha_2 & 1 \end{pmatrix} w^{-1} a_r^{-1} \iota_{\vp}(\beta) a_r w \begin{pmatrix} 1 & \\ \alpha_2 & 1 \end{pmatrix} = \begin{pmatrix} u-\alpha_2 v \varpi^{-r} & -v \varpi^{-r} \\ v\fa \varpi^r + \alpha_2 v \fb + \alpha_2^2 v \varpi^{-r} & v\fb + u + \alpha_2 v \varpi^{-r} \end{pmatrix} \in \Gamma_{\vp} $$
	$$ \Longleftrightarrow v\fa \varpi^r + \alpha_2 v \fb + \alpha_2^2 v \varpi^{-r} \in \vp^n \Longleftrightarrow \fa \varpi^{2r} + \alpha_2 \fb \varpi^r + \alpha_2^2 \in \vp^{n+r-l}. $$
	
\begin{lemma}
	Assume $n \geq 1$ and $0 \leq r \leq l$. Write
	$$ N(r; n,l) = N_0(r; n,l) + N_{\infty}(r; n,l); $$
	$$ N_0(r; n,l) := \extnorm{ \left\{ \alpha \in \vp/\vp^n \ \middle| \ 1 - \alpha \fb \varpi^r + \alpha^2 \fa \varpi^{2r} \in \vp^{n+r-l} \right\} }, $$
	$$ N_{\infty}(r; n,l) := \extnorm{ \left\{ \alpha \in \vo/\vp^n \ \middle| \ \fa \varpi^{2r} - \alpha \fb \varpi^r + \alpha^2 \in \vp^{n+r-l} \right\} }. $$
	Then we have the following formulae.
\begin{itemize}
	\item[(1)] $N_0(r; n,l)$ is non-vanishing only if $l \geq n$, in which case it is given by
	$$ N_0(r; n,l) = q^{n-1} \cdot \Char_{0 \leq r \leq l-n}. $$
	\item[(2)] If $\E/\F$ is unramified at $\vp$ and if $l \geq n$, then we have
	$$ N_{\infty}(r; n,l) = \left\{ \begin{matrix} q^n & \text{if } 0 \leq r \leq l-n \\ q^{\IntP{\frac{n+l-r}{2}}} & \text{if } l-n < r \leq l \end{matrix} \right. ; $$
	while if $l < n$, then we have
	$$ N_{\infty}(r; n,l) = \left\{ \begin{matrix} 0 & \text{if } 0 \leq r < n-l \\ q^{\IntP{\frac{n+l-r}{2}}} & \text{if } n-l \leq r \leq l \end{matrix} \right. . $$
	\item[(3)] If $\E/\F$ is ramified at $\vp$ and if $l \geq n-1$, then we have
	$$ N_{\infty}(r; n,l) = \left\{ \begin{matrix} q^n & \text{if } 0 \leq r \leq l-n \\ q^{\IntP{\frac{n+l-r}{2}}} & \text{if } l-n < r \leq l \end{matrix} \right. ; $$
	while if $l < n-1$, then we have
	$$ N_{\infty}(r; n,l) = \left\{ \begin{matrix} 0 & \text{if } 0 \leq r < n-l-1 \\ q^{\IntP{\frac{n+l-r}{2}}} & \text{if } n-l-1 \leq r \leq l \end{matrix} \right. . $$
\end{itemize}
\label{PtCountNS}
\end{lemma}
\begin{proof}
	For (1), since $1 - \alpha \fb \varpi^r + \alpha^2 \fa \varpi^{2r} \in 1+\vp$ for any $\alpha \in \vp$, $N_0(r; n,l) \neq 0$ only if $n+r-l \leq 0$, in which case all $\alpha \in \vp/\vp^n$ contribute and $N_0(r; n,l) = q^{n-1}$. For the unramified case (2), we notice that
	$$ \fa \varpi^{2r} - \alpha \fb \varpi^r + \alpha^2 = \Nr(\alpha - \beta \varpi^r), $$
where $\Nr$ is the relative norm map for $\E_{\vp} / \F_{\vp}$. We thus have
	$$ 2 \min(\mathrm{ord}_{\vp}(\alpha), r) = \mathrm{ord}_{\vp}(\Nr(\alpha - \beta \varpi^r)) \geq n+r-l \quad \Longleftrightarrow \quad r \geq n-l \ \& \ \mathrm{ord}_{\vp}(\alpha) \geq \IntCP{\frac{n+r-l}{2}}. $$
	The formulae for $N_{\infty}(r; n,l)$ then follows easily. For the ramified case (3), there is $x_0 \in \vo$ such that $\beta - x_0$ is a uniformizer of $\E_{\vp}$. We thus get
	$$ \fa \varpi^{2r} - \alpha \fb \varpi^r + \alpha^2 = \Nr((\alpha - x_0 \varpi^r) - (\beta - x_0) \varpi^r), $$
	$$ \min(2 \cdot \mathrm{ord}_{\vp}(\alpha-x_0 \varpi^r), 2r+1) = \mathrm{ord}_{\vp}(\Nr(\alpha - \beta \varpi^r)) \geq n+r-l $$
	$$ \Longleftrightarrow \quad r \geq n-l-1 \ \& \ \mathrm{ord}_{\vp}(\alpha-x_0 \varpi^r) \geq \IntCP{\frac{n+r-l}{2}}. $$
	The formulae for $N_{\infty}(r; n,l)$ then follows the same way as the previous case.
\end{proof}

	We have the obvious relation
	$$ \RSO_{\gamma}(r, \Char_{\Gamma_{\vp}}) = \frac{N(r; n,l)}{q^n + q^{n-1}}. $$
	Writing
	$$ P_{\vp}(s; l, n, \E_{\vp} / \F_{\vp}) := \sideset{}{_{r=0}^{\infty}} \sum \RSO_{\gamma}(r, \Char_{\Gamma_{\vp}}) \cdot \mathrm{wt}_{\vp}(s, r, \E_{\vp} / \F_{\vp}) $$
	so that we have
	$$ \tilde{I}_{\vp}(s; t, \Char_{\Gamma_{\vp}}) = \frac{\zeta_{\E,\vp}(s+1/2)}{\zeta_{\F,\vp}(2s+1)} P_{\vp}(s+1/2; l, n, \E_{\vp} / \F_{\vp}), $$
	Lemma \ref{PtCountNS} then readily implies the following formulae (recall $Z=q^{s-1/2}$):
\begin{itemize}
	\item[(1)] If $\E/\F$ is unramifed at $\vp$ and if $l \geq n$, then
\begin{align}
	P_{\vp}(s; l, n, \E_{\vp} / \F_{\vp}) &= q^{\frac{l-n}{2}} \cdot \frac{ (Z^{l-n+1} - Z^{-(l-n+1)}) + q^{-\frac{1}{2}} (Z^{l-n} - Z^{-(l-n)}) }{Z-Z^{-1}} \label{PSLocPolyUnr1} \\
	&\quad + \sideset{}{_{r=1}^n} \sum \frac{q^{\IntP{-\frac{r}{2}}}}{1+q^{-1}} \cdot \frac{ (Z-q^{-1}Z^{-1}) (q^{\frac{1}{2}}Z)^{r+l-n} - (Z^{-1}-q^{-1}Z) (q^{\frac{1}{2}}Z^{-1})^{r+l-n} }{Z-Z^{-1}}; \nonumber
\end{align}
	while if $l < n$, then
\begin{equation}
	P_{\vp}(s; l, n, \E_{\vp} / \F_{\vp}) = q^{l-n} \sideset{}{_{r=0}^{2l-n}} \sum \frac{q^{\IntP{-\frac{r}{2}}}}{1+q^{-1}} \cdot \frac{ (Z-q^{-1}Z^{-1}) (q^{\frac{1}{2}}Z)^{r+n-l} - (Z^{-1}-q^{-1}Z) (q^{\frac{1}{2}}Z^{-1})^{r+n-l} }{Z-Z^{-1}}.
\label{PSLocPolyUnr2}
\end{equation}
	In particular, $P_{\vp}(s; l, n, \E_{\vp} / \F_{\vp})$ is non-vanishing only if $n \leq 2l$.
	\item[(2)] If $\E/\F$ is ramifed at $\vp$ and if $l \geq n-1$, then
\begin{align}
	P_{\vp}(s; l, n, \E_{\vp} / \F_{\vp}) &= q^{\frac{l-n}{2}} \cdot \frac{ Z^{l-n+1} - Z^{-(l-n+1)} }{Z-Z^{-1}} \label{PSLocPolyR1} \\
	&\quad + \sideset{}{_{r=1}^n} \sum \frac{q^{\IntP{-\frac{r}{2}}}}{1+q^{-1}} \cdot \frac{ (Z-q^{-\frac{1}{2}}) (q^{\frac{1}{2}}Z)^{r+l-n} - (Z^{-1}-q^{-\frac{1}{2}}) (q^{\frac{1}{2}}Z^{-1})^{r+l-n} }{Z-Z^{-1}}; \nonumber
\end{align}
	while if $l < n$, then
\begin{equation}
	P_{\vp}(s; l, n, \E_{\vp} / \F_{\vp}) = q^{l-n} \sideset{}{_{r=0}^{2l-n}} \sum \frac{q^{\IntP{-\frac{r}{2}}}}{1+q^{-1}} \cdot \frac{ (Z-q^{-\frac{1}{2}}) (q^{\frac{1}{2}}Z)^{r+n-l} - (Z^{-1}-q^{-\frac{1}{2}}) (q^{\frac{1}{2}}Z^{-1})^{r+n-l} }{Z-Z^{-1}}.
\label{PSLocPolyR2}
\end{equation}
	In particular, $P_{\vp}(s; l, n, \E_{\vp} / \F_{\vp})$ is non-vanishing only if $n \leq \max(2l,l+1)$.
\end{itemize}

\begin{remark}
	Although these polynomials $P_{\vp}(s; l,n, \E_{\vp} / \F_{\vp}) = P_{\vp}(1-s; l,n, \E_{\vp} / \F_{\vp})$ still satisfy the functional equation, local Riemann hypothesis fails in general (ie.\ the zeros of $P_{\vp}(s; l,n, \E_{\vp} / \F_{\vp})$ may not lie on the critical line $\operatorname{Re}(s)=1/2$). For example,
	$$ P_{\vp}(s; 1,1, \E_{\vp} / \F_{\vp}) = \left\{ \begin{matrix} \frac{1}{q^{\frac{1}{2}} + q^{-\frac{1}{2}}}(Z+Z^{-1}) + 1 & \text{if } \E/\F \text{ is unramified at } \vp \\  \frac{1}{q^{\frac{1}{2}} + q^{-\frac{1}{2}}}(Z+Z^{-1}) + \frac{1}{1+q^{-1}} & \text{if } \E/\F \text{ is ramified at } \vp \end{matrix} \right. . $$
\end{remark}

	We deduce the general case
	$$ \tilde{I}_{\eta_{\vp}}(s; t, \Char_{\Gamma_{\vp}}) = \frac{L_{\vp}(s+1/2, \eta_{\vp}) L_{\vp}(s+1/2, \eta_{\vp} \eta_{\E,\vp})}{\zeta_{\F,\vp}(2s+1)} \cdot P_{\eta_{\vp}}(s; l, n, \E_{\vp} / \F_{\vp}), $$
	where $P_{\eta_{\vp}}(s; l, n, \E_{\vp} / \F_{\vp})$ is obtained from $P_{\vp}(s; l, n, \E_{\vp} / \F_{\vp})$ by re-defining $Z = \eta_{\vp}(\varpi_{\vp}) \Nr(\vp)^{s-1/2}$.

		\subsubsection{$\vp$ is split}
		
	We omit the subscript $\vp$ and write $q = \Nr(\vp)$. Recall the decomposition
	$$ \GL_2(\vo) = \sideset{}{_{\alpha_1 \in \vp/\vp^n}} \bigsqcup w \begin{pmatrix} 1 & \\ \alpha_1 & 1 \end{pmatrix} \gp{K}_0[\vp^n] \bigsqcup \sideset{}{_{\alpha_2 \in \vo/\vp^n}} \bigsqcup \begin{pmatrix} 1 & \\ \alpha_2 & 1 \end{pmatrix} \gp{K}_0[\vp^n]. $$
	We have
	$$ \begin{pmatrix} 1 & \\ -\alpha_1 & 1 \end{pmatrix} w^{-1} a_r^{-1} \iota_{\vp}(\beta) a_r w \begin{pmatrix} 1 & \\ \alpha_1 & 1 \end{pmatrix} = \begin{pmatrix} v \bar{\theta} + u & 0 \\ v(\theta - \bar{\theta})(\alpha_1 - \varpi^{-r}) & v \theta + u \end{pmatrix} \in \Gamma_{\vp} $$
	$$ \Longleftrightarrow v(\theta - \bar{\theta})(\alpha_1 - \varpi^{-r}) \in \vp^n \Longleftrightarrow c \in \vp^{r+n}; $$
	$$ \begin{pmatrix} 1 & \\ -\alpha_2 & 1 \end{pmatrix} a_r^{-1} \iota_{\vp}(\beta) a_r \begin{pmatrix} 1 & \\ \alpha_2 & 1 \end{pmatrix} = \begin{pmatrix} v \beta + u + \alpha_2 \varpi^{-r} v (\theta - \bar{\theta}) & \varpi^{-r} v (\theta - \bar{\theta}) \\ - v (\theta - \bar{\theta}) \alpha_2 (1 + \alpha_2 \varpi^{-r}) & v \bar{\theta} + u + \alpha_2 \varpi^{-r} v (\theta - \bar{\theta}) \end{pmatrix} \in \Gamma_{\vp} $$
	$$ \Longleftrightarrow v (\theta - \bar{\theta}) \alpha_2 (1 + \alpha_2 \varpi^{-r}) \in \vp^n \Longleftrightarrow \alpha_2 (\alpha_2 + \varpi^r) \in \vp^{n+r-l}. $$
	
\begin{lemma}
	Assume $n \geq 1$ and $0 \leq r \leq l$. Write
	$$ N(r; n,l) = N_0(r; n,l) + N_{\infty}(r; n,l); $$
	$$ N_0(r; n,l) := \extnorm{ \left\{ \alpha \in \vp/\vp^n \ \middle| \ v(\theta - \bar{\theta})(\alpha - \varpi^{-r}) \in \vp^n \right\} }, $$
	$$ N_{\infty}(r; n,l) := \extnorm{ \left\{ \alpha \in \vo/\vp^n \ \middle| \ \alpha (\alpha + \varpi^r) \in \vp^{n+r-l} \right\} }. $$
	Then we have the following formulae.
\begin{itemize}
	\item[(1)] $N_0(r; n,l)$ is non-vanishing only if $l \geq n$, in which case it is given by
	$$ N_0(r; n,l) = q^{n-1} \cdot \Char_{0 \leq r \leq l-n}. $$
	\item[(2)] If $l \geq n$, then we have
	$$ N_{\infty}(r; n,l) = \left\{ \begin{matrix} q^n & \text{if } 0 \leq r \leq l-n \\ q^{\IntP{\frac{n+l-r}{2}}} & \text{if } l-n < r \leq l \end{matrix} \right. ; $$
	while if $l < n$, then we have
	$$ N_{\infty}(r; n,l) = \left\{ \begin{matrix} 2q^l & \text{if } 0 \leq r < n-l \\ q^{\IntP{\frac{n+l-r}{2}}} & \text{if } n-l \leq r \leq l \end{matrix} \right. . $$
\end{itemize}
\label{PtCountS}
\end{lemma}
\begin{proof}
	(1) is easy. For (2), if $l \geq n$, then $\alpha (\alpha + \varpi^r) \in \vp^{n+r-l}$ is equivalent to $\alpha^2 \in \vp^{n+r-l}$, from which we easily deduce the desired formula. Assume $l < n$. We introduce for integers $0 \leq u,v \leq n$
	$$ N(u,v) := \extnorm{ \left\{ \alpha \in \vo / \vp^n \ \middle| \ \mathrm{ord}_{\vp}(\alpha) = u, \mathrm{ord}_{\vp}(\alpha + \varpi^r) = v \right\} }. $$
	If $N(u,v) \neq 0$ then we must have $\min(u,v) \leq r$, which we assume from now on. We distinguish two cases: (i) $u,v \leq r$; (ii) $\min(u,v) \leq r, \max(u,v) > r$. In the case (i), we have
	$$ N(u,v) = \left\{ \begin{matrix} 0 & \text{if } u \neq v \\ q^{n-u} - q^{n-u-1} & \text{if } u=v<r \\ q^{n-r} - 2 q^{n-r-1} & \text{if } u=v=r \end{matrix} \right. . $$
	In the case (ii), we have
	$$ N(u,v) = \left\{ \begin{matrix} q^{n-\max(u,v)} - q^{n-\max(u,v) - 1} & \text{if } \min(u,v) = r < \max(u,v) < n \\ 1 & \text{if }  \min(u,v) = r < \max(u,v) = n \\ 0 & \text{if } \min(u,v) < r < \max(u, v) \end{matrix} \right. . $$
	We obviously have the relation
	$$ N_{\infty}(r; n,l) = \sideset{}{_{u+v \geq n+r-l}} \sum N(u,v). $$
	If $r < n-l$, i.e., if $2r < n+r-l$, then
	$$ N_{\infty}(r; n,l) = \sideset{}{_{\substack{ \min(u,v) = r \\ \max(u,v) \geq n-l }}} \sum N(u,v) = 2q^l. $$
	If $n-l \leq r \leq l$, i.e., if $2r \geq n+r-l$, then\footnote{A simpler way is to observe that in this case $\alpha(\alpha + \varpi^r) \in \vp^{n+r-l}$ is equivalent to $\alpha \in \vp^{\IntCP{\frac{n+r-l}{2}}}$. This observation applies also to the non-split case.}
	$$ N_{\infty}(r; n,l) = \left( \sideset{}{_{\substack{ \min(u,v) = r \\ \max(u,v) > r }}} \sum + \sideset{}{_{r \geq u=v \geq \IntCP{ \frac{n+r-l}{2} }}} \sum \right) N(u,v) = q^{\IntP{\frac{n+l-r}{2}}}. $$
\end{proof}

	We have the obvious relation
	$$ \RSO_{\gamma}(r, f_{\vp}) = \frac{N(r; n,l)}{q^n + q^{n-1}}. $$
	Writing
	$$ P_{\vp}(s; l, n, \E_{\vp} / \F_{\vp}) := \sideset{}{_{r=0}^{\infty}} \sum \RSO_{\gamma}(r, f_{\vp}) \cdot \mathrm{wt}_{\vp}(s, r, \E_{\vp} / \F_{\vp}) $$
	so that we have
	$$ I_{\vp}(s; t, f_{\vp}) = \frac{\zeta_{\E,\vp}(s+1/2)}{\zeta_{\F,\vp}(2s+1)} P_{\vp}(s+1/2; l, n, \E_{\vp} / \F_{\vp}), $$
	Lemma \ref{PtCountS} then readily implies the following formulae (recall $Z=q^{s-1/2}$):
\begin{itemize}
	\item[(1)] If $l \geq n$, then
\begin{align}
	P_{\vp}(s; l, n, \E_{\vp} / \F_{\vp}) &= q^{\frac{l-n}{2}} \cdot \frac{ (Z^{l-n+1} - Z^{-(l-n+1)}) - q^{-\frac{1}{2}} (Z^{l-n} - Z^{-(l-n)}) }{Z-Z^{-1}} + \sideset{}{_{r=1}^n} \sum \frac{q^{\IntP{-\frac{r}{2}}}}{1+q^{-1}} \cdot \label{PSLocPolyS1} \\
	&\quad \frac{ (Z+q^{-1}Z^{-1}-2q^{-\frac{1}{2}}) (q^{\frac{1}{2}}Z)^{r+l-n} - (Z^{-1}+q^{-1}Z-2q^{-\frac{1}{2}}) (q^{\frac{1}{2}}Z^{-1})^{r+l-n} }{Z-Z^{-1}}; \nonumber
\end{align}
	\item[(2)] If $l < n$, then
\begin{align}
	&\quad P_{\vp}(s; l, n, \E_{\vp} / \F_{\vp}) = \frac{2q^{l} }{q^n+q^{n-1}} \cdot \frac{ (Z^{n-l} - Z^{-(n-l)}) - q^{-\frac{1}{2}} (Z^{n-l-1} - Z^{-(n-l-1)}) }{Z-Z^{-1}} \label{PSLocPolyS2} \\
	&\quad + \sideset{}{_{r=n-l}^{l}} \sum \frac{q^{\IntP{\frac{n+l-r}{2}}}}{q^n+q^{n-1}} \cdot \frac{ (Z+q^{-1}Z^{-1}-2q^{-\frac{1}{2}}) (q^{\frac{1}{2}}Z)^{r} - (Z^{-1}+q^{-1}Z-2q^{-\frac{1}{2}}) (q^{\frac{1}{2}}Z^{-1})^{r} }{Z-Z^{-1}}. \nonumber
\end{align}
\end{itemize}

\begin{remark}
	Although these polynomials $P_{\vp}(s; l,n, \E_{\vp} / \F_{\vp}) = P_{\vp}(-s; l,n, \E_{\vp} / \F_{\vp})$ still satisfy the functional equation, local Riemann hypothesis fails in general. For example,
	$$ P_{\vp}(s; 1,1, \E_{\vp} / \F_{\vp}) = \frac{1}{q^{\frac{1}{2}} + q^{-\frac{1}{2}}}(Z+Z^{-1}) + \frac{1-q^{-1}}{1+q^{-1}} . $$
\end{remark}

	We deduce the general case
	$$ \tilde{I}_{\eta_{\vp}}(s; t, \Char_{\Gamma_{\vp}}) = \frac{L_{\vp}(s+1/2, \eta_{\vp}) L_{\vp}(s+1/2, \eta_{\vp} \eta_{\E,\vp})}{\zeta_{\F,\vp}(2s+1)} \cdot P_{\eta_{\vp}}(s; l, n, \E_{\vp} / \F_{\vp}), $$
	where $P_{\eta_{\vp}}(s; l, n, \E_{\vp} / \F_{\vp})$ is obtained from $P_{\vp}(s; l, n, \E_{\vp} / \F_{\vp})$ by re-defining $Z = \eta_{\vp}(\varpi_{\vp}) \Nr(\vp)^{s-1/2}$.

	\subsection{Proof of Main Result}
	
	We give the final part of the proof in the case of principal congruence subgroups, the other case being similar.
	
	Inserting the local computations into (\ref{GL2Form}), we get
\begin{align*}
	\frac{I(s;t,f)}{\Vol(\GL_2(\widehat{\vo}))} &= \frac{1}{\Vol(\I_{\F} / \I_{\F}^2)} \cdot \extnorm{\frac{\Dis_{\E}}{t^2-4}}_{\infty}^{\frac{s}{2}+\frac{1}{4}} \cdot \mathcal{Z}f_{\infty}(s+\frac{1}{2},-t) \cdot \\
	&\quad \sideset{}{_{\substack{ \eta \in \Cl(\F)^{\vee} \\ \eta^2 = 1 }}} \sum \eta(J_t) \frac{L(s+1/2, \eta) L(s+1/2, \eta \eta_{\E})}{\zeta_{\F}(2s+1)} P_{\eta}(s+\frac{1}{2}, \idlN),
\end{align*}
	where we recall
\begin{itemize}
	\item $\E = \E_t \simeq \F[X] / (X^2 - tX +1)$,
	\item $J_t$ is the ideal determined by $(t^2-4)\vo = J_t^2 \Dis_{\E}$,
	\item $P_{\eta}(s,\idlN)$ is the product of local polynomials $P_{\eta_{\vp}}(s,\idlN)$ given in (\ref{PSLocPoly}).
\end{itemize}
	By Corollary \ref{ResEis} and Lemma \ref{AdelRel}, we also have
	$$ \Res_{s=1/2} I(s;t,f) = \frac{\Vol(\widehat{\Gamma})}{\norm[ \Cl(\F) ]} \cdot \frac{\Lambda_{\F}^*(1)}{2 \Lambda_{\F}(2)} \cdot I(t,f_{\infty}), $$
	with $I(t,f_{\infty})$ given by (\ref{STGeom}). Comparing the above two equations with (\ref{STGeom}), taking into account Lemma \ref{QTamM} and Lemma \ref{ZagTransAt1}, we see that $I_{\eta}(s; t, f)$ is $L(s+1/2, \eta) L(s+1/2, \eta \eta_{\E})$ times a factor holomorphic at $s=1/2$. In particular it is holomorphic at $s=1/2$ unless $\eta=1$. Hence we get
	$$ \frac{\norm[\Dis_{\E}]_{\infty}^{1/2}}{2} \frac{\Gamma_{\F_{\infty}}(1)}{\Gamma_{\F_{\infty}}(2)} \frac{\zeta_{\F}^*(1) L(1,\eta_{\E/\F})}{\zeta_{\F}(2)} P(1,\idlN) = \frac{\Vol(\widehat{\Gamma})}{\norm[ \Cl(\F) ]} \cdot \frac{\Lambda_{\F}^*(1)}{2 \Lambda_{\F}(2)} \cdot \sum_{[\gamma]: \Tr(\gamma) = t} \frac{l(\gamma_0)}{\Reis(\gamma)} \cdot \left\{ \begin{matrix} 1 & \text{if } \F_{\infty}=\R \\ 2\pi & \text{if } \F_{\infty}=\C \end{matrix} \right. . $$
	Theorem \ref{MainPS} is thus proved by summing over suitable $t$ and by taking $P_{\Gamma}(s) = P_1(s,\idlN)$.
\begin{remark}
	The contribution from $I_{\eta}(s;t,f), \eta \neq 1$ to $\Res_{s=1/2} I(s;t,f)$ is zero, which is consistent with the discussion in \cite[\S 3.2 Case 1]{JZ87}. It reflects the orthogonality between the restriction to the diagonal of the elliptic terms in the geometric side of the trace formula and $\eta(\det x)$.
\end{remark}

\section{A Prime Geodesic Theorem for Principal Congruence Subgroups}\label{S5}

    \subsection{Preliminary Results}
    
    We consider the special cases $\F = \Q$ or $\Q(i)$. Let $\idlN$ be an integral ideal generated by $N \in \vo = \Z$ or $\Z[i]$. Let $\Gamma = \Gamma(\idlN)$ be the principal congruence subgroup modulo $\idlN$ of $\SL_2(\vo)$. As the first preliminary task, we need to make our formula explicit in this setting.

\begin{corollary}
    Recall the condition of summation
\begin{equation} 
    \max \left\{ \extnorm{\frac{n+\sqrt{n^2-4}}{2}}_{\infty}, \extnorm{\frac{n-\sqrt{n^2-4}}{2}}_{\infty} \right\} \leq x.
\label{SumCond}
\end{equation}
    The counting function $\Psi_{\Gamma}$ is equal to
    $$ \Psi_{\Gamma}(x) = C [\PSL_2(\vo):\Gamma] \sum_{\substack{n \in 2+\idlN^2 \\ n \text{ satisfies } (\ref{SumCond})}} \sqrt{\Nr\left( \frac{n^2-4}{N^2} \right)} \cdot \mathcal{L} \left( 1, \frac{n^2-4}{N^2} \right), $$
    where $\mathcal{L}(s,\delta)$ is Zagier's $L$-function defined by
\begin{equation}
\mathcal{L}(s,\delta) = \frac{\zeta_{\F}(2s)}{\zeta_{\F}(s)} \sum_{0 \neq q \in \vo} \frac{\rho_q(\delta)}{\Nr(q)^s} =: \sum_{0 \neq q \in \vo} \frac{\lambda_q(\delta)}{\Nr(q)^s},
\label{def:ZagierL}
\end{equation}
    and $\rho_q(\delta)$ is the counting function
    $$ \rho_q(\delta) := \extnorm{ \left\{ x \pmod{2q} \ \middle| \ x^2 \equiv \delta \pmod{4q} \right\} }; $$
    and $C = 1$ or $1/(2\pi)$ according to $\F = \Q$ or $\Q(i)$.
\label{PCSubgpBF}
\end{corollary}
\begin{remark}
\label{1104:rmk}
    We emphasize that in this paper we count the closed geodesics without orientation. Hence there is a factor of $2$ missing in our formula, for example compared to Soundararajan-Young's version \cite[Proposition 2.2]{SY13}.
\end{remark}
\begin{proof}
    This is a direct consequence of Theorem \ref{MainPS} as follows. We first need to determine the set $\Tr(\Gamma)$, which we claim to be $2+\idlN^2$. In fact, any element in $\Gamma$ is of the form
	$$ \begin{pmatrix} 1 + Na & Nb \\ Nc & 1+Nd \end{pmatrix} $$
	for some $a,b,c,d \in \vo$, subject to the condition
	$$ 1 = \det \begin{pmatrix} 1 + Na & Nb \\ Nc & 1+Nd \end{pmatrix} \quad \Leftrightarrow \quad a+d = N(bc - ad) \in \idlN. $$
	Hence the trace $2+N(a+d)$ lies in $2+\idlN^2$. Conversely, let $t \in \vo$. Then the following element has trace $2+N^2t$ and lies in $\Gamma$:
	$$ \begin{pmatrix} 1+N^2t & Nt \\ N & 1 \end{pmatrix}. $$
	Then we notice that the extra local non-vanishing condition in Lemma \ref{NonVanLoc} (2) is automatically satisfied. Now if $\Gamma_0 := \SL_2(\Z[i])$ denotes the full modular group, then (\ref{PSLocPoly}) implies
	$$ L_{\Gamma}(1,n^2-4) = [\PSL_2(\vo):\Gamma] L_{\Gamma_0} \left( 1, \frac{n^2-4}{N^2} \right), \quad n \in 2+\idlN^2. $$
	Theorem \ref{MainPS} implies
	$$ \Psi_{\Gamma}(x) = \sum_{n \in 2+\idlN^2} \norm[d_{n^2-4}]_{\infty}^{1/2} \cdot L_{\Gamma}(1,n^2-4). $$
	The last displayed equation in \cite[\S 3.1]{BBCL20}, which relates $L_{\Gamma_0}(s,\delta)$ to Zagier's $L$-function, together with the obvious fact $d_{n^2-4} = d_{(n^2-4)/N^2}$ yields the desired formula.
\end{proof}

    As the second preliminary task, we need to study some exponential sum which is analoguous to the Kloosterman sums. They will appear when we bound $\Psi_{\Gamma}(X+Y) - \Psi_{\Gamma}(X)$ via the above formula in Corollary \ref{PCSubgpBF}. Precisely, we are interested in the size of
\begin{equation}
    S_q(k,N) := \sum_{b (q)} \rho_q \left( b (4+N^2b) \right) e\left( \extPairing{\frac{b}{q}}{k} \right),
\label{AuxExpSum}
\end{equation} 
    where $k \in \vo/q\vo$ and
\begin{itemize}
    \item the sum is over a complete system of representatives of $\vo/q\vo$,
    \item $e(x) := e^{2\pi i x}$ and $\Pairing{x}{y} = xy$ if $\F = \Q$ or $(xy + \overline{xy})/2$ if $\F = \Q(i)$.
\end{itemize}

\begin{lemma}
    Write $q = q_1q_2$ with $q_1 \mid N^{\infty}$ and $(q_2,N) = 1$. Recall the function $d(q)$ is the number of prime divisors of $q$. Then $S_q(k,N)$ for $k \neq 0$ satisfies the bound
    $$ \extnorm{ S_q(k,N) } \ll d(q_2) \Nr((k,q))^{\frac{1}{2}} \Nr(q)^{\frac{1}{2}}. $$
    While for $k=0$, we have
    $$ S_q(0,N) = \Nr(q_1) \phi(q_2), $$
    where $\phi(\cdot)$ is the Euler's function.
\label{AuxExpSumBd}
\end{lemma}
\begin{proof}
    It is easy to see that $S_q(k,N)$ is multiplicative in $q$, as well as the desired bound and formula. Hence it suffices to consider the case $q=q_2$ and the case $q = p^m$ for some prime $p \mid N$ and some $m \in \Z_{>0}$. We give details for the case $k \neq 0$ as follows.
    
\noindent \textbf{Case 1:} $(q,N) = 1$. Making the changes of variables $x = 2y + Nb$ then $y=y+\tilde{N}$, we see
\begin{align}
    \rho_q(b(4+N^2b)) &= \extnorm{ \left\{ y (q) \ \middle| \ y^2+Nby-b \equiv 0 (q) \right\} } \label{1stCurveCount} \\
    &= \extnorm{ \left\{ y (q) \ \middle| \ y^2+(2\tilde{N}+Nb)y+\tilde{N}^2 \equiv 0 (q) \right\} }, \label{2ndCurveCount}
\end{align}
    where we denote by $\tilde{N}$ an inverse of $N$ mod $q$. For a solution $y(q)$ of (\ref{2ndCurveCount}), we necessarily have $(q,y) = 1$. Thus
    $$ b \equiv -\tilde{N}y-\tilde{N}^3y^{-1}-2\tilde{N}^2 (q). $$
    Inserting the above expression, we get
	\begin{equation}\label{1410:eq001}
	S_q(k,N) = e \left( \extPairing{-\frac{2\tilde{N}^2}{q}}{k} \right) \cdot \sum_{\substack{y (q) \\ (y,q) = 1 }} e \left( \extPairing{\frac{-\tilde{N}y-\tilde{N}^3y^{-1}}{q}}{k} \right),
	\end{equation}
	which is equal to the Kloosterman sum $S(k\tilde{N}, k\tilde{N}^3, q)$ up to a unitary phase factor. The classical bound as recalled in \cite[(16)]{BBCL20} concludes the proof in this case.
	
\noindent \textbf{Case 2:} $q=p^m \mid N^{\infty}$. We may assume $(p,k)=1$, since the general case easily reduces to this special one. Any solution $y(q)$ to (\ref{1stCurveCount}) satisfies $(Ny-1,q) = 1$. Thus 
	$$ b \equiv y^2(1-Ny)^{-1} (q). $$
	We obtain
    \begin{equation}\label{1410:eq002}
    S_q(k,N) = \sum_{y (q)} e \left( \extPairing{\frac{y^2(1-Ny)^{-1}}{q}}{k} \right).
    \end{equation}
    Write $S_m(N)$ for $S_q(k,N)$ and $\psi(x)$ for $e(\Pairing{x}{k})$. We distinguish three cases.
	
\noindent \textbf{Case 2.1:} $m=1$. Then $S_1(N)$ reduces to a quadratic Gauss sum and the required bound follows.

\noindent \textbf{Case 2.2:} $m=2\alpha$ with $\alpha \geq 1$. Write $y = y_0 + p^{\alpha}y_1$ where $y_0, y_1$ traverse $\vo/p^{\alpha}\vo$. We have
	$$ y^2(1-Ny)^{-1} \equiv (1-Ny_0)^{-1} \cdot (y_0^2 + P p^{\alpha} y_1 ) \pmod{p^m}, $$
	where $P = (N(1-Ny_0)^{-1}y_0 + 2) y_0$ and the inverse is understood modulo $p^m$. Thus
	$$ S_m(N) = \sum_{y_0, y_1 (p^{\alpha})} \psi \left( \frac{y_0^2(1-Ny_0)^{-1}}{p^m} \right) \cdot \psi \left( \frac{P}{p^{\alpha}} y_1 \right). $$
	The inner sum over $y_1$ is non-vanishing only if $P \in p^{\alpha}\vo$, which is equivalent to $y_0 \in p^{\alpha}\vo$. We thus get
	$$ S_m(N) = \sum_{y_1 (p^{\alpha})} 1 = \Nr(p)^{\alpha} = \Nr(q)^{\frac{1}{2}}. $$
	
\noindent \textbf{Case 2.3:} $m=2\alpha+1$ with $\alpha \geq 1$. Write $y = y_0 + p^{\alpha}y_1$ where $y_0$ resp. $y_1$ traverse $\vo/p^{\alpha}\vo$ resp. $\vo/p^{\alpha+1}\vo$. We similarly get
	$$ S_m(N) = \sum_{y_0 (p^{\alpha})} \psi \left( \frac{y_0^2(1-Ny_0)^{-1}}{p^m} \right) \sum_{y_1 (p^{\alpha+1})} \psi \left( \frac{y_1^2}{p} + \frac{P}{p^{\alpha+1}} y_1 \right), $$
	where $P = (N(1-Ny_0)^{-1}y_0 + 2) y_0$. The inner sum over $y_1$ is non-vanishing only if $P \in p^{\alpha}\vo$, which is equivalent to $y_0 \in p^{\alpha}\vo$. Writing $P = p^{\alpha} P_1$ in this case, we get
	$$ S_m(N) = \sum_{y_1 (p^{\alpha+1})} \psi \left( \frac{y_1^2}{p} + \frac{P_1}{p} y_1 \right) = \Nr(p)^{\alpha} \sum_{y_1 (p)} \psi \left( \frac{y_1^2}{p} + \frac{P_1}{p} y_1 \right). $$
	Bounding the Gauss sum by $\Nr(p)^{1/2}$, we obtain
	$$ \extnorm{S_m(N)} \leq \Nr(p)^{\alpha+\frac{1}{2}} = \Nr(q)^{\frac{1}{2}}. $$
\end{proof}

Before moving to the proof of Theorem \ref{MainPGT} we also give a small
lemma about Dirichlet series.

\begin{lemma}\label{2610:lemma}
Let $N\in\Z[i]$, $N\neq 0$. For $q\neq 0$, let $g(q)=\Nr(q')\phi(q'')$
where $q'=\gcd(q,N^{\infty})$ and $q''=q/q'$.
Let $s\in\C$ with $\Re(s)>0$. Then
\[
\sum_{q\neq 0} \frac{1}{\Nr(q)^{1+s}} \sum_{q_1^2q_2q_3=q} \mu(q_2)\frac{g(q_3)}{\Nr(q_3)}
=
\frac{\zeta_{\Q(i)}(2+2s)}{\zeta_{\Q(i)}(2+s)}\prod_{p|N}\left(1-\frac{1}{\Nr(p)^{2+s}}\right)^{-1}.
\]
\end{lemma}

\begin{proof}
Since $\Re(s)>0$ the sum is absolutely convergent. Writing $q=q_1^2r^2b$ with $b$ squarefree,
we can therefore rewrite the left-hand side as
\begin{equation}\label{2610:eq002}
\zeta(2+2s)\sum_{r,b\neq 0} \frac{\mu^2(b)}{\Nr(r)^{2+2s}\Nr(b)^{1+s}} \sum_{q_2q_3=r^2b} \mu(q_2)\frac{g(q_3)}{\Nr(q_3)}
\end{equation}
Since $g$ is multiplicative, the inner sum factors as
\[
\sum_{q_2q_3=r^2b} \mu(q_2)\frac{g(q_3)}{\Nr(q_3)}
=
\prod_{\substack{p^{\alpha}||r^2b\\p| N}}
\left(1+\mu(p)\right)
\prod_{\substack{p^{\alpha}||r^2b\\p\nmid N}}
\left(\frac{\phi(p^\alpha)}{\Nr(p^\alpha)}+\mu(p)\frac{\phi(p^{\alpha-1})}{\Nr(p^{\alpha-1})}\right)
=
\frac{\mu(r^2b)}{\Nr(r^2b)}\mathbf{1}_{(r^2b,N)=1}.
\]
In particular, we must have $r=1$. Inserting the above in \eqref{2610:eq002}
we obtain
\[
\sum_{b\neq 0} \frac{\mu(b)\mathbf{1}_{(b,N)=1}}{\Nr(b)^{2+s}}
=
\frac{1}{\zeta_{\Q(i)}(2+s)}\prod_{p|N}\left(1-\frac{1}{\Nr(p)^{2+s}}\right)^{-1}.
\]
\end{proof}

\subsection{Proof of Theorem \ref{MainPGT}}\label{S5.2}
Let us start by taking a smooth version of the counting function $\Psi_\Gamma$,
defined by
\[
\Psi_\Gamma(X,k) = \int_{Y}^{2Y} \Psi_\Gamma(X+u) k(u) du,
\]
where $X^{1/2}<Y\leq X$ and $k$ is a smooth, real-valued function with compact support on $(Y,2Y)$,
of unit mass and such that
$\int_{Y}^{2Y}|k^{(j)}(u)|du \ll Y^{-j}$ for all $j\geq 0$.
This choice of smoothing allows us to write
\begin{equation}\label{3009:eq001}
\Psi_\Gamma(X) = \Psi_\Gamma(X,k) - \int_{Y}^{2Y} (\Psi_\Gamma(X+u)-\Psi_\Gamma(X))k(u)du.
\end{equation}
We claim now that
\begin{equation}\label{2309:eq001}
\Psi_\Gamma(X,k) =
\frac{1}{4}\int_{Y}^{2Y}\Biggl((X+u)^{2}
+
2\Re\Biggl(\sum_{r_j\leq X^{1+\epsilon}/Y} \frac{(X+u)^{1+ir_j}}{1+ir_j}\Biggr)\Biggr)k(u)du
+
O(X^{\frac{71}{64}+\epsilon}),
\end{equation}
and
\begin{equation}\label{2309:eq002}
\Psi_\Gamma(X+u)-\Psi_\Gamma(X)
=
\frac{1}{2}\left(Xu +\frac{u^2}{2}\right) + O(u^{2/5}X^{(4\theta+6)/5+\epsilon}).
\end{equation}
Inserting \eqref{2309:eq001} and \eqref{2309:eq002} into \eqref{3009:eq001}
and using the estimate
\[
\sum_{r_j\leq T} X^{ir_j} \ll T^3,
\]
which follows from the Weyl law \cite[Chap.8 Theorem 9.1]{EGM98}, we obtain
\[
\Psi_\Gamma(X) = \frac{1}{4}X^{2}
+
O\left(\frac{X^{3+\epsilon}}{Y^2} + Y^{2/5}X^{(4\theta+6)/5+\epsilon} + X^{\frac{71}{64}+\epsilon}\right).
\]
Picking $Y=X^{\frac{3}{4}-\frac{\theta}{3}}$ the error becomes $O(X^{\frac{3}{2}+\frac{2\theta}{3}+\epsilon})$,
giving Theorem \ref{MainPGT}.

It is therefore left to prove \eqref{2309:eq001} and \eqref{2309:eq002}.
Let us start with the second formula and use Corollary \ref{PCSubgpBF}
to express the left-hand side as a sum of Zagier's $L$-functions.
For convenience we simplify the condition \eqref{SumCond}
to $X<\Nr(n)\leq X+u$ and replace $\Nr(n^2-4)$ by $\Nr(n^2)+O(\Nr(n))$.
Observing that we also have $\mathcal{L}(1,(n^2-4)/N^2)\ll \Nr(n)^\epsilon$,
we can write
\begin{equation}\label{2309:eq003}
\Psi_\Gamma(X+u) - \Psi_\Gamma(X)
=
\quad C_\Gamma\!\!\!\!
\sum_{\substack{X<\Nr(n)\leq X+u\\n\in 2+\mathfrak{N}^2}}
   \Nr(n) \mathcal{L}\left(1,\frac{n^2-4}{N^2}\right) + O\left(X^{1+\epsilon}\right),
\end{equation}
where by Corollary \ref{PCSubgpBF} the constant $C_\Gamma$ equals
\begin{equation}\label{1104:CGamma}
C_\Gamma= \frac{[\PSL_2(\vo) : \Gamma]}{2\pi\Nr(N)} = \frac{\Nr(N)^2}{2\pi}\prod_{p|N}\left(1-\frac{1}{\Nr(p)^2}\right).
\end{equation}
Let $\delta=(n^2-4)/N^2$ and let $V>0$ a parameter to be chosen later.
We can write
\begin{equation}\label{3009:eq010}
\frac{1}{2\pi i} \int_{(1+\epsilon)} \Gamma(s-1) \mathcal{L}(s,\delta) V^{s-1} ds
=
\mathcal{L}(1,\delta)
+
\frac{1}{2\pi i} \int_{(1/2)} \Gamma(s-1) \mathcal{L}(s,\delta) V^{s-1} ds.
\end{equation}
Recall Theorem \ref{MainPS} shows that Zagier's $L$-function equals a
quadratic Dirichlet $L$-function of modulus dividing $\delta$,
up to a Dirichlet polynomial that can be
bounded on the critical line $\Re s=1/2$ by $O(|s|^A\Nr(\delta)^\epsilon)$.
Therefore we can bound $\mathcal{L}(1/2+it,\delta)\ll (1+|t|)^A \Nr(\delta)^{\theta+\epsilon}$,
where $\theta$ is a subconvexity exponent as in the statement of Theorem \ref{MainPGT}.
We deduce that we have the estimate
\[
\sum_{\substack{X<\Nr(n)\leq X+u\\n\in 2+\mathfrak{N}^2}}
   \frac{\Nr(n)}{2\pi i} \int_{(1/2)} \Gamma(s-1) \mathcal{L}\left(s,\frac{n^2-4}{N}\right) V^{s-1} ds
\ll
u X^{1+2\theta+\epsilon} V^{-1/2}.
\]
We can therefore rewrite \eqref{2309:eq003} as
\[
\Psi_\Gamma(X+u) - \Psi_\Gamma(X)
=
C_\Gamma \sum_{\substack{X<\Nr(n)\leq X+u\\n\in 2+\mathfrak{N}^2}}
   \Nr(n) G_V\left(\frac{n^2-4}{N^2}\right)
+
O(u X^{1+2\theta+\epsilon} V^{-1/2}+X^{1+\epsilon}),
\]
where $G_V(\delta)$ denotes the integral on the line $(1+\epsilon)$ in \eqref{3009:eq010}.
Opening $\mathcal{L}(s,\delta)$ as Dirichlet series, the above sum becomes
\begin{equation}\label{3009:eq011}
\sum_{\substack{X<\Nr(n)\leq X+u\\n\in 2+\mathfrak{N}^2}}
   \Nr(n) G_V\left(\frac{n^2-4}{N^2}\right)
=
\sum_{q\neq 0} \frac{e^{-\Nr(q)/V}}{\Nr(q)} 
\sum_{\substack{X<\Nr(n)\leq X+u\\n\in 2+\mathfrak{N}^2}}
   \Nr(n) \lambda_q\left(\frac{n^2-4}{N^2}\right).
\end{equation}
In analogy to \cite[Lemma 2.2]{BBCL20}, one can prove that
\begin{equation}\label{3009:eq012}
\sum_{\substack{\Nr(n)\leq Z\\n+2\in\mathfrak{N}^2}} \lambda_q\left(\frac{n^2-4}{N^2}\right)
=
\frac{\pi Z}{\Nr(N^2)} \sum_{q_1^2q_2q_3=q} \mu(q_2)\frac{g(q_3)}{\Nr(q_3)}
+
O(Z^{1/3+\epsilon}\Nr(q)^{1/3+\epsilon}\Nr(q_{*})),
\end{equation}
where $q_{*}$ in the error is a maximal element such that $q/q_{*}$ is squarefree
and $g(q_3)=\Nr(q_3')\phi(q_3'')$ with $q_3'=\gcd(q_3,N)$ and $q_3''=q_3/q_3'$.
To see this,
take a smooth function $f:\mathbb{R}^2\to[0,1]$,
supported on $|z|\leq \sqrt{Z}+\Delta$ for some $\Delta>0$
and with total mass $\widehat{f}(0)=\pi Z$
(for example, take the normalized convolution of the indicator function of a ball
of radius $\sqrt{Z}$ and one of radius $\Delta$).
Starting from the sum of $\rho_q(n)$, we can write
\[
\sum_{\substack{\Nr(n)\leq Z\\n+2\in\mathfrak{N}^2}} \rho_q\left(\frac{n^2-4}{N^2}\right)
=
\sum_{b\;(q)} \rho_q(b(4+N^2b)) \sum_{m\in\mathbb{Z}[i]} f(2+N^2b+N^2qm)
+
O(\Nr(q)^\epsilon\Nr(q_1) \sqrt{Z} \Delta),
\]
where we use pointwise bounds on $\rho_q(n)$ (see \cite[(15)]{BBCL20})
for the error term.
Applying Poisson summation we transform the sum into
\[
\frac{1}{\Nr(qN^2)} \sum_{k\in\Z[i]} \widehat{f}\left(\frac{k}{qN^2}\right) e\left(\left\langle\frac{2}{qN^2},k\right\rangle\right) S_q(k,N),
\]
with $S_q(k,N)$ the exponential sum defined in \eqref{AuxExpSum}.
The term $k=0$ gives the main contribution, while for the rest we estimate
$S_q(k,N)$ by Lemma \ref{AuxExpSumBd} and we bound $\widehat{f}$
in absolute value. Since $S_q(0,N)=\phi(q)$ if $(q,N)=1$ and $S_q(0,N)=\Nr(q)$ if $q \mid N^{\infty}$
(cf.~\eqref{1410:eq001},\eqref{1410:eq002}), we obtain
\[
\begin{split}
\sum_{\substack{\Nr(n)\leq Z\\n+2\in\mathfrak{N}^2}} \rho_q\left(\frac{n^2-4}{N^2}\right)
&=
\frac{\pi Z g(q)}{\Nr(qN^2)}
+
O(\Nr(q)^\epsilon \Nr(q_1) \sqrt{Z} \Delta + \Nr(q)^{1/2+\epsilon}(1+Z^{1/4}\Delta^{-1/2}))
\\
&=
\frac{\pi Z g(q)}{\Nr(qN^2)}
+
O(Z^{1/3+\epsilon}\Nr(q)^{1/3+\epsilon}\Nr(q_1))
\end{split}
\]
upon taking $\Delta=\Nr(q)^{1/3}Z^{-1/6}$.
Since $\lambda_q$ can be expressed as a divisor sum of $\rho_q$,
see \eqref{def:ZagierL} and \cite[(14)]{BBCL20},
one arrives at \eqref{3009:eq012}.
Using then \eqref{3009:eq012} into \eqref{3009:eq011} we obtain
\[
\sum_{\substack{X<\Nr(n)\leq X+u\\n\in 2+\mathfrak{N}^2}}
   \!\!\!\!\!\!\!\!\! \Nr(n) G_V\left(\frac{n^2-4}{N^2}\right)
=
\frac{\pi}{\Nr(N^2)}\left(Xu+\frac{u^2}{2}\right)
\!\!\!\! \sum_{\substack{q\neq 0\\ q=q_1^2q_2q_3}} \!\!\!\! \frac{\mu(q_2)g(q_3)}{\Nr(q_3q)}e^{-\Nr(q)/V}
+
O\left(X^{4/3+\epsilon}V^{1/3+\epsilon}\right).
\]
We express the exponential in terms of its Mellin transform and we
recognize the sum by means of Lemma \ref{2610:lemma}, so we can write
\[
\sum_{\substack{q\neq 0\\ q=q_1^2q_2q_3}} \!\!\!\! \frac{\mu(q_2)g(q_3)}{\Nr(q)}e^{-\Nr(q)/V}
=
\frac{1}{2\pi i} \int_{(\epsilon)} \Gamma(s)V^{s} \frac{\zeta_{\Q(i)}(2+2s)}{\zeta_{\Q(i)}(2+s)}\prod_{p|N}\left(1-\frac{1}{\Nr(p)^{2+s}}\right)^{-1} ds.
\]
Moving the line of integration to $\Re(s)=-1/2$ and picking up the residue at $s=0$
we deduce that the above equals
\[
\prod_{p|N}\left(1-\frac{1}{\Nr(p)^{2}}\right)^{-1} + O(V^{-1/2}),
\]
which in conclusion leads (in combination with \eqref{1104:CGamma}) to
\[
\Psi_\Gamma(X+u)-\Psi_\Gamma(X)
=
\frac{1}{2}\Bigl(Xu+\frac{u^2}{2}\Bigr)
+
O(X^{4/3+\epsilon}V^{1/3+\epsilon}+u X^{1+2\theta+\epsilon} V^{-1/2}+X^{1+\epsilon}).
\]
We choose $V=(Y^3X^{6\theta-1})^{6/15}$ and obtain \eqref{2309:eq002}.

Let us now prove \eqref{2309:eq001}.
Denote by $Z_\Gamma(s)$ the Selberg zeta function attached to $\Gamma$.
In a box of the form $[-\epsilon,2+\epsilon]\times[-T,T]$ in the complex plane,
the real poles of $Z_\Gamma'/Z_\Gamma$ are at $s=2$
and at a finite number of $s_j\in[1,71/64]$
corresponding to the small eigenvalues attached to $\Gamma$ (see \cite{BB11,Na12}).
Moreover, there are simple poles $s_j=1\pm ir_j$ on the line $\Re s=1$ and
another $O(T\log T)$ simple poles $\rho_j$ in the strip $\Re s\in(0,1)$.
By Perron's formula we start by writing
\[
\frac{1}{2\pi i} \int_{2+\epsilon-iT}^{2+\epsilon+iT} \frac{Z_\Gamma'}{Z_\Gamma}(s)\frac{X^s}{s} ds
=
2\Psi_\Gamma(X) + O\left(\frac{X^{2+\epsilon}}{T}\right)
\]
(recall Remark \ref{1104:rmk} for the extra factor of two on the right).
Next we move the line of integration to $\Re s=-\epsilon$
and we pass the poles of $Z_\Gamma'/Z_\Gamma$, obtaining
\begin{equation}\label{1410:eq003}
\begin{split}
\Psi_\Gamma(X)
=&
\frac{1}{4}X^2
+
X\Re\Biggl(\sum_{r_j\leq T} \frac{X^{ir_j}}{1+ir_j}\Biggr)
\\
&
\pm\frac{1}{4\pi i} \int_{C_1^{\pm}} \frac{Z_\Gamma'}{Z_\Gamma}(s) \frac{X^s}{s} ds
+
\frac{1}{4\pi i} \int_{C_2} \frac{Z_\Gamma'}{Z_\Gamma}(s) \frac{X^s}{s} ds
+
O\left(X^{71/64+\epsilon} + X T^\epsilon + \frac{X^{2+\epsilon}}{T}\right),
\end{split}
\end{equation}
where $C_1^{\pm}$ are the two horizontal segments $[-\epsilon\pm iT,2+\epsilon\pm iT]$
while $C_2$ denotes the vertical segment $[-\epsilon-iT,-\epsilon+iT]$.
At this point we evaluate the above at $X+u$ and integrate against $k(u)$.
Moreover, we select $T\approx X$ and recall that $u\in(Y,2Y)$ with $X^{1/2}< Y \leq X$.
Repeated integration by parts allows us to write
\begin{equation}\label{1410:eq004}
\int_{Y}^{2Y} (X+u)^s k(u) du \ll_{l} \frac{X^{\Re s+l}}{|s Y|^{l}},
\end{equation}
which in turn implies that we can truncate the sum over $r_j$ in \eqref{1410:eq003}
at height $X^{1+\epsilon}/Y$. By standard properties of the Selberg zeta function,
it also implies that we can bound the integral over $C_1^{\pm}$ by
$O(X^{2+\epsilon}T^{-1})=O(X^{1+\epsilon})$.
As for the integral over $C_2$, \eqref{1410:eq004} implies 
that we can truncated the integral at height $|s|\ll X^{1+\epsilon}/Y$
and in that range it gives $O(X^{2+\epsilon}/Y^{2})=O(X^{1+\epsilon})$.
Summarizing, we obtain
\[
\Psi_\Gamma(X)
=
\frac{1}{4}\int_{Y}^{2Y}\Biggl((X+u)^{2}
+
\Re\Biggl(\sum_{r_j\leq X^{1+\epsilon}/Y} \frac{(X+u)^{1+ir_j}}{1+ir_j}\Biggr)\Biggr)k(u)du
+
O(X^{71/64+\epsilon}),
\]
as desired. This proves \eqref{2309:eq001} and thus completes the proof of Theorem \ref{MainPGT}.

To end this section we sketch the proof of Theorem \ref{MainPGT2}
for the case of principal congruence subgroups of $\mathrm{SL}_2(\Z)$.
While the general strategy is the same as for Theorem \ref{MainPGT},
equation \ref{3009:eq012} becomes
\[
\sum_{\substack{n\leq Z\\n+2\in\mathfrak{N}^2}} \lambda_q\left(\frac{n^2-4}{N^2}\right)
=
\frac{Z}{N^2}\sum_{q_1^2q_2q_3=q}\frac{\mu(q_2)g(q_3)}{q_2}
+
O(q^{1/2+\epsilon}q_1),
\]
with $q=q_1^2q_2$ and $q_2$ squarefree, which leads to the identity
\begin{equation}\label{1410:eq010}
\Psi_\Gamma(X+u)-\Psi_\Gamma(X)
=
\frac{u}{2} + O(u^{1/2}X^{1/4+\theta/2+\epsilon}).
\end{equation}
On the other hand, Perron's formula gives the expansion
\begin{equation}\label{1410:eq011}
\Psi_\Gamma(X)
=
\frac{1}{2}\int_{Y}^{2Y}\Biggl(X+u
+
2\Re\Biggl(\sum_{r_j\leq X^{1+\epsilon}/Y} \frac{(X+u)^{1/2+ir_j}}{1+ir_j}\Biggr)\Biggr)k(u)du
+
O(X^{39/64+\epsilon}),
\end{equation}
the exponent $39/64$ following from the estimate $\lambda_1\geq 975/4096$ from \cite{KS02}.
Combining \eqref{1410:eq010} and \eqref{1410:eq011} and using the bound \cite{LS95}
\[
\sum_{r_j \leq T} X^{ir_j} \ll T^{5/4+\epsilon} X^{1/8+\epsilon},
\]
we conclude that
\[
\Psi_\Gamma(X) = \frac{1}{2}X + O(Y^{1/2}X^{1/4+\theta/2+\epsilon} + X^{7/8+\epsilon}Y^{-1} + X^{39/64+\epsilon}),
\]
which gives Theorem \ref{MainPGT2} upon picking $Y=X^{5/6-2\theta/3}$.

\section{Appendix: Finiteness Properties}

	We shall prove Proposition \ref{FinEllConj}. Recall by \cite[Theorem 2.7]{EGM98} that $\Gamma \backslash \ag{H}_3$ has a fundamental domain given by a Poincar\'e normal polyhedron $\vP_Q(\Gamma)$ for some $Q=rj \in \ag{H}_3$ with $r \geq 1$.

\begin{definition}
	For any $P \in \vP_Q(\Gamma)$, write
	$$ \LatL(P) := \{ Q_1 \in \Gamma.Q: d(P,Q_1) = d(P,Q) \}. $$
\end{definition}

\noindent The following lemma is geometrically intuitive. We leave the detail of the proof to the reader.
\begin{lemma}
	We can distinguish the position of a point $P \in \vP=\vP_Q(\Gamma)$ as follows.
\begin{itemize}
	\item[(1)] $P$ lies in the interior of $\vP_Q(\Gamma)$ iff $\LatL(P) = \{ Q \}$ is reduced to a single point.
	\item[(2)] $P$ lies in the interior of a face $\mathcal{S}$ in $\partial \vP_Q(\Gamma)$ iff $\LatL(P) = \{ Q, \gamma.Q \}$ with a unique $1 \neq \gamma \in \Gamma$. The geodesic linking $Q,\gamma.Q$ is perpendicular to $\mathcal{S}$.
	\item[(3)] $P$ lies in the interior of an edge $\mathfrak{s}$ in $\partial \vP_Q(\Gamma)$ iff $\LatL(P)$ is a set of at least three points, all lying in a geodesic plane perpendicular to $\mathfrak{s}$.
	\item[(4)] $P$ is a vertex of $\vP_Q(\Gamma)$ iff $\LatL(P)$ is not contained in any geodesic plane.
\end{itemize}
\end{lemma}

\begin{corollary}
	If $P \in \vP$ as above is in the case (k) and $\gamma \in \Gamma$ such that $\gamma.P \in \vP$, then $\gamma.P$ is also in the case (k), $k=1,2,3,4$.
\label{PNCons}
\end{corollary}
\begin{proof}
	We must have $\gamma.Q \in \LatL(P)$ in this case. Now if $\gamma'.Q \in \LatL(P)$, then
	$$ d(\gamma.P, \gamma \gamma'.Q) = d(P, \gamma'.Q) = d(P,Q) = d(\gamma.P,\gamma.Q) \geq d(\gamma.P,Q) \geq d(P,Q), $$
	we must have equality everywhere, proving that $\gamma.\LatL(P) \subset \LatL(\gamma.P)$. Exchanging the roles of $P$ and $\gamma.P$, we get $\gamma^{-1}\LatL(\gamma.P) \subset \LatL(P)$. Hence $\LatL(\gamma.P) = \gamma.\LatL(P)$. The nature of $\LatL(\gamma.P)$ is the same as $\LatL(P)$.
\end{proof}

\begin{proof}[Proof of Proposition \ref{FinEllConj}]
	Let $[\gamma_0]$ be an elliptic conjugacy class in $\Gamma$. Let $\ell_0$ be the geodesic invariantly fixed by a representative $\gamma_0$. We may assume $P_0 \in \ell_0 \cap \vP_Q(\Gamma)$ exists. $P_0$ can not lie in the interior of $\vP$.
	
\noindent (1) If $P_0$ lies in the interior of a face $\mathcal{S}_0$, we have $\LatL(P_0)=\{ Q, \gamma.Q \}$. Then from
	$$ d(P_0,Q) = d(\gamma_0^n.P_0,Q) = d(P_0,\gamma_0^{-n}.Q), \forall n \in \ag{Z} $$
	we deduce that $\gamma_0^n \in \{ 1,\gamma \}$, hence $\gamma = \gamma_0$ is cyclic of order $2$. Thus $\gamma_0$ is the rotation about the axis $\ell_0$ of angle $\pi$. Consequently, $\ell_0$ and the geodesic linking $Q$ and $\gamma_0.Q$ lie in a geodesic plane and they are perpendicular with each other. Hence $\ell_0$ lies in the geodesic plane containing $\mathcal{S}_0$. As the rotation $\gamma_0$ must map the interior of $\mathcal{S}_0$ into itself by Corollary \ref{PNCons}, $\ell_0$ must be an axis of symmetry of the hyperbolic polygon $\mathcal{S}_0$.
	
\noindent (2) If $P_0$ lies in the interior of an edge $\mathfrak{s}_0$, and if $\ell_0$ does not contain $\mathfrak{s}_0$, then $P_0$ must be the middle point of $\mathfrak{s}_0$ and $\gamma_0$ is a rotation of angle $\pi$, since $\gamma_0$ maps the interior of $\mathfrak{s}_0$ into itself by Corollary \ref{PNCons}. We also have $\gamma_0 \LatL(P_0) = \LatL(P_0)$ by the proof of Corollary \ref{PNCons}, hence $\ell_0$ is an axis of symmetry of the polygon determined by $\LatL(P_0)$. If $\ell_0$ does contain $\mathfrak{s}_0$, then $\gamma_0$ is a rotation about $\mathfrak{s}_0$ which permutes $\LatL(P_0)$, since we still have $\gamma_0 \LatL(P_0) = \LatL(P_0)$.

\noindent (3) If $P_0$ is a vertex of $\vP$, we claim that there exist $P_1 \in \ell_0$ and $\gamma \in \Gamma$ such that $\gamma.P_1 \in \vP$ is not a vertex, hence we can replace $\gamma_0$ resp. $P_0$ with $\gamma \gamma_0 \gamma^{-1}$ resp. $P_1$ and reduce to the previous cases. In fact, otherwise, the orbits of the vertices under $\Gamma$, which is countably many, would cover $\ell_0$, which is uncountably many. Contradiction.

\noindent We have shown that up to conjugation by elements of $\Gamma$, $\gamma_0$ is
\begin{itemize}
	\item either a rotation of angle $\pi$ about an axis of symmetry of a face of $\vP$;
	\item or a rotation of angle $\pi$ about an axis of symmetry of the polygon determined by $\LatL(P_0)$, where $P_0$ is the middle point of an edge of $\vP$;
	\item or a rotation about an edge of $\vP$, which permutes $\LatL(P_0)$ for any $P_0$ lying in the interior of that edge.
\end{itemize}
	Hence there are only finitely many options for $\gamma_0$ and we conclude the proof.
\end{proof}

\bibliographystyle{acm}
	
\bibliography{mathbib}

\begin{thebibliography}{10}

\bibitem{BCCFL18}
{\sc Balkanova, O., Chatzakos, D., Cherubini, G., Frolenkov, D., and Laaksonen,
  N.}
\newblock Prime geodesic theorem in the 3-dimensional hyperbolic space.
\newblock arXiv: 1712.00880v2, August 2018.

\bibitem{BF18}
{\sc Balkanova, O., and Frolenkov, D.}
\newblock Prime geodesic theorem for the {Picard} manifold.
\newblock arXiv: 1804.00275v2, August 2018.

\bibitem{BF20}
{\sc Balkanova, O., and Frolenkov, D.}
\newblock The second moment of the second moment of symmetric square
  {$L$}-functions over gaussian integers.
\newblock arXiv: 2008.13399, 2020.

\bibitem{BBCL20}
{\sc Balog, A., Bir\'o, A., Cherubini, G., and Laaksonen, N.}
\newblock {Bykovskii}-type theorem for the {Picard} manifold.
\newblock {\em International Mathematical Research Notes 00}, 0 (2020), 1--29.

\bibitem{BB11}
{\sc Blomer, V., and Brumley, F.}
\newblock On the {R}amanujan conjecture over number fields.
\newblock {\em Annals of Mathematics\/} (2011), 581--605.

\bibitem{By94}
{\sc Bykovskii, V.~A.}
\newblock Density theorems and the mean value of arithmetic functions in short
  intervals {(Russian)}.
\newblock {\em Zap. Nauchn. Semin. POMI / translation in J. Math. Sci. (N.Y.)
  212 / 83}, / 6 (1994 / 1997), 56--70 / 720--730.

\bibitem{Cai02}
{\sc Cai, Y.}
\newblock Prime geodesic theorem.
\newblock {\em Journal de Th\'eorie des Nombres de Bordeaux 14}, 1 (2002),
  59--72.

\bibitem{CI00}
{\sc Conrey, J.~B., and Iwaniec, H.}
\newblock The cubic moment of central values of automorphic {$L$}-functions.
\newblock {\em Annals of Mathematics 151}, 3 (2000), 1175--1216.

\bibitem{EGM98}
{\sc Elstrodt, J., Grunewald, F., and Mennicke, J.}
\newblock {\em Groups Acting on Hyperbolic Space}.
\newblock Springer Monographs in Mathematics. Springer-Verlag, 1998.

\bibitem{GJ79}
{\sc Gelbart, S.~S., and Jacquet, H.}
\newblock Forms of {GL}(2) from the analytic point of view.
\newblock In {\em Proceedings of Symposia in Pure Mathematics\/} (1979),
  vol.~33, pp.~213--251.

\bibitem{Iw84}
{\sc Iwaniec, H.}
\newblock Prime geodesic theorem.
\newblock {\em Journal {f\"ur} die reine und angewandte Mathematik 1984}, 349
  (1984), 136--159.

\bibitem{JZ87}
{\sc Jacquet, H., and Zagier, D.}
\newblock Eisenstein series and the {Selberg} trace formula. {\Rmnum{2}}.
\newblock {\em Transactions of the American Mathematical Society 300}, 1
  (1987), 1--48.

\bibitem{Jo90}
{\sc Joyner, D.}
\newblock On the {Kuznetsov-Bruggeman} formula for a {Hilbert} modular surface
  having one cusp.
\newblock {\em Mathematische Zeitschrift 203\/} (1990), 59--104.

\bibitem{Ka66}
{\sc Kac, M.}
\newblock Can one hear the shape of a drum?
\newblock {\em Amer. Math. Monthly 73\/} (1966), 1--23.

\bibitem{KS02}
{\sc Kim, H.~H., and Sarnak, P.}
\newblock Refined estimates towards the {R}amanujan and {S}elberg conjectures
  ({A}ppendix 2 to {H.Kim, F}unctoriality for the exterior square of {$GL_4$}
  and the symmetric fourth of {$GL_2$}).
\newblock {\em Journal of the American Mathematical Society 16}, 1 (2002),
  139--183.

\bibitem{Koy01}
{\sc Koyama, S.}
\newblock Prime geodesic theorem for the {Picard} manifold under the mean
  {Lindel\"of} hypothesis.
\newblock {\em Forum Mathematicum 13}, 6 (2001), 781--793.

\bibitem{LS95}
{\sc Luo, W., and Sarnak, P.}
\newblock Quantum ergodicity of eigenfunctions on {$\mathrm{PSL}_2(\mathbb{Z})
  \backslash \mathbb{H}^2$}.
\newblock {\em Publications {math\'ematiques} de {l'IH\'ES} 81\/} (1995),
  207--237.

\bibitem{Martin}
{\sc Martin, K.}
\newblock Refined dimensions of cusp forms, and equidistribution and bias of
  signs.
\newblock {\em Journal of Number Theory 188\/} (2018), 1--17.

\bibitem{Miy06}
{\sc Miyake, T.}
\newblock {\em Modular Forms}, 2nd~ed.
\newblock Springer Monographs in Mathematics. Springer-Verlag, 2006.

\bibitem{Mue07}
{\sc M\"uller, W.}
\newblock Weyl's law in the theory of automorphic forms.
\newblock In {\em Groups and Analysis: The Legacy of Hermann Weyl}, T.~Katrin,
  Ed., vol.~354 of {\em London Mathematical Society Lecture Note Series}.
  Cambridge University Press, 2008.
\newblock arxiv: 0710.2319.

\bibitem{Na12}
{\sc Nakasuji, M.}
\newblock Generalized {R}amanujan conjecture over general imaginary quadratic
  fields.
\newblock {\em Forum Math. 24}, 1 (2012), 85--98.

\bibitem{Ne20}
{\sc Nelson, P.}
\newblock Eisenstein series and the cubic moment for $\mathrm{PGL}_2$.
\newblock arXiv: 1911.06310, 2020.

\bibitem{NPS13}
{\sc Nelson, P., Pitale, A., and Saha, A.}
\newblock Bounds for {Rankin-Selberg} integrals and quantum unique ergodicity
  for powerful levels.
\newblock {\em Journal of the American Mathematical Society 27}, 1 (2013),
  147--191.

\bibitem{Sar83}
{\sc Sarnak, P.}
\newblock The arithmetic and geometry of some hyperbolic three manifolds.
\newblock {\em Acta Math. 151\/} (1983), 253--295.

\bibitem{Sie56}
{\sc Siegel, C.}
\newblock Die {Funktionalgleichung einiger Dirichletscher Reihen}.
\newblock {\em Mathematische Zeitschrift 63\/} (1956), 363--373.

\bibitem{Sie61}
{\sc Siegel, C.}
\newblock Lectures on advanced analytic number theory.
\newblock {\em Tata Institute Bombay\/} (1961 (reissued 1965)).

\bibitem{SY13}
{\sc Soundararajan, K., and Young, M.~P.}
\newblock The prime geodesic theorem.
\newblock {\em Journal f\"{u}r die reine und angewandte Mathematik 2013}, 676
  (2013), 105--120.

\bibitem{SugiyamaTsuzuki}
{\sc Sugiyama, S., and Tsuzuki, M.}
\newblock An explicit trace formula of {J}acquet--{Z}agier type for {H}ilbert
  modular forms.
\newblock {\em Journal of Functional Analysis 275}, 11 (2018), 2978--3064.

\bibitem{Sz83}
{\sc Szmidt, J.}
\newblock The {Selberg} trace formula for the {Picard} group $\mathrm{SL}(2,
  \mathbb{Z}[i])$.
\newblock {\em Acta Arithmetica 42}, 4 (1983), 391--424.

\bibitem{Vi80_Ann}
{\sc Vign\'{e}ras, M.-F.}
\newblock Vari\'et\'e riemanniennes isospectrales et non isom\'etriques.
\newblock {\em Annals of Mathematics 112\/} (1980), 21--32.

\bibitem{Wu9}
{\sc Wu, H.}
\newblock Deducing {Selberg} trace formula via {Rankin-Selberg} method for
  {$\mathrm{GL}_2$}.
\newblock {\em Transactions of the American Mathematical Society 372}, 12
  (2019), 8507--8551.

\bibitem{Za76}
{\sc Zagier, D.}
\newblock Modular forms whose {Fourier} coefficients involve zeta-functions of
  quadratic fields.
\newblock In {\em Modular functions of one variable, {\Rmnum{6}}, Proc. 2nd
  Internat. Conf. (Bonn 1976)\/} (1977), vol.~627 of {\em Lecture Notes in
  Mathematics}, Springer, Berlin.

\bibitem{Za81}
{\sc Zagier, D.}
\newblock Eisenstein series and the {Selberg} trace formula. {\Rmnum{1}}.
\newblock In {\em Automorphic Forms, Representation Theory and Arithmetic},
  no.~10 in Tata Inst. Fund. Res. Stud. Math. Springer-Verlag, 1981,
  pp.~303--355.

\end{thebibliography}
	
\address{\quad\\ Giacomo Cherubini Charles University \\ Faculty of Mathematics and Physics\\ Department of Algebra\\ Sokolov\-sk\' a 83\\ 18600 Praha~8\\ Czech Republic\\ cherubini@karlin.mff.cuni.cz}
	
\address{\quad \\ Han WU \\ School of Mathematical Sciences \\ Queen Mary University of London \\ Mile End Road \\ London, E1 4NS \\ United Kingdom \\ wuhan1121@yahoo.com}

\address{\quad \\ Gergely Z\'abr\'adi \\ E\"otv\"os Lor\'and University of Budapest (\& MTA R\'enyi Int\'ezet Lend\"ulet Automorphic Research Group) \\ Institute of Mathematics \\ P\'azm\'any P.s. 1/C \\ 1117 Budapest \\ Hungary}
	
\end{document}